\documentclass[a4paper,reqno]{amsart}

\usepackage{lmodern}
\usepackage{slantsc}
\newcommand{\slkz}{{\normalfont\textsl{\textsc{kz}}}}
\newcommand{\kz}{{\normalfont\textsc{kz}}}
\newcommand{\mathkz}{{\normalfont\kz}}

\usepackage[all,2cell]{xy}
\usepackage{mathrsfs,amssymb,stmaryrd,bbm,mathabx,mathtools}
\xyoption{v2}
\UseAllTwocells
\usepackage{amsmath}
\numberwithin{equation}{section}

\usepackage{color}
\definecolor{darkgreen}{rgb}{0,0.45,0}
\usepackage[colorlinks,citecolor=darkgreen,final,backref=page,hyperindex]{hyperref}
\usepackage{autonum}
\usepackage[textsize=scriptsize,color=white]{todonotes}
\usepackage{gitinfo}
\usepackage[inline]{enumitem}
\usepackage[pagewise,modulo]{lineno}
\usepackage{graphicx}

\theoremstyle{plain}
\newtheorem{thm}[equation]{Theorem}
\newtheorem{cor}[equation]{Corollary}
\newtheorem{lemma}[equation]{Lemma}
\newtheorem{prop}[equation]{Proposition}

\theoremstyle{remark}

\newtheorem{rmk}[equation]{Remark}

\newtheorem{ex}[equation]{Example}
\newtheorem{notation}[equation]{Notation}

\theoremstyle{definition}

\newtheorem{df}[equation]{Definition}

\newcommand{\Cat}{\ensuremath{\mathbf{Cat}}}
\newcommand{\C}{\ensuremath{\mathcal{C}}}
\newcommand{\id}{\ensuremath{\mathrm{id}}}
\DeclareMathOperator{\dom}{dom}
\DeclareMathOperator{\cod}{cod}

\newcommand{\two}{{\mathbf{2}}}

\newcommand{\Sq}{\mathbb{S}\mathrm{q}}
\newcommand{\R}{\ensuremath{\mathbb{\bar{R}_+}}}

\newcommand{\Ord}{\ensuremath{\mathbf{Ord}}}

\newcommand{\Top}{\ensuremath{\mathbf{Top}}}
\newcommand{\lan}[2]{\mathrm{lan}_{#1}{#2}}

\DeclareMathOperator{\lari}{Lari}
\DeclareMathOperator{\rali}{Rali}

\newcommand{\Tzero}{\ensuremath{{\normalfont\textsc{t}}_0}}

\begin{document}
\title{Lax orthogonal factorisations in ordered structures}
\author[M~M~Clementino]{Maria Manuel Clementino}
\address{CMUC, Department of Mathematics\\ Univ.~Coimbra\\3001-501 Coimbra\\ Portugal}
\thanks{Research partially supported by Centro de Matem\'{a}tica da Universidade de
  Coimbra -- UID/MAT/00324/2013, funded by the Portuguese Government
  through FCT/MCTES and co-funded by the European Regional Development Fund
  through the Partnership Agreement PT2020.}
\email{mmc@mat.uc.pt}
\author[I~L\'{o}pez~Franco]{Ignacio L\'{o}pez Franco}
\address{
Department of Mathematics and Applications\\
  {CURE}, Universidad de la Rep\'ublica\\
  Tacuaremb\'o s/n, Maldonado, Uruguay
}
\thanks{The second author was supported by a Research Fellowship of
  Gonville and Caius College, Cambridge, the Department of Mathematics
  and Mathematical Statistics of the University of Cambridge, \textsc{PEDECIBA}, and
  \textsc{SNI--ANII}}
\email{ilopez@cure.edu.uy}
\date\today
\subjclass[2010]{Primary 18D05, 18A32. Secondary 55U35.}
\keywords{
  lax orthogonal factorization system,
  lax idempotent monad,
  order-enriched category,
  weak factorization system,
  reflective factorization system,
  continuous lattice.
}
\begin{abstract}
  We give an account of lax orthogonal factorisation systems on
  order-enriched categories. Among them, we define and characterise the
  \kz-reflective ones, in a way that mirrors the characterisation of reflective
  orthogonal factorisation systems. We use simple monads to construct lax
  orthogonal factorisation systems, such as one on the category of $T_0$
  topological spaces closely related to continuous lattices.
\end{abstract}
\maketitle


\section{Introduction}
\label{sec:introduction}

Weak factorisation systems (\textsc{wfs}s) have been a feature of Homotopy
Theory even before Quillen's definition of model categories and the recognition
of their importance. \textsc{Wfs}s, whose definition can be found in \S
\ref{sec:algebr-weak-fact}, can be described as a pair of classes of morphisms
$(\mathcal{L},\mathcal{R})$ that satisfy three properties. First, each morphism
of the category must be a composition of a morphism from $\mathcal{L}$ followed
by one of $\mathcal{R}$ (in a not necessarily unique way). Secondly, each
$r\in \mathcal{R}$ must have the \emph{right lifting property} with respect to
each $\ell\in\mathcal{L}$; in other words, each commutative square, as
displayed, has a (not necessarily unique) diagonal filler.
\begin{equation}
  \label{eq:145}
  \diagram
  \cdot\ar[r]\ar[d]_\ell&\cdot\ar[d]^r\\
  \cdot\ar[r]\ar@{..>}[ur]&\cdot
  \enddiagram
\end{equation}
Lastly, $(\mathcal{L},\mathcal{R})$ is, in a precise way, maximal. Each one of
Quillen's model categories comes equipped with two \textsc{wfs} (by
definition).


Orthogonal factorisations systems (\textsc{ofs}) arose at
the same time as \textsc{wfs}s and can be described as \textsc{wfs}s in
which the diagonal filler~\eqref{eq:145} not only exists but it is unique. This
makes the factorisation of a morphism $f$ as $f=r\cdot\ell$, with
$\ell\in\mathcal{L}$ and ${r}\in\mathcal{R}$, unique up to unique
isomorphism. Two typical examples of \textsc{ofs}s are the factorisation of a
function as a surjection followed by an injection, and of a continuous map
between topological spaces as a surjection followed by an embedding (ie an
homeomorphism onto its image).

When the ambient category has a terminal object, denoted by $1$, there is a
case of~\eqref{eq:145} of special interest, namely:
\begin{equation}
  \label{eq:147}
  \diagram
  \cdot\ar[r]\ar[d]_\ell&A\ar[d]\\
  \cdot\ar[r]\ar@{..>}[ur]&1
  \enddiagram
\end{equation}
If the unique morphism $A\to 1$ has the right (unique) lifting property with
respect to $\ell$, one says that $A$ is injective with respect (resp.,
orthogonal to) $\ell$. Clearly each \textsc{ofs} $(\mathcal{L},\mathcal{R})$
gives rise to a class of objects that are orthogonal to each member of
$\mathcal{L}$: those objects $A$ such that $A\to 1$ belongs to
$\mathcal{R}$. The extent to which $(\mathcal{L},\mathcal{R})$ is determined by
this class of objects is the subject of study of~\cite{MR779198}. The
\textsc{ofs}s so determined are called \emph{reflective}.

In addition to their widespread use in Homological Algebra,
injective objects play a role in many other areas of Mathematics. For example,
in the category of metric spaces and non-expansive maps, \emph{hyperconvex}
spaces are the objects injective with respect to the family of isometries
(see~\cite{MR0084762} and~\cite{MR0182949}).

There are examples, as those introduced by D.~Scott~\cite{MR0404073}, of
squares~\eqref{eq:147} where the diagonal filler is not unique but there exists
a smallest one (with respect to an ordering between morphisms). The main example
from~\cite{MR0404073} consists of those topological spaces that arise from
endowing \emph{continuous lattices} with the Scott topology. These spaces are
characterised by their injectivity with respect to topological embeddings. In fact, if
$\ell$ is a topological embedding and $A$ is a continuous lattice
in~\eqref{eq:147}, there is a diagonal filler that is the smallest with respect
to the (opposite of) the pointwise specialisation order (see \S
\ref{sec:filter-monads} for more details).

Another example comes from complete lattices, which can be characterised as
those posets that are injective with respect to embeddings of posets. As in the
previous example, in the situation~\eqref{eq:147} where $A$ is a complete
lattice and $\ell$ is a poset embedding, there exists a smallest diagonal
filler.

Motivated by the above examples, one can generalise the existence of a smallest
diagonal filler in the situation~\eqref{eq:147} to the situation of a
commutative square~\eqref{eq:145}. By doing so, one arrives to the notion of lax
orthogonal factorisation system.

The present paper gives an account, in the context of order-enriched categories,
of \emph{lax orthogonal factorisation systems} (\textsc{lofs}), a notion that
sits between \textsc{ofs}s and \textsc{wfs}s.
\begin{center}
  \begin{tabular}{ccccc}
    orthogonal &
    &
      lax orthogonal&
    &
      weak
    \\
    factorisation&
                   \scalebox{1.5}{$\subset$}&
                                                                     factorisation&
                                                                                    \scalebox{1.5}{$\subset$}&
                                                                                                                                      factorisation\\
    system&
    &
      system&
    &
      system
  \end{tabular}
\end{center}
\textsc{Lofs}s were introduced and studied in the context of 2-categories by
the authors in~\cite{clementino15:_lax}. We cover here some of the same material
in the much simpler framework of order-enriched categories and some completely
new results on reflective \textsc{lofs}s, as well as new examples (see below).

In a \textsc{lofs}, the existence of a diagonal filler~\eqref{eq:145} is
replaced by the existence of a smallest diagonal filler. More precisely, there
is a diagonal filler $d$ with the property that $d\leq d'$ for any other
diagonal filler $d'$.
\begin{equation}
  \label{eq:146}
  \xymatrixcolsep{1.7cm}
  \diagram
  \cdot\ar[r]\ar[d]_\ell&\cdot\ar[d]^r\\
  \cdot\ar[r]\ar@/^8pt/[ur]^(.4)d\ar@/_8pt/[ur]_(.6){d'}\urtwocell<\omit>{'\leq}&\cdot
  \enddiagram
\end{equation}
Since morphisms between two objects in an order-enriched category form a poset,
the above property uniquely defines the smallest diagonal filler. There are,
however, advantages in providing these diagonals by means of an algebraic
structure, instead of postulating the existence of a smallest diagonal
filler. This algebraic structure is provided by the \emph{algebraic weak
  factorisation systems} (\textsc{awfs}s), introduced with a different name
in~\cite{MR2283020} and slightly modified in~\cite{MR2506256}; we use the
definition given in the latter.

An \textsc{awfs} on an order-enriched category \C\ consists of a locally monotone
comonad $\mathsf{L}$ and a locally monotone monad $\mathsf{R}$ on $\C^\two$
interrelated by
axioms, and that define a locally monotone functorial factorisation
$f=Rf\cdot Lf$. Inspired by the observation of~\cite{MR2283020} that
\textsc{ofs}s correspond to \textsc{awfs}s whose monad and comonad are
idempotent, we defined in~\cite{clementino15:_lax} \textsc{lofs}s as \textsc{awfs}s whose monad and comonad
are lax idempotent, or Kock-Z{\"o}berlein. We reprise this definition in the
context of order-enriched categories, which enables some simplifications.

A fundamental example of \textsc{lofs} on the order-enriched category of posets
factors each morphism as a left adjoint right inverse (or \textsc{lari})
followed by a split opfibration. This factorisation can be
constructed on any order-enriched category with sufficient (finite) limits, and
plays a similar role for \textsc{lofs}s as the factorisation
isomorphism--morphism (that factors $f$ as $1_{\dom(f)}$ followed by $f$) plays
for \textsc{ofs}s (\S \ref{sec:laris-awfss}).

We introduce \slkz-\emph{reflective} \textsc{lofs}s as those \textsc{lofs}s
$(\mathsf{L},\mathsf{R})$ that are determined by the restriction of the monad
$\mathsf{R}$ on $\C^\two$ to $\C$ (here $\C$ is viewed as the full subcategory
of $\C^\two$ with objects of the form $A\to 1$). We characterise \kz-reflective
\textsc{lofs}s in a way that mirrors the characterisation in~\cite{MR779198} of
reflective \textsc{ofs}s $(\mathcal{L},\mathcal{R})$ as those with the following
property: if $g\cdot f$ and $g$ belong to $\mathcal{L}$, then so does $g$ (\S \ref{sec:reflective-lofss}). For
example, the \textsc{lofs} of \textsc{lari}--split opfibration mentioned above
will be reflective with our definition.

Another contribution of~\cite{MR779198} was the construction of reflective
\textsc{ofs}s from the so-called simple reflections. The morphisms inverted by
them always form a left class of an \textsc{ofs}. We introduce \emph{simple}
monads in the order-enriched context, as those satisfying a certain property that
allows us to build \textsc{lofs}s. After providing sufficient conditions for a
lax idempotent monad to be simple (\S \ref{sec:simple-monads-1}), we recover the example of topological spaces
discussed above in this introduction as a consequence of the simplicity of a
certain monad: the \emph{filter monad}, which associates to each topological
space the space of filters of its open subsets endowed with a natural
topology (\S \ref{sec:filter-monads}). The algebras for the filter monad are precisely the continuous
lattices (with the Scott topology). The induced \textsc{lofs} on ($T_0$)
topological spaces has an associated \textsc{wfs} that was considered
in~\cite{MR2927175}.
We also provide easy-to-verify conditions guaranteeing that a submonad of a
simple lax idempotent monad enjoys these same properties (\S\ref{sec:subm-simple-monads}). When applied to the
filter monad we obtain \textsc{lofs}s closely related to continuous Scott
domains, stably compact spaces and sober spaces.

Another example that we obtain from a simple monad is a \textsc{lofs} on the
order-enriched category of (skeletal) generalised metric spaces \S \ref{sec:metric-spaces}. The restriction
of this \textsc{lofs} to the category of metric spaces yields an \textsc{ofs}
whose left class of morphisms are the dense inclusions. Further examples are
explored in~\cite{1701.05510} in a very general framework that covers, for
example, R.~Lowen's \emph{approach spaces} as well as the examples mentioned
above.

An appendix \S \ref{sec:accessible-awfss} discusses part of the theory of \textsc{lofs}s that can be developed
in the setting of locally presentable categories, where, under mild hypotheses,
there is a reflection between the category of accessible lax idempotent monads
and the category of accessible \textsc{lofs}s.
The appendix demands more knowledge of some parts of Category Theory.

\section{Order-enriched categories and lax idempotent monads}
\label{sec:ord-enrich-categ}

By an \emph{ordered set} we shall mean what is usually called a \emph{poset},
that is, a pair $(X,\leq)$ where $X$ is a set and $\leq$ is a
relation on $X$ that is reflexive, transitive and antisymmetric. Ordered sets
can be identified with small categories with at most one morphism between any
two objects and whose isomorphisms are identity morphisms.

We denote by \Ord\ the category of ordered sets and monotone maps
(functions that
preserve $\leq$). This is a cartesian closed category, with exponential $Y^X$
defined as the set of all order-morphisms $X\to Y$, and endowed with the
pointwise order.

A category enriched in $\Ord$, or $\Ord$-category, is a
locally small category $\C$ whose hom-sets are equipped with an order
structure, and whose composition preserves the inequality: if $g\leq g'$ then
$h\cdot g\leq h\cdot g'$ and $g\cdot f\leq g'\cdot f$, whenever these
compositions are defined. In other words, the composition functions
\begin{equation}
  \label{eq:1}
  \C(Y,Z)\times\C(X,Y)\longrightarrow\C(X,Z)
\end{equation}
are monotone maps.

The category \Ord\ of ordered sets can be
regarded as a full subcategory of the category of small categories \Cat\ by
regarding ordered sets as small categories, as mentioned above. This
means that \Ord-categories can be regarded as 2-categories, but we do not go to
that level of generality.

A locally monotone functor $F\colon\C\to\mathcal{D}$, or \emph{\Ord-functor},
between \Ord-categories is an ordinary functor between the
underlying ordinary categories that is moreover monotone on homs; ie that each
$\C(X,Y)\to\mathcal{D}(FX,FY)$ is a monotone map.

The category of \Ord-categories and \Ord-functors will be denoted by
$\Ord\text-\Cat$. It is a cartesian closed category.

\begin{ex}
  \label{ex:1}
  The category \Ord\ has a canonical structure of an \Ord-category
  $\underline{\Ord}$ whose ordered sets are $\underline{\Ord}(X,Y)=Y^X$.
  Many other categories constructed from \Ord\ are \Ord-enriched, such as the
  categories of join-semilattices, complete lattices, distributive lattices, and
  Heyting algebras.
\end{ex}

\begin{ex}
  \label{ex:2}
  If $X$ is a topological space, define a preorder on $X$ by $x\leq y$ if all the
  neighbourhoods of $y$ are also neighbourhoods of $x$, or, equivalently, denoting by $\mathcal{O}X$ the topology of $X$, $x\in U$ whenever $y\in U$ for every $U\in\mathcal{O}X$; in other words, $x\leq
  y$ if $y\in\overline{\{x\}}$. The opposite of this order is usually called the
  \emph{specialisation order} and was introduced by D.~Scott
  in~\cite{MR0404073}. The preorder $(X,\leq)$ is an ordered set precisely when $X$
  is a \Tzero\ space.

  Any continuous function $f\colon X\to Y$ between topological spaces preserves
  the order $\leq$. The category $\Top_0$ of \Tzero\ topological spaces and continuous maps can be endowed with
  an \Ord-category structure if we define, for any pair $f,g\colon X\to Y$ of
  continuous functions, $f\leq g$ if $f(x)\leq g(x)$ for all $x\in X$.
\end{ex}

\subsection{Full morphisms and locally full functors}
\label{sec:adjunct-betw-ord}
\begin{df}
  \label{df:15}
  \begin{enumerate}
  \item \label{item:15}
    A monotone map $f$ between ordered sets is \emph{full} if
    it reflects inequalities; ie $f(x)\leq f(y)$ implies $x\leq y$.
  \item \label{item:16}
    A locally monotone functor $F\colon\mathcal{A}\to\mathcal{B}$ between
    \Ord-categories is \emph{locally full} if each monotone map
    \begin{equation}
      \label{eq:98}
      F_{A,B}\colon\mathcal{A}(A,B)\longrightarrow\mathcal{B}(FA,FB)
    \end{equation}
    is full.
  \item A morphism $g\colon X\to Y$ in an \Ord-category \C\ is
    \emph{full} if for each $Z\in\C$ the monotone map
    \begin{equation}
      \label{eq:99}
      \C(Z,g)\colon \C(Z,X)\longrightarrow\C(Z,Y)
    \end{equation}
    is full.
  \end{enumerate}
\end{df}

Full morphisms are necessarily monomorphisms; for if $f\colon X\to Y$ is a full monotone morphism of
ordered sets and $f(x)=f(y)$, then we have both $x\leq y$ and $y\leq x$, so
$x=y$.

\begin{lemma}
  \label{l:9}
  Suppose that $F\dashv U\colon \mathcal{B}\to\mathcal{A}$ is an adjunction of
  locally monotone functors between \Ord-categories, with unit $\eta\colon
  1_{\mathcal{A}}\Rightarrow UF$. Then $F$ is locally full if each component
  $\eta_{A}\colon A\to UFA$ is a full morphism.
\end{lemma}
\begin{proof}
  The naturality of $\eta$ is expressed by the commutativity of the following
  diagram.
  \begin{equation}
    \label{eq:100}
    \diagram
    \mathcal{A}(A,B)\ar[r]^-{F_{A,B}}\ar[drr]_-{\mathcal{A}(1,\eta_B)}&
    \mathcal{B}(FA,FB)\ar[r]^{U_{FA,FB}}&
    \mathcal{A}(UFA,UFB)\ar[d]^{\mathcal{A}(\eta_A,1)}\\
    &&\mathcal{A}(A,UFB)
    \enddiagram
  \end{equation}
  If $\eta_B$ is full, the diagonal morphism is full and therefore
  $F_{A,B}$ must be full too.
\end{proof}

\subsection{Order-enriched (co)limits}
\label{sec:order-enrich-limits}

\subsubsection*{Limits}
The category of ordered sets admits the construction of two-dimensional limits,
which will be fundamental for us. We denote by $\two$ the order with two
elements $0\leq 1$. If $X$ is an ordered set, then the exponential $X^\two$ is
\begin{equation}
  \label{eq:2}
  X^\two=\{(x,y)\in X\times X:x\leq y\}\subseteq X\times X
\end{equation}
with the order inherited from $X\times X$. We denote by $d_0$ and $d_1$ the two
projections from $X^\two$ onto $X$. Slightly more involved is the comma-object
of two order morphisms $f\colon X\to Z\leftarrow Y\colon g$
\begin{equation}
  \label{eq:3}
  f\downarrow g=\{(x,y)\in X\times Y:f(x)\leq g(y)\}\subseteq X\times Y
\end{equation}
that can equally well be constructed from $Z^\two$ by taking the limit of the
following diagram.
\begin{equation}
  \label{eq:4}
  X\xrightarrow{f}Z\xleftarrow{d_0}Z^\two\xrightarrow{d_1}Z\xleftarrow{g}Y
\end{equation}

The constructions of the previous paragraphs can be defined in any
\Ord-category \C. If $X\in\C$, then define $X^\two$ as an object equipped with
two morphisms $d_0\leq d_1\colon X^\two\to X$ that induce isomorphisms of orders
\begin{equation}
  \label{eq:5}
  \C(Z,X^\two)\cong \C(Z,X)^\two
\end{equation}
for all $Z\in\C$, in the sense that, for each pair of morphisms $f_0\leq
f_1\colon Z\to X$, there exists a unique morphism $h\colon Z\to X^\two$ such that
$f_0=d_0\cdot h$ and $f_1=d_1\cdot h$. Furthermore, if $k\colon Z\to X^\two$ is
another morphism, then the conjunction of $d_0\cdot h\leq d_0\cdot k$ and
$d_1\cdot h\leq d_1\cdot k$ implies $h\leq k$. This universal property guarantees that
$X^\two$ is unique up to canonical isomorphism.

Similarly, given morphisms $f\colon X\to Z\leftarrow Y\colon g$ in \C, one can
define a comma-object $f\downarrow g$ in \C\ as an object equipped with two
morphisms $d_0$ and $d_1$ as shown
\begin{equation}
  \label{eq:6}
  \xymatrixcolsep{.5cm}
  \xymatrixrowsep{.5cm}
  \diagram
  &f\downarrow g\ar[dl]_{d_0}\ar[dr]^{d_1}\ddtwocell<\omit>{'\leq}&\\
  X\ar[dr]_f&&Y\ar[dl]^g\\
  &Z&
  \enddiagram
\end{equation}
that induce an order-isomorphism
\begin{equation}
  \label{eq:7}
  \C(W,f\downarrow g)\cong \C(W,f)\downarrow\C(W,g)
\end{equation}
for all $W\in\C$. In other words, for each pair of morphisms $h_0\colon W\to X$ and
$h_1\colon W\to Y$ such that $f\cdot h_0\leq g\cdot h_1$, there exists a unique
$h\colon W\to f\downarrow g$ satisfying $d_0\cdot h=h_0$ and $d_1\cdot
h=h_1$. Furthermore, if $k\colon W\to f\downarrow g$ is another morphism, then
the conjunction of $d_0\cdot h\leq d_0\cdot k$ and $d_1\cdot h\leq d_1\cdot k$
implies $h\leq k$.

\subsubsection*{Colimits}
Let $\mathcal{D}$ be an ordinary category. If $D\colon \mathcal{D}\to \C$ is a
functor (ie a diagram in \C), we say that an object $C\in\C$ together with a
natural transformation $\alpha_X\colon D(X)\to C$ is a colimit of $D$ if
\begin{equation}
  \label{eq:156}
  \C(\alpha_X,C')\colon \C(C,C')\longrightarrow \C(D(X),C')
\end{equation}
is a limiting cone in the category \Ord, for all $C'\in\C$. This is the same as
saying that $(C,\alpha)$ is a limit of sets and the bijection
$\C(C,C')\cong\lim \C(D-,C')$ is an isomorphism of posets.

It is not hard to verify that \emph{filtered} colimits in \Ord\ can be
constructed in a completely analogous way to those in the category of
sets. Furthermore, it can easily be verified that filtered colimits
\emph{commute}, or \emph{distribute}, over finite enriched limits in \Ord, in
the sense that the \Ord-functor $\lim\colon [\mathcal{F},\Ord]\to\Ord$ preserves
filtered colimits if $\mathcal{F}$ is finite. For example, the functor
$(-)^\two\colon\Ord\to\Ord$ preserves filtered colimits, as do pullbacks, and
therefore comma-objects preserve colimits (since comma-objects can be
constructed from $(-)^\two$ and pullbacks). This phenomenon is part of the
general theory of locally finitely presentable enriched categories developed
in~\cite{Kelly:StructuresFiniteLimits}.

\subsection{Adjunctions, extensions and liftings}
\label{sec:adjunctions}

An adjunction in an \Ord-category \C\ consists of two morphisms $f\colon X\to Y$
and $g\colon Y\to X$ in opposite directions with inequalities
\begin{equation}
  \label{eq:8}
  1_X\leq g\cdot f\qquad \text{and}\qquad f\cdot g\leq 1_Y.
\end{equation}
Such an adjunction is usually written $f\dashv g$.

By the usual argument, adjoints are unique up to canonical isomorphism, which in our
case, by the antisymmetry of the ordering, means that adjoints are unique. For, if $f\dashv g$ and $f\dashv g'$,
then
\begin{equation}
  g= 1_X\cdot g\leq g'\cdot f\cdot g\leq g'\cdot 1_Y=g'
  \label{eq:10}
\end{equation}
and symmetrically, $g'\leq g$.

A notion related to adjunctions is that of a \emph{left extension}. If $j\colon
X\to Y$ and $f\colon X\to Z$ are morphisms in the \Ord-category \C, we say that
an inequality $f\leq\lan{j}{f}\cdot j$ exhibits $\lan{j}{f}\colon Y\to Z$ as a left
extension of $f$ by $j$ if,
for any other $g\colon Y\to Z$ that satisfies $f\leq g\cdot j$, the
  inequality $\lan{j}{f}\leq g$ holds.
\begin{equation}
  \label{eq:11}
  \diagram
  X\ar[rr]^j\ar[drr]_f&&Y\ar[d]^g\ar@<-20pt>@{}[d]|(.4)\leq\\
  &&Z
  \enddiagram
  \qquad=\quad
  \xymatrixcolsep{1.7cm}
  \diagram
  X\ar[r]^j\ar[dr]_f&Y\ar[d]|{\lan{j}{f}}\ar@/^25pt/[d]^{g}\ar@<17pt>@{}[d]|\leq
  \ar@<-20pt>@{}[d]|(.4)\leq\\
  &Z
  \enddiagram
\end{equation}
This universal property makes $\lan{j}{f}$ unique --~if it exists.

When $j$ has a right adjoint $j^*$, there always exists a left extension
$\lan{j}{f}$, for any $f$: the extension is given by $\lan{j}{f}=f\cdot j^*$.

The notion dual to that of a left extension is called \emph{left lifting}. If
$j\colon X\to Y$ and $f\colon Z\to Y$ are morphisms in \C, we say that an
inequality $f\leq j\cdot \prescript{j}{}{f}$ as depicted exhibits
$\prescript{j}{}{f}$ as a left lifting of $f$ through $j$ if, for any other
morphism $g$, the inequality $f\leq j\cdot g$ implies $\prescript{j}{}{f}\leq g$.
\begin{equation}
  \label{eq:46}
  \diagram
  X\ar@<20pt>@{}[d]|(.4)\geq\ar[rr]^j&&Y\\
  Z\ar[u]^g\ar[urr]_f&
  \enddiagram
  \qquad=\quad
  \xymatrixcolsep{1.7cm}
  \diagram
  X\ar[r]^j\ar@<-17pt>@{}[d]|\geq\ar@<20pt>@{}[d]|(.4)\geq&Y
  \\
  Z\ar[u]|{\prescript{j}{}{f}}\ar[ru]_f\ar@/^25pt/[u]^{g}&
  \enddiagram
\end{equation}

When $j$ has a left adjoint $j^\ell$, then $j^\ell\cdot f$ is a left lifting of
$f$ through $j$.

\subsection{Lax idempotent monads}
\label{sec:lax-idemp-monads}

Before recalling the notion of order-enriched mo\-nad, let us remind the reader
of the definition of a monad on a category. A monad on a category $\mathcal{A}$
is a triple $\mathsf{T}=(T,\eta,\mu)$ where $T$ is an endofunctor of
$\mathcal{A}$ and $\eta\colon 1_{\mathcal{A}}\Rightarrow T\Leftarrow T^2\colon
\mu$ are natural transformations that satisfy the associativity and unit axioms:
\begin{equation}
  \label{eq:13}
  \diagram
  T^3\ar[r]^-{T\mu}\ar[d]_{\mu T}&T^2\ar[d]^\mu\\
  T^2\ar[r]^-\mu&T
  \enddiagram
  \qquad
  \diagram
  T\ar[r]^-{T\eta}\ar[dr]_1&
  T^2\ar[d]^\mu&
  T\ar[l]_{\eta T}\ar[dl]^1\\
  &T&
  \enddiagram
\end{equation}
An algebra for the monad $\mathsf{T}$, or a \emph{$\mathsf{T}$-algebra}, is a
pair $(A,a)$ where $a\colon TA\to A$ is a morphism in $\mathcal{A}$ that
satisfies two axioms:
\begin{equation}
  \label{eq:14}
  \diagram
  T^2A\ar[r]^-{Ta}\ar[d]_{\mu_A}&
  TA\ar[d]^a\\
  TA\ar[r]^-a&
  A
  \enddiagram
  \qquad
  \diagram
  A\ar[r]^{\eta_A}\ar[dr]_1&
  TA\ar[d]^a\\
  &A
  \enddiagram
\end{equation}
A morphism of $\mathsf{T}$-algebras $(A,a)\to (B,b)$ is a morphism $f\colon A\to
B$ in $\mathcal{A}$ that satisfies $b\cdot Tf=f\cdot a$. Algebras and their
morphisms form a category $\mathsf{T}\text-\mathrm{Alg}$, that comes equipped
with a forgetful functor into $\mathcal{A}$.


Let \C\ be an \Ord-category. An order-enriched monad, or \emph{\Ord-monad}, on
\C\ consists of a monad $\mathsf{T}=(T,\eta,\mu)$ on the ordinary category \C\
with the additional requirement that $T$ be an $\Ord$-functor. When the context
is clear, we will refer to \Ord-monads simply as monads.

\begin{df}
  \label{df:2}
  A monad $\mathsf{T}=(T,\eta,\mu)$ on an \Ord-category \C\ is \emph{lax idempotent}, or
  \emph{Kock--Z\"oberlein}, if it satisfies any of the following equivalent conditions.
  \begin{enumerate}
  \item \label{item:3} $T\eta \cdot \mu\leq 1$.
  \item \label{item:4} $1\leq \eta T\cdot \mu$.
  \item \label{item:5} For any $\mathsf{T}$-algebra $a\colon TA\to A$, the
    inequality $1_{TA}\leq \eta_A\cdot a$ holds.
    \item A morphism $l:TA\to A$ defines a $\mathsf{T}$-algebra structure $(A,l)$ if and only if $l\dashv \eta_A$ with $l\cdot\eta_A=1_A$.
  \item \label{item:6} $T\eta\leq \eta T$.
  \item \label{item:7} For any pair of $\mathsf{T}$-algebras $(A,a)$ and
    $(B,b)$ and all morphisms $f\colon A\to B$ in \C, $b\cdot Tf\leq f\cdot a$ holds.
  \item \label{item:8} For any $\mathsf{T}$-algebra $(A,a)$ and any morphism
    $f\colon X\to A$ in \C, the equality $a\cdot Tf\cdot \eta_X=f$ exhibits
    $a\cdot Tf$ as a left extension of $f$ along $\eta_X\colon X\to TX$.
  \end{enumerate}
\end{df}
The equivalences of the above conditions can be found, in the more general case
of 2-categories, in~\cite{Kelly:Prop-like}.
Morphisms $f$ satisfying condition~(\ref{item:7}) are called \emph{lax
  morphisms} of $\mathsf{T}$-algebras, even for a monad $\mathsf{T}$ that is not
lax idempotent; so condition~(\ref{item:7}) says that $\mathsf{T}$ is lax
idempotent if any morphism in \C\ between $\mathsf{T}$-algebras is a lax
morphism of $\mathsf{T}$-algebras.

\begin{df}
  The notion of a lax idempotent comonad $\mathsf{G}=(G,\varepsilon,\delta)$ is
  a dual one: $\mathsf{G}$ is a lax idempotent comonad on \C\ if
  $(G^{\mathrm{op}},\varepsilon^{\mathrm{op}},\delta^{\mathrm{op}})$, the
  corresponding monad on $\C^\mathrm{op}$, is lax idempotent. We only translate
  explicitly condition~(\ref{item:8}) of Definition~\ref{df:2}: for any
  $\mathsf{G}$-coalgebra $a\colon A\to GA$ and any morphism $f\colon A\to X$ in
  \C, the equality $f= \varepsilon_ X\cdot Gf\cdot a$ exhibits $Gf\cdot a$ as a
  left lifting of $f$ through $\varepsilon_X$ (see \S \ref{sec:adjunctions} for
  the definition of left liftings).
  \begin{equation}
    \label{eq:47}
  \diagram
  GX\ar@<20pt>@{}[d]|(.4)\geq\ar[rr]^{\varepsilon_X}&&X\\
  A\ar[u]^{Gf\cdot a}\ar[urr]_f&
  \enddiagram
  \end{equation}
  \label{df:6}
\end{df}
\begin{ex}
  \label{ex:3}
  Given an ordered set $X$, denote by $P(X)$ the set of down-closed subsets
  of $X$, ie the set of those subsets $Y\subseteq X$ satisfying $(x\leq
  y)\wedge(y\in Y) \Rightarrow x\in Y$; the set $P(X)$ is canonically ordered by
  the inclusion of subsets of $X$. We denote by $\eta_X\colon X\to P(X)$ the
  monotone function
  \begin{equation}
    \label{eq:37}
    \eta_X\colon X\longrightarrow P(X)\qquad
    x\mapsto \mathnormal{\downarrow} x=\{y\in X:y\leq x\}.
  \end{equation}
  The assignment $X\mapsto P(X)$ can be extended to a functor whose value on a
  monotone function $f\colon X\to Y$ is
  \begin{equation}
    \label{eq:40}
    P(X)\xrightarrow{f_*}P(Y)\qquad
    f_*(Z)=\{y\in Y:(\exists x\in Z )(y\leq f(x))\}=
    \cup_{x\in Z}\mathnormal{\downarrow}f(x).
  \end{equation}
   Observe that $f_*$ always has a right adjoint $f^*\colon P(Y)\to P(X)$ given by
  \begin{equation}
    f^*(Z)=\{x\in X:\exists z\in Z \text{ such that } f(x)\leq z\}.
  \end{equation}
  Clearly, $f_*\leq g_*$ if $f\leq g$, so $P$ is an \Ord-functor.
  It is well-known that $X\mapsto P(X)$
  defines a monad on $\Ord$, where $P(X)$ is ordered by inclusion, with unit
  $\eta$ and multiplication $\mu$ given by
  \begin{equation}
    \label{eq:38}
    P^2(X)\longrightarrow P(X)\qquad
    \big(\mathcal{U}\subseteq P(X)\big)\mapsto \cup\{Y\in \mathcal{U}\}\subseteq X.
  \end{equation}
  This \Ord-monad on the \Ord-category $\underline{\Ord}$ is lax idempotent, since
  \begin{equation}
    \label{eq:39}
    P\eta_X(Z)=\cup_{x\in Z}\mathnormal{\downarrow}(\mathnormal{\downarrow}x)
    \subseteq\mathnormal\downarrow Z=\eta_{P(X)}(Z).
  \end{equation}
  The \Ord-category $\mathsf{P}\text-\mathrm{Alg}$ is the category of complete
  lattices (posets with arbitrary suprema or joins) with morphisms those
  monotone maps that preserve arbitrary suprema.
\end{ex}
\begin{ex}
  \label{ex:4}
  If $\Top_0$ is the category of \Tzero\ topological spaces and $\underline{\Top}_0$ is the
  associated \Ord-category, with ordering induced by the opposite of the
  specialisation order, as in Example~\ref{ex:2}, there is an endo-\Ord-functor
  $F\colon\underline{\Top}_0\to\underline{\Top}_0$ that sends $X$ to the set $F(X)$
  of filters of open sets of $X$, with topology generated by the subsets
  $U^{\hash}=\{\varphi\in F(X):U\in\varphi\}$, for $U\in\mathcal{O}X$. This is in fact the functor part of the lax idempotent \emph{filter monad} on $\underline{\Top}_0$ that will be studied in Section~\ref{sec:filter-monads}.
\end{ex}

There is a well-known result about algebras for lax idempotent monads on
\Ord-categories (see~\cite{Kock:KZmonads} and~\cite{MR1718976})
that can be summarised by saying that algebras are closed under
retracts. More precisely:
\begin{lemma}
  \label{l:14}
  If $\mathsf{T}=(T,\eta,\mu)$ is a lax idempotent monad on an \Ord-category,
  the following conditions on an object $A$ are equivalent.
  \begin{enumerate}
  \item \label{item:19}
    $A$ admits a (unique) $\mathsf{T}$-algebra structure (we simply say that $A$
    is a $\mathsf{T}$-algebra).
  \item \label{item:20}
    $\eta_A\colon A\to TA$ has a right inverse.
  \item \label{item:21} $A$ is a retract of $TA$.
  \item \label{item:22} $A$ is a retract of a $\mathsf{T}$-algebra.
  \end{enumerate}
\end{lemma}

Given two monads $\mathsf{S}=(S,\nu,\theta)$ and $\mathsf{T}=(T,\eta,\mu)$ on
the category \C\, we recall that a \emph{monad morphism}
$\varphi:\mathsf{S}\to\mathsf{T}$ is a natural transformation such that, for
every object $X$ of $\C$, the following diagrams commute.
\begin{equation}
  \xymatrixrowsep{.03cm}
  \diagram
  &TSX\ar[dr]^{T\varphi_X}&\\
  SSX\ar[ru]^-{\varphi_{SX}}\ar[dddd]_{\theta_X}\ar[dr]_{S\varphi_X}&&
  TTX\ar[dddd]^{\mu_X}
  \\
  &STX\ar[ur]_{\varphi_{TX}}&
  \\
  \\
  \\
  SX\ar[rr]^-{\varphi_X}&&TX
  \enddiagram
  \qquad
  \xymatrixrowsep{1cm}
  \diagram
  &X\ar[dl]_{\nu_X}\ar[dr]^{\eta_X}&\\
  SX\ar[rr]^-{\varphi_X}&&TX
  \enddiagram
\end{equation}
(There is a more general notion of morphism between monads on different
categories, which we will not need.)

\begin{lemma}
  \label{l:15}
  Let $\mathsf{T}$ and $\mathsf{S}$ be monads on an \Ord-category. Then there is
  at most one monad morphism $\mathsf{T}\to \mathsf{S}$ if $\mathsf{T}$ is lax
  idempotent.
\end{lemma}
\begin{proof}
  Suppose that $\varphi_X\colon TX\to SX$ are the components of a monad
  morphism. The morphism
  \begin{equation}
    \psi_X\colon TSX\xrightarrow{\varphi_{SX}}S^2X\xrightarrow{\mu^S_X}SX
  \end{equation}
  is a $\mathsf{T}$-algebra structure on $SX$, and therefore it is uniquely
  defined as the left adjoint to the unit $SX\to TSX$. Therefore,
  $\varphi_X=\psi_X\cdot T(\eta_X^S)$ is uniquely defined.
\end{proof}

\section{Orthogonal factorisations and simple reflections, revisited}
\label{sec:simple-reflections}
In this section we revisit some of the material of Cassidy--H\'{e}bert--Kelly
work on simple reflections~\cite{MR779198} from a slightly different
perspective, more amenable to generalisation.

Suppose that $T\colon\mathcal{A}\to \mathcal{A}$ is a reflection, with unit
$\eta_A\colon A\to TA$, on the category $\mathcal{A}$, which we assume to admit
pullbacks. The corresponding reflective subcategory will be denoted by
$\mathsf{T}\text-\mathrm{Alg}$, as it consists of the algebras for the idempotent monad
$\mathsf{T}$
associated to $T$, whose invertible multiplication we denote by $\mu\colon T^2\Rightarrow T$.

We say that a morphism $f$ in $\mathcal{A}$ is a
\emph{$T$-isomorphism}, or is \emph{$T$-invertible}, if $Tf$ is an isomorphism.

Each morphism $f\colon A\to B$ can be factorised through a pullback
square, as displayed.
\begin{equation}
  \label{eq:15}
  f=Rf\cdot Lf\qquad
  \diagram
  A\ar[rd]_{Lf}\ar@/^10pt/[drr]^{\eta_A}\ar@/_10pt/[ddr]_f&&\\
  &Kf\ar@{}[dr]|{{\mathrm{pb}}}\ar[r]^{q_f}\ar[d]|{Rf}&TA\ar[d]^{Tf}\\
  &B\ar[r]^-{\eta_B}&  TB
  \enddiagram
\end{equation}
\begin{rmk}
  \label{rmk:2}
  The factorisation $f=Rf\cdot Lf$ is \emph{functorial}, in the sense that, if
  $(h,k)\colon f \to g$ is a morphism in the arrow category $\mathcal{A}^\two$,
  then there is a morphism $K(h,k)\colon Kf\to Kg$
  \begin{equation}
    \label{eq:19}
    \diagram
    \cdot\ar[d]_f\ar[r]^h&\cdot\ar[d]^g\\
    \cdot\ar[r]^-k&\cdot
    \enddiagram
    \longmapsto
    \diagram
    \cdot\ar[r]^h\ar[d]_{Lf}&\cdot\ar[d]^{Lg}\\
    Kf\ar[r]^{K(h,k)}\ar[d]_{Rf}&Kg\ar[d]^{Rg}\\
    \cdot\ar[r]^k&\cdot
    \enddiagram
  \end{equation}
  yielding a functor $K\colon\mathcal{A}^\two\to\mathcal{A}$.
\end{rmk}
\begin{rmk}
  \label{rmk:3}
  The assignment that sends a morphism $f\mapsto Lf$ is part of an endofunctor
  on $\mathcal{A}^\two$, given on morphisms by
  \begin{equation}
    \label{eq:20}
    f\xrightarrow{(h,k)}g\quad\longmapsto \quad Lf\xrightarrow{(h,K(h,k))}Lg.
  \end{equation}
  Furthermore, there is a natural transformation $\Phi\colon L\Rightarrow 1$ with components
  \begin{equation}
    \label{eq:21}
    \Phi_f\quad = \quad
    \diagram
    \cdot\ar[d]_{Lf}\ar@{=}[r]&\cdot\ar[d]^f\\
    \cdot\ar[r]^{Rf}&\cdot
    \enddiagram
  \end{equation}
\end{rmk}
\begin{rmk}
  \label{rmk:4}
  The assignment $f\mapsto Rf$ underlies a monad on the arrow category
  $\mathcal{A}^\two$. Its unit and multiplication are given by
  \begin{equation}
    \label{eq:27}
    \Lambda_f\quad=\quad
    \diagram
    \cdot\ar[d]_f\ar[r]^{Lf}&\cdot\ar[d]^{Rf}\\
    \cdot\ar@{=}[r]&\cdot
    \enddiagram
    \qquad
    \Pi_f\quad=\quad
    \diagram
    \cdot\ar[d]_{R^2f}\ar[r]^{\pi_f}&\cdot\ar[d]^{Rf}\\
    \cdot\ar@{=}[r]&\cdot
    \enddiagram
  \end{equation}
  where the morphism $\pi_f\colon KRf\to Kf$ is the unique morphism into
  the pullback $Kf$ such that
  \begin{equation}
    q_f\cdot \pi_f=\mu_{\dom(f)}\cdot Tq_f\cdot q_{Rf} \quad\text{and}\quad
    Rf\cdot\pi_f=RRf.
    \label{eq:28}
\end{equation}
\end{rmk}
One of the contributions of~\cite{MR779198} is to introduce a property on the
reflection $T$ that guarantees that the factorisation $f=Rf\cdot Lf$ is an
orthogonal factorisation system (\textsc{ofs}): the property of being simple.

\begin{df}
  \label{df:1}
  The reflection $\mathsf{T}=(T,\eta)$ is \emph{simple} if $Lf$ is a $T$-isomorphism.
\end{df}

As pointed out in~\cite{MR779198}, if $\mathsf{T}$ is simple then the factorisation
$f=Rf\cdot Lf$ defines an orthogonal factorisation system, with left class of
morphisms that of $T$-isomorphisms. To say only a few words about this fact, any
morphism of the form $Tf$ is orthogonal to $T$-isomorphisms, and so $Rf$, as a
pullback of $Tf$, is also orthogonal to $T$-isomorphisms; together with the
simplicity hypothesis that $Lf$ be a $T$-isomorphism, we obtain an orthogonal
factorisation.

If we denote by $F^T\colon \mathcal{A}\to \mathsf{T}\text-\mathrm{Alg}$ the left adjoint of the
inclusion $\mathsf{T}\text-\mathrm{Alg}\subset\mathcal{A}$, then we can consider the full
subcategory $\mathsf{T}\text-\mathrm{Iso}\subset\mathcal{A}^\two$ whose objects are those
morphisms of $\mathcal{A}$ that are $T$-isomorphisms (equivalently, those morphisms $f$
such that $F^T(f)$ is an isomorphism) as a pullback.
\begin{equation}
  \label{eq:16}
  \diagram
  \mathsf{T}\text-\mathrm{Iso}\ar[r]\ar[d]\ar@{}[dr]|{\mathrm{pb}}&
  \mathrm{Iso}\ar[d]\\
  \mathcal{A}^\two\ar[r]^-{(F^T)^\two}&
  \mathsf{T}\text-\mathrm{Alg}^\two
  \enddiagram
\end{equation}

\begin{lemma}
  \label{l:1}
  The subcategory $\mathsf{T}\text-\mathrm{Iso}\hookrightarrow{}\mathcal{A}^\two$ is
  coreflective if and only if the reflection $\mathsf{T}$ is simple. In this
  case, the associated
  idempotent comonad is given by $f\mapsto Lf$ and has counit
  \begin{equation}
    \label{eq:17}
    \diagram
    \cdot\ar[d]_{Lf}\ar@{=}[r]&\cdot\ar[d]^f\\
    \cdot\ar[r]^-{Rf}&\cdot
    \enddiagram
  \end{equation}
\end{lemma}
\begin{proof}
  If $\mathsf{T}$ is simple, we know that the $T$-isomorphisms are the left class of an
  orthogonal factorisation system, and thus coreflective in
  $\mathcal{A}^\two$. To be more explicit, if $(\mathscr E,\mathscr M)$ is an
  orthogonal factorisation system in $\mathcal{A}$, and $f=m\cdot e$ with
  $e\in\mathscr E$ and $m\in\mathscr M$, then the morphism
  \begin{equation}
    \label{eq:25}
    \diagram
    \cdot\ar[d]_e\ar@{=}[r]&\cdot\ar[d]^f\\
    \cdot\ar[r]^-m&\cdot
    \enddiagram
  \end{equation}
  exhibits $e$ as a coreflection of $f$ into the full subcategory of
  $\mathcal{A}^\two$ defined by $\mathscr{E}$.

  Before moving to proving the converse, we make the observation that, for any
  category $\mathcal{B}$, the full subcategory
  $\mathrm{Iso}\subset\mathcal{B}^\two$ of isomorphisms is coreflective (as
  well as reflective) with coreflection given by $\Upsilon_f\colon
  I(f)=1_{\dom(f)}\to f$
  \begin{equation}
    \label{eq:29}
    \Psi_f
    \quad=\quad
    \diagram
    {\cdot}
    \ar@{=}[r]^-{}\ar[d]_{1_{\dom(f)}}
    &
    {\cdot}
    \ar[d]^{f}
    \\
    {\cdot}
    \ar[r]_-{f}
    &
    {\cdot}
    \enddiagram
  \end{equation}

  To prove the converse, suppose that the inclusion of $\mathsf{T}\text-\mathrm{Iso}$
  into $\mathcal{A}^\two$ is coreflective, with coreflection given by counits
  $\Psi_f\colon Gf \to f$ in $\mathcal{A}^\two$. Then the pullback
  diagram~\eqref{eq:16} can be rewritten in the following form, where the
  categories of coalgebras are those for the respective copointed endofunctors
  $\Psi\colon G\Rightarrow 1$ and $\Upsilon\colon I\Rightarrow 1$.
  \begin{equation}
    \label{eq:31}
    \diagram
    (G,\Psi)\text-\mathrm{Coalg}
    \ar[r]^-{}\ar[d]_{}\ar@{}[dr]|{\mathrm{pb}}
    &
    {(I,\Upsilon)\text-\mathrm{Coalg}}
    \ar[d]^{}
    \\
    {\mathcal{A}^\two}
    \ar[r]^-{(F^T)^\two}
    &
    {\mathsf{T}\text-\mathrm{Alg}^\two}
    \enddiagram
  \end{equation}
  It is well known that, in these circumstances, $(G,\Psi)$ is given by a
  pullback in the category of endofunctors of $\mathcal{A}^\two$
  \begin{equation}
    \label{eq:32}
    \diagram
    G\ar[d]_\Psi\ar[r]
    &
    (U^T)^\two I(F^T)^\two \ar[d]^{(U^T)^\two \Upsilon (F^T)^\two}
    \\
    1_{\mathcal{A}^\two}\ar[r]^-{\eta^\two}
    &
    T^\two
    \enddiagram
  \end{equation}
  If we apply the domain functor $\dom\colon\mathcal{A}^\two\to\mathcal{A}$ to
  this pullback, we obtain that $\dom(\Psi)$ can be taken to be the identity
  transformation, since $\dom(U^T\Upsilon_{F^T(f)})$ is an identity morphism for
  any $f$. If we apply the codomain functor $\cod$ instead, we obtain a pullback square
  \begin{equation}
    \label{eq:33}
    \diagram
    {\cod(Gf)}
    \ar[r]^-{}\ar[d]_{\cod\Psi_f}
    &
    {T\dom f}
    \ar[d]^{Tf}
    \\
    {\cod(f)}
    \ar[r]_-{\eta_{\cod(f)}}
    &
    {T\cod (f)}
    \enddiagram
  \end{equation}
  (we have used that $\cod U^T I(F^T(f))=U^T\cod
  (1_{\dom(F^T(f))})=T\dom(f)$). In other words, $\cod\Psi_f=Rf$ and
  $\cod(Gf)=Kf$ as defined in diagram~\eqref{eq:15}. From here it is
  straightforward to verify that $Gf=Lf$. Therefore $Lf\in \mathsf{T}\text-\mathrm{Iso}$,
  which says that $T$ is a simple reflection, concluding the proof.
\end{proof}

The lemma proved above gives a characterisation of simple reflections, so one
could define simple reflections as those reflections $\mathsf{T}$ on $\mathcal{A}$ such
that the full subcategory $\mathsf{T}\text-\mathrm{Iso}\subset\mathcal{A}^\two$ is
coreflective. The associated idempotent comonad on
$\mathcal{A}^\two$ is given by
$f\mapsto Lf$.

\section{Lax orthogonal factorisations}
\label{sec:lax-orth-fact}
We now proceed to study lax orthogonal factorisation systems on
\Ord-categories. Before that, we briefly recall basic facts on algebraic weak
factorisation systems.
\subsection{Weak factorisation systems}
\label{sec:algebr-weak-fact}
This short section recalls the definition of weak factorisation system, a notion
that appeared as part of Quillen's definition of model category~\cite{MR0223432}.

We say that a morphism \emph{$g$ has the right lifting property with respect to
another $f$}, and that \emph{$f$ has the left lifting property with respect to $g$}, if
every time we have a commutative square as shown, there exists (a not
necessarily unique) diagonal filler.
\begin{equation}
  \label{eq:12}
  \diagram
  \cdot\ar[r]\ar[d]_f&\cdot\ar[d]^g\\
  \cdot\ar@{..>}[ur]\ar[r]&\cdot
  \enddiagram
\end{equation}

A \emph{weak factorisation system} (\textsc{wfs}) in a category consists of two
families of morphisms $\mathcal{L}$ and $\mathcal{R}$ such that:
\begin{itemize}
\item $\mathcal{R}$ consists of those morphisms with the right lifting property
  with respect to each morphism of $\mathcal{L}$.
\item $\mathcal{L}$ consists of those morphisms with the left lifting property
  with respect to each morphism of $\mathcal{R}$.
\item Each morphism in the category is equal to the composition of one element
  of $\mathcal{L}$ followed by one of $\mathcal{R}$.
\end{itemize}

\subsection{Algebraic weak factorisation systems}
\label{sec:algebr-weak-fact-1}

Algebraic weak factorisation systems (\textsc{awfs}s) where first introduced by
M.~Grandis and W.~Tholen in~\cite{MR2283020}, with an extra distributivity condition
later added by R.~Garner in~\cite{MR2506256}. In this section we shall give the
definition of \textsc{awfs}s on order-enriched categories, which is the case we
will need, even though the definitions remain virtually unchanged.
\begin{df}
  \label{df:4}
  An \Ord-functorial factorisation on an \Ord-category \C\ consists of a factorisation
  \begin{equation}
    \label{eq:18}
    \dom\xRightarrow{\lambda}E\xRightarrow{\rho}\cod
  \end{equation}
  in the category of locally monotone functors $\C^\two\to \C$ of the natural
  transformation $\dom\Rightarrow\cod$ with component at $f\in \C^\two$ equal to
  $f\colon\dom(f)\to\cod(f)$. It is important that in this factorisation $E$
  should be a locally monotone functor.

  As in the case of functorial factorisations on ordinary categories, an
  \Ord-func\-to\-rial factorisation as the one described in the previous paragraph
  can be equivalently described as:
  \begin{itemize}
  \item A copointed endo-\Ord-functor $\Phi\colon L\Rightarrow 1_{\C^\two}$ on
    $\C^\two$ with $\dom(\Phi)=1$.
  \item A pointed endo-\Ord-functor $\Lambda\colon 1_{\C^\two}\Rightarrow R$ on
    $\C^\two$ with $\cod(\Lambda)=1$.
  \end{itemize}
  The three descriptions of an \Ord-functorial factorisation are related by:
  \begin{equation}
    \dom(\Lambda_f)=Lf=\lambda_f\qquad \cod(\Phi_f)=Rf=\rho_f.\label{eq:22}
  \end{equation}
\end{df}
\begin{df}
  \label{df:3}
  An \emph{algebraic weak factorisation system}, abbreviated {\textsc{awfs}}, on an
  \Ord-category \C\ consists of a pair $(\mathsf{L},\mathsf{R})$, where
  $\mathsf{L}=(L,\Phi,\Sigma)$ is an \Ord-comonad and
  $\mathsf{R}=(R,\Lambda,\Pi)$ is an \Ord-monad on $\C^\two$, such that $(L,\Phi)$ and
  $(R,\Lambda)$ represent the same \Ord-functorial factorisation on \C\ (ie, the
  equalities~\eqref{eq:22} hold), plus a distributivity condition that we
  proceed to explain.

  The unit axiom $\Pi\cdot (\Lambda R)=1$ of the monad $\mathsf{R}$ implies,
  since $\cod(\Lambda)=1$, that $\cod(\Pi)=1$; dually $\dom(\Sigma)=1$, so these
  transformations have components that look like:
  \begin{equation}
    \label{eq:23}
    \Sigma_f=
    \diagram
    \cdot\ar[d]_{Lf}\ar@{=}[r]&\cdot\ar[d]^{L^2f}\\
    \cdot\ar[r]^{\sigma_f}&\cdot
    \enddiagram
    \quad\text{and}\quad
    \Pi_f=
    \diagram
    \cdot\ar[d]_{R^2f}\ar[r]^{\pi_f}&\cdot\ar[d]^{Rf}\\
    \cdot\ar@{=}[r]&\cdot
    \enddiagram
  \end{equation}
  One can form a transformation
  \begin{equation}
    \label{eq:24}
    \Delta\colon LR\Longrightarrow RL\qquad
    \Delta_f=
    \diagram
    Kf\ar@{..>}[dr]^1\ar[d]_{LRf}\ar[r]^{\sigma_f}\ar[d]&KLf\ar[d]^{RLf}\\
    KRf\ar[r]^-{\pi_f}&Kf
    \enddiagram
  \end{equation}
  The distributivity axiom requires $\Delta$ to be a mixed distributive law between
  the comonad $\mathsf{L}$ and the monad $\mathsf{R}$; this amounts to the
  commutativity of the following diagrams.
  \begin{equation}
    \label{eq:30}
    \diagram
    LR^2\ar[r]^-{\Delta R}\ar[d]_{L\Pi}&
    RLR\ar[r]^-{R\Delta}&
    R^2L\ar[d]^{\Pi L}\\
    LR\ar[rr]^-{\Delta}&&RL
    \enddiagram
    \quad
    \diagram
    LR\ar[d]_{\Sigma R}\ar[rr]^{\Delta}&&
    RL\ar[d]^{R\Sigma}\\
    L^2R\ar[r]^-{L\Delta}&
    LRL\ar[r]^{\Delta L}&
    RL^2
    \enddiagram
  \end{equation}
  (The two axioms of a mixed distributive law that involve the unit of the monad
  and the counit of the comonad automatically hold.)
\end{df}

\begin{ex}
  \label{ex:5}
  Each \textsc{ofs} $(\mathscr{E},\mathscr{M})$ on \C\ gives
  rise (upon choosing an $(\mathscr{E},\mathscr{M})$-facto\-risat\-ion for each
  morphism) to an \textsc{awfs} $(\mathsf{L},\mathsf{R})$, where $\mathsf{L}$ is the
  idempotent comonad associated to the coreflective subcategory
  $\mathscr{E}\subset\C^\two$ and $\mathsf{R}$ is the idempotent monad
  associated to the reflective inclusion
  $\mathscr{M}\subset\C^\two$. Conversely, an \textsc{awfs}
  $(\mathsf{L},\mathsf{R})$ with both $\mathsf{L}$ and $\mathsf{R}$ idempotent
  induces an \textsc{ofs}. This was first shown in~\cite[Thm.~3.2]{MR2283020},
  and \cite[Prop.~3]{MR3393453} further shows that it suffices that either
  $\mathsf{L}$ or $\mathsf{R}$ be idempotent.
\end{ex}

If $(\mathsf{L},\mathsf{R})$ is an \textsc{awfs} on \C, an
$\mathsf{L}$-coalgebra structure on $f$ and an $\mathsf{R}$-algebra structure on
$g$ can be depicted by commutative squares
\begin{equation}
  \label{eq:9}
  \diagram
  \cdot\ar[d]_f\ar@{=}[r]^{\phantom{p}}&\cdot\ar[d]^{Lf}\\
  \cdot\ar[r]^-s&\cdot
  \enddiagram
  \qquad
  \diagram
  \cdot\ar[d]_{Rg}\ar[r]^-p&\cdot\ar[d]^g\\
  \cdot\ar@{=}[r]_{}&\cdot
  \enddiagram
\end{equation}
and the (co)algebra axioms can be written in the following way (where the
morphisms $\sigma_f$ and $\pi_g$ are those described in Definition~\ref{df:3}).
\begin{gather}
  Rf\cdot s=1\qquad K(1,s)\cdot s=\sigma_f\cdot s\\
  p\cdot Lg=1\qquad p\cdot K(p,1)= p\cdot\pi_g
\end{gather}
A morphism of $\mathsf{L}$-coalgebras $(f,s)\to (f',s')$ is a morphism
$(h,k)\colon f\to f'$ in $\C^\two$ that is compatible with the coalgebra
structures in the usual way:
\begin{equation}
  K(h,k)\cdot s=s'\cdot k.
\end{equation}
Similarly, a morphism of $\mathsf{R}$-algebras $(g,p)\to(g',p')$ is a morphism
$(u,v)\colon g\to g'$ such that
\begin{equation}
  p'\cdot K(u,v)=u\cdot p.
\end{equation}
With the obvious composition and identities we obtain categories
$\mathsf{L}\text-\mathrm{Coalg}$ and $\mathsf{R}\text-\mathrm{Alg}$, equipped
with forgetful functors into $\C^\two$. These are \Ord-categories by stipulating
that the ordering of morphisms of (co)algebras is inherited from the ordering of
morphisms in $\C^\two$; as a consequence, the forgetful functors from
$\mathsf{L}\text-\mathrm{Coalg}$ and $\mathsf{R}\text-\mathrm{Alg}$ to $\C^\two$
become $\Ord$-enriched.

\subsection{Underlying {\normalfont\textsc{wfs}s}}
\label{sec:underly-norm}

Each \textsc{awfs} $(\mathsf{L},\mathsf{R})$ (enriched or not) has an
\emph{underlying} \textsl{\textsc{wfs}} $(\mathcal{L},\mathcal{R})$. The class
$\mathcal{L}$ consists of all those morphisms that admit a structure of
coalgebra over the copointed endofunctor $(L,\Phi)$ that underlies $\mathsf{L}$;
similarly, $\mathcal{R}$ consists of all those morphisms that admit a structure
of an algebra over the pointed endofunctor $(R,\Lambda)$ that underlies
$\mathsf{R}$.

\subsection{{\normalfont \textsc{Lari}s} and {\normalfont\textsc{awfs}s}}
\label{sec:laris-awfss}
One of the most important examples of \textsc{awfs}s for us will be provided by
the so-called \textsc{lari}s.
\begin{df}
  \label{df:16}
  A \emph{left adjoint right inverse}, or \textsc{lari}, in an \Ord-category is
  a morphism $f$ that is part of an adjunction $f\dashv g$ with $1=g\cdot f$. In
  the same situation, we say that $g$ is a \emph{right adjoint left inverse}, or
  \textsc{rali}.

  Suppose given another adjunction $f'\dashv g'$ with $1=g'\cdot f'$, and
  morphisms $h$ and $k$ as in the displayed diagram.
  \begin{equation}
    \label{eq:79}
    \xymatrixcolsep{1.5cm}
    \diagram
    X\ar@<-5pt>[d]_f\ar@{}[d]|\dashv\ar@<5pt>@{<-}[d]^g
    \ar[r]^h
    &
    X'\ar@<-5pt>[d]_{f'}\ar@{}[d]|\dashv\ar@<5pt>@{<-}[d]^{g'}
    \\
    Y\ar[r]_k
    &
    Y'
    \enddiagram
  \end{equation}
  We say that $(h,k)$ is a \emph{morphism of} \textsl{\textsc{lari}s} $f\to f'$,
  and that $(h,k)$ is a \emph{morphism of} \textsl{\textsc{rali}s} $g\to g'$, if
  $f'\cdot h=k\cdot f$ and $g'\cdot k=h\cdot g$. With the obvious notion of
  composition, \textsc{lari}s and \textsc{rali}s form categories that come
  equipped with forgetful functors into $\C^\two$. Furthermore, if \C\ is an
  \Ord-category, there are \Ord-categories $\lari(\C)$ and
  $\rali(\C)$ with objects and morphisms described above, and
  ordering between morphisms those of $\C^\two$.
\end{df}

\begin{ex}
  \label{ex:6}
  Consider the \emph{free (split) opfibration monad} $\mathsf{M}$ on \Ord, given on $f\colon
  X\to Y$ by $M(f)$
  \begin{equation}
    Kf=f\downarrow 1_{\cod(f)}=\bigl\{(x,y)\in X\times Y:f(x)\leq y\bigr\}\xrightarrow{Mf} Y
    \qquad
    (x,y)\mapsto y
    \label{eq:42}
  \end{equation}
  with ordering inherited from that of $X\times Y$. Furthermore, $M$ is a
  locally monotone endofunctor of $\Ord^\two$. The category
  $\mathsf{M}\text-\mathrm{Alg}$ of algebras for this monad has objects the
 \emph{(split) opfibrations}, ie monotone functions $f\colon X\to Y$ with a choice for
  each $x\in X$ and $y\in Y$ that satisfy $f(x)\leq y$, of an $x_y\in X$ such
  that: $x\leq x_y$, $f(x_y)=y$, and $(x\leq x')\wedge (f(x')=y)$ implies $x_y\leq
  x'$.
  As an aside comment, we note that there is no difference between the notions
  of an opfibration and of a split opfibration in \Ord\ due to the antisymmetry
  property satisfied by the orderings.

  Any monotone function $f\colon X\to Y$ can be factorised as
  \begin{equation}
    \label{eq:43}
    f:X\xrightarrow{Ef}Kf\xrightarrow{Mf} Y
  \end{equation}
  where $Ef(x)=(x,f(x))\in f\downarrow Y=f\downarrow 1_Y$. This is in fact part of an
  \textsc{awfs}, as we proceed to show. As the functorial factorisation is the
  one just described, the locally monotone endofunctor $E$ of $\Ord^\two$ has a
  copoint $\Phi_f=(1_X,Mf)\colon Ef\to f$. The monotone function $Ef\colon X\to
  f\downarrow Y$ has a right adjoint $r_f\colon f\downarrow Y\to X$, given by
  $r_f(x,y)=x_y$. We can define
  \begin{equation}
    \label{eq:44}
    \sigma_f\colon Kf=f\downarrow Y\longrightarrow KEf= Ef\downarrow Kf
    \qquad
    (x,y)\mapsto (r_f(x,y),(x,y))
  \end{equation}
  and morphisms $\Sigma_f$ that form the comultiplication of a comonad
  $\mathsf{E}=(\mathsf{E},\Phi,\Sigma)$.
  \begin{equation}
    \label{eq:45}
    \Sigma_f\colon {E}f\longrightarrow {E}^2f
    \qquad
    \diagram
    X\ar[d]_{{E}f}\ar@{=}[r]&X\ar[d]^{{E}^2f}\\
    Kf\ar[r]^-{\sigma_f}&K{E}(f)
    \enddiagram
  \end{equation}
  The morphism $ME(f)\colon KEf\to Kf$ is a left adjoint to $\sigma_f$, as can
  be easily verified. Furthermore, $\Phi_{Ef}\dashv \Sigma_f$, which means that
  the comonad $\mathsf{E}$ is lax idempotent. The distributivity axiom of
  \textsc{awfs}s can be verified by hand, or, alternatively, one can appeal to
  Theorem~\ref{thm:12}.

  We conclude with the observation that the endofunctors $E$ and $M$ preserve
  filtered colimits; equivalently, the functor $K\colon\Ord^\two\to\Ord$
  preserves filtered colimits. This is so because $K$ is constructed by menas of
  comma-objects and the comments at the end of \S \ref{sec:order-enrich-limits}.
\end{ex}
\begin{ex}
  \label{ex:11}
  Precisely the same construction can be carried out in any \Ord-category that
  admits comma-objects (see~\S \ref{sec:order-enrich-limits}); for example, in
  any \Ord-category that admits cotensor products with $\two$ and pullbacks. The
  morphism $Mf$ is a projection in the comma-object depicted.
  \begin{equation}
    \label{eq:101}
    \diagram
    Kf\ar[r]^-{r_f}\ar[d]_{Mf}\ar@{}[dr]|{\geq}&X\ar[d]^f\\
    B\ar@{=}[r]&B
    \enddiagram
  \end{equation}
  The left part of the factorisation $Ef\colon X\to Kf$ is the unique morphism
  defined by the conditions
  \begin{equation}
    \label{eq:102}
    Mf\cdot Ef=f \quad\text{and}\quad r_f\cdot Ef=1_{X}.
  \end{equation}
  It is not hard to show that $Ef\dashv r_f$.

  The endo-\Ord-functor $f\mapsto Mf$ is part of the free
  (split) opfibration monad on \C. The endo-\Ord-functor $E$ is part of a
  comonad with counit $\Phi^E_f=(1,Mf)\colon Ef\to f$ and comultiplication
  $\Sigma_f=(1,\sigma_f)\colon Ef\to E^2f$ defined by
  \begin{equation}
    r_{Ef}\cdot \sigma_f=r_f \quad\text{and}\quad MEf\cdot\sigma_f=1_{Kf}.
  \end{equation}

\end{ex}
\begin{lemma}
  \label{l:5}
  Suppose that $\C$ is an \Ord-category with comma-objects and
  $(\mathsf{E},\mathsf{M})$ the {\normalfont\textsl{\textsc{awfs}}} constructed
  in the previous example.
  If $\Phi^{{E}}\colon {E}\Rightarrow 1$ is the underlying copointed endofunctor
  of the comonad $\mathsf{E}$, then:
  \begin{enumerate}
  \item \label{item:17} There is an isomorphism
    $\lari(\C)\cong\mathsf{E}\text-\mathrm{Coalg}$ over $\C^\two$.
  \item \label{item:18} The forgetful functor
  \begin{equation}
    \label{eq:80}
    \mathsf{E}\text-\mathrm{Coalg}\longrightarrow({E},\Phi^{{E}})\text-\mathrm{Coalg}
  \end{equation}
  is an isomorphism.
\end{enumerate}
\end{lemma}
\begin{proof}
  This proof follows a direction not suggested by the statement. We shall first
  prove that there is an isomorphism between  $\lari(\C)$ and
  $(E,\Phi^E)\text-\mathrm{Coalg}$ and then show that \eqref{eq:80}~is an
  isomorphism. The reason the lemma is stated in the present form is that this
  form extends to 2-categories~\cite{clementino15:_lax}.

  Suppose given a morphism in $\C^\two$ as depicted.
  \begin{equation}
    \label{eq:95}
    \diagram
    A\ar[d]_f\ar@{=}[r]&A\ar[d]^{Ef}\\
    B\ar[r]^-s&Kf
    \enddiagram
  \end{equation}
  The morphism $s\colon B\to Kf=f\downarrow B$ corresponds to a pair of
  morphisms $r\colon B\to A$ and $u\colon B\to B$ that satisfy $f\cdot r\leq
  u$. The morphisms $r$ and $u$ are the composition of $s$ with, respectively,
  the projections $f\downarrow B\to A$ and $Mf\colon f\downarrow B\to B$. The
  commutativity of~\eqref{eq:95} translates into $r\cdot f=1$ and $u\cdot f=f$.

  Now suppose that \eqref{eq:95}~is a morphism of $(E,\Phi^E)$-coalgebras, ie
  that $Mf\cdot s=1$. By definition of $u$, this is equivalent to saying that
  $u=1$. Therefore, to give an $(E,\Phi^E)$-coalgebra structure on $f$ is
  equivalent to giving a morphism $r\colon B\to A$ such that $f\cdot r\leq 1$
  and $r\cdot f=1$. In other words, an $(E,\Phi^E)$-coalgebra structure on $f$
  is the same as a \textsc{lari} structure on $f$.

  To conclude the proof, we show that any $(E,\Phi^E)$-coalgebra structure
  $(1,s)\colon f\to Ef$ is an $\mathsf{E}$-coalgebra, ie it satisfies the
  coassociativity equality
  \begin{equation}
    \label{eq:103}
    \sigma_f\cdot s=K(1,s)\cdot s.
  \end{equation}
  The codomain of the morphisms at both sides of the equality is $KEf$, so
  \eqref{eq:103}~holds precisely when it does after composing with the
  projections $MEf\colon KEf\to Kf$ and $r_{Ef}\colon KEf\to X$. One of
  these equalities is obvious, since
  \begin{equation}
    MEf\cdot\sigma_f\cdot s=1\cdot s=s = s\cdot 1 = s\cdot Mf\cdot s
    = MEf\cdot K(1,s)\cdot s.
  \end{equation}
  The second equality holds by the following string of equalities, the first of
  which uses the definition of $\sigma_f$ and the last uses $r_{Ef}\cdot
  K(1,s)=r_f$.
  \begin{equation}
    r_{Ef}\sigma_f\cdot s = r_f\cdot s =r_{Ef}\cdot K(1,s)\cdot s.
  \end{equation}
  This completes the proof of the lemma.
\end{proof}

\subsection{Lax orthogonal factorisation systems}
\label{sec:lax-orthogonal-awfss}
\begin{df}
  \label{df:5}
  An \textsc{awfs} $(\mathsf{L},\mathsf{R})$ on an \Ord-category \C\ is a \emph{lax
  orthogonal factorisation system} (abbreviated \textsc{lofs}) if either of the following equivalent conditions holds:
  \begin{itemize}
  \item The comonad $\mathsf{L}$ is lax idempotent.
  \item The monad $\mathsf{R}$ is lax idempotent.
  \end{itemize}
\end{df}
Before proving the equivalence between the above properties we describe more
explicitly what it means for $(\mathsf{L},\mathsf{R})$ to be lax orthogonal.

According to our notation, the unit and multiplication of $\mathsf{R}$ and the
counit and comultiplication of $\mathsf{L}$ are depicted as morphisms in
$\C^\two$ as follows.
\begin{equation}
  \label{eq:56}
  \diagram
  \cdot\ar@{}[dr]|{\Lambda_f}
  \ar[r]^-{Lf}\ar[d]_{f}
  &
  \cdot
  \ar[d]^{Rf}
  \\
  \cdot
  \ar@{=}[r]_-{\phantom{Rf}}
  &
  \cdot
  \enddiagram
  \quad
  \diagram
  \cdot\ar@{}[dr]|{\Pi_f}
  \ar[r]^-{\pi_f}\ar[d]_{R^2f}
  &
  \cdot
  \ar[d]^{Rf}
  \\
  \cdot
  \ar@{=}[r]_-{\phantom{Rf}}
  &
  \cdot
  \enddiagram
  \quad
  \diagram
  \cdot\ar@{}[dr]|{\Phi_f}
  \ar@{=}[r]^-{\phantom{\pi}}\ar[d]_{Lf}
  &
  \cdot
  \ar[d]^{f}
  \\
  \cdot
  \ar[r]_-{Rf}
  &
  \cdot
  \enddiagram
  \quad
  \diagram
  \cdot\ar@{}[dr]|{\Sigma_f}
  \ar@{=}[r]^-{\phantom{Lf}}\ar[d]_{Lf}
  &
  \cdot
  \ar[d]^{L^2f}
  \\
  \cdot
  \ar[r]_-{\sigma_f}
  &
  \cdot
  \enddiagram
\end{equation}
Then, $(\mathsf{L},\mathsf{R})$ is lax orthogonal if and only if any of the
following conditions hold (the equivalence of these conditions will be shown in
Proposition~\ref{prop:1}):
\begin{equation}
  \label{eq:57}
  K(Lf,1)\cdot\pi_f\leq 1\qquad 1\leq LRf\cdot\pi_f
  \qquad 1\leq \sigma_f\cdot RLf \qquad
  \sigma_f\cdot K(1,Rf)\leq 1.
\end{equation}
In terms of $\mathsf{R}$-algebras and $\mathsf{L}$-coalgebras, the lax
idempotency of $(\mathsf{L},\mathsf{R})$ is described as follows. If $(f,s)$ is
an $\mathsf{L}$-coalgebra and $(g,p)$ is an $\mathsf{R}$-algebra, as displayed
below,
\begin{equation}
  \label{eq:58}
  \diagram
  \cdot\ar@{}[dr]|{(f,s)}
  \ar@{=}[r]^-{\phantom{p}}\ar[d]_{f}
  &
  \cdot
  \ar[d]^{Lf}
  \\
  \cdot
  \ar[r]_-{s}
  &
  \cdot
  \enddiagram
  \qquad
  \diagram
  \cdot\ar@{}[dr]|{(g,p)}
  \ar[r]^-{p}\ar[d]_{Rg}
  &
  \cdot
  \ar[d]^g
  \\
  \cdot
  \ar@{=}[r]_-{\phantom{s}}
  &
  \cdot
  \enddiagram
\end{equation}
then the \textsc{awfs} is lax orthogonal if and only if any of the following two
equivalent conditions hold, for all $(f,s)$ and $(g,p)$ (again, the equivalence
will be shown in Proposition~\ref{prop:1}):
\begin{equation}
  \label{eq:59}
  1\leq s\cdot Rf \quad\text{and}\quad 1\leq Lg\cdot p.
\end{equation}

\begin{prop}
  \label{prop:1}
  If $(\mathsf{L},\mathsf{R})$ is an {\normalfont\textsl{\textsc{awfs}}} on an
  \Ord-category \C, then $\mathsf{L}$ is lax idempotent if and only if
  $\mathsf{R}$ is lax idempotent.
\end{prop}
\begin{proof}
  In this proof we use the following general property of \textsc{awfs}s, whose details can
  be found in~\cite[\S 2.8]{MR3393453}. If $(\mathsf{L},\mathsf{R})$ is an
  \textsc{awfs} on an (ordinary) category and $f$, $g$ are two composable
  morphisms each one of which carries an $\mathsf{L}$-coalgebra structure, then
  their composition $g\cdot f$ carries a canonical $\mathsf{L}$-coalgebra
  structure. We regard morphisms of the form $Lf$ as $\mathsf{L}$-coalgebras
  with structure given by the comultiplication $\Sigma_f=(1,\sigma_f)\colon
  Lf\to L^2f$. Furthermore, we use the following fact, whose proof can be found
  in~\cite[\S 3.1]{MR3393453}: the morphism $(1,\pi_f)$ depicted is a morphism
  of $\mathsf{L}$-coalgebras from $LRf\cdot Lf$ to $Lf$.
  \begin{equation}
    \xymatrixrowsep{.55cm}
    \diagram
    A\ar[d]_{Lf}\ar@{=}[r]&
    A\ar[dd]^{Lf}\\
    Kf\ar[d]_{LRf}&\\
    KRf\ar[r]^-{\pi_f}&
    Kf
    \enddiagram
  \end{equation}

  Assuming that $\mathsf{L}$ is lax idempotent,
  we shall show that $\mathsf{R}$ is lax idempotent by exhibiting an inequality $R\Lambda \cdot \Pi\leq 1$, where $\Lambda$ and $\Pi$ are
  the unit and multiplication of the monad. The converse, namely that
  $\mathsf{L}$ is lax idempotent if $\mathsf{R}$ is so, is not necessary to
  prove, as it follows by a duality argument, more specifically, by taking the
  opposite \Ord-category.

  Let $f\colon A\to B$ be a morphism
  of \C, and consider the composition of the morphisms $(1_A,\pi_f)\colon
  LRf\cdot Lf\to Lf$ with $L\Lambda_f=(Lf,K(Lf,1))\colon Lf\to LRf$, as depicted.
  \begin{equation}
    \label{eq:34}
    \xymatrixrowsep{.55cm}
    \diagram
    A\ar[d]_{Lf}\ar@{=}[r]&
    A\ar[dd]_{Lf}\ar[r]^{Lf}&
    Kf\ar[dd]_{LRf}\\
    Kf\ar[d]_{LRf}&
    &
    \\
    KRf\ar[r]^-{\pi_f}&
    Kf\ar[r]^-{K(Lf,1)}&
    KRf
    \enddiagram
  \end{equation}
  The composition of this diagram with the counit $\Phi_{Rf}=(1,R^2f)$ equals the
  morphism $(Lf,R^2f)\colon LRf\cdot Lf\to Rf$, depicted on the right below, since
  \begin{equation}
    \label{eq:35}
    R^2f\cdot K(Lf,1)\cdot\pi_f= Rf\cdot \pi_f= R^2f.
    \qquad
    \xymatrixrowsep{.55cm}
    \diagram
    A\ar[d]_{Lf}\ar[r]^{Lf}&
    Kf\ar[dd]^{Rf}\\
    Kf\ar[d]_{LRf}&\\
    KRf\ar[r]^{R^2f}&B
    \enddiagram
  \end{equation}
  Since $\mathsf{L}$ is lax idempotent, the $\mathsf{L}$-coalgebra morphism~\eqref{eq:34}
  is a left
  lifting of~\eqref{eq:35} through $\Phi_{Rf}$ (see Definition~\ref{df:6}).

  On the other hand, the morphism in $\C^\two$ depicted below is also equal
  to~\eqref{eq:35} upon composition with the counit $\Phi_{Rf}$
  \begin{equation}
    \label{eq:36}
    \xymatrixrowsep{.55cm}
    \diagram
    A\ar[r]^{Lf}\ar[d]_{Lf}&
    Kf\ar[dd]^{LRf}\\
    Kf\ar[d]_{LRf}&\\
    KRf\ar@{=}[r]&
    KRf
    \enddiagram
  \end{equation}
  and by the universal property of liftings we deduce that \eqref{eq:34}~is less
  or equal than~\eqref{eq:36}, so $K(Lf,1)\cdot \pi_f\leq 1_{KRf}$. It remains
  to prove that this defines an inequality in $\C^\two$ with identity codomain
  component; in other words, that the inequality becomes an equality upon
  composition with $R^2f$. But this holds, since both sides become equal:
  \begin{equation}
    \label{eq:41}
    R^2f\cdot K(Lf,1)\cdot \pi_f=Rf\cdot \pi_f=R^2f,
  \end{equation}
  concluding the proof.
\end{proof}

\begin{ex}
  \label{ex:13}
  The \textsc{awfs} $(\mathsf{E},\mathsf{M})$ of Example~\ref{ex:6}, for which
  $\mathsf{M}$-algebras are opfibrations and $\mathsf{E}$-coalgebras are
  \textsc{lari}s, is lax orthogonal. Indeed, the monad $\mathsf{M}$ is
  well-known to be lax idempotent.
\end{ex}

\subsection{Categories of {\normalfont\textsc{awfs}s}}
\label{sec:categ-norm}

There is a category $\mathbf{AWFS}(\C)$ whose objects are \textsc{awfs}s on the
\Ord-category \C. A morphism $(\mathsf{L},\mathsf{R})\longrightarrow
(\mathsf{L}',\mathsf{R}')$ is a natural family of morphisms $\varphi_f$ that make the
following diagrams commute.
\begin{equation}
  \label{eq:106}
  \diagram
  \cdot\ar[d]_{Lf}\ar@{=}[r]&
  \cdot\ar[d]^{L'f}\\
  Kf\ar[r]^-{\varphi_f}\ar[d]_{Rf}&
  K'f\ar[d]^{R'f}\\
  \cdot\ar@{=}[r]&
  \cdot
  \enddiagram
\end{equation}
Furthermore, the morphisms $(1,\varphi_f)\colon Lf\to L'f$ must form a comonad
morphism $\mathsf{L}\to\mathsf{L}'$, and the morphisms $(\varphi_f,1)\colon
Rf\to R'f$ must form a monad morphism $\mathsf{R}\to\mathsf{R}'$.

There is a full subcategory $\mathbf{LOFS}(\C)$ of $\mathbf{AWFS}(\C)$
consisting of the \textsc{lofs}s.
\begin{lemma}
  \label{l:16}
  $\mathbf{LOFS}(\C)$ is a preorder.
\end{lemma}
\begin{proof}
  If the morphisms $\varphi_f$ as in~\eqref{eq:106} form a morphism from
  $(\mathsf{L},\mathsf{R})$ to $(\mathsf{L}',\mathsf{R}')$, then the morphisms
  $(\varphi_f,1)\colon Rf\to R'f$ define a morphism of monads. There can only be
  one such, by Lemma~\ref{l:15}.
\end{proof}

\section{Lifting operations}
\label{sec:lifting-operations}
In this section we introduce \textsc{kz} lifting operations and  explain the
motivation behind the definition of lax orthogonal factorisation
systems. Before all
that, we must say something about how lifting operations work in relation to
\textsc{awfs}s on \Ord-categories.
\subsection{Lifting operations on \Ord-categories}
\label{sec:lift-oper-categ}
Suppose that $U\colon\mathcal{A}\to\C^\two\leftarrow\mathcal{B}\colon V$ are
locally monotone functors between \Ord-categories.
A \emph{lifting operation} from $U$ to $V$ can be described as a
choice of a diagonal filler $\phi_{a,b}(h,k)$ for each morphism $(h,k)\colon
Ua\to Vb$ in $\C^\two$.
\begin{equation}
  \label{eq:48}
  \xymatrixcolsep{2cm}
  \diagram
  \cdot\ar[r]^h\ar[d]_{Ua}&\cdot\ar[d]^{Vb}\\
  \cdot\ar[r]_k\ar@{..>}[ur]|{\phi_{a,b}(h,k)}&\cdot
  \enddiagram
\end{equation}
These diagonal fillers must satisfy a naturality condition with respect to
morphisms in $\mathcal{A}$ and $\mathcal{B}$. If $\alpha\colon a'\to a$ and
$\beta\colon b\to b'$ are morphisms in $\mathcal{A}$ and $\mathcal{B}$
respectively, then
\begin{equation}
  \label{eq:50}
  \phi_{a',b'}
  \big(\dom V\beta\cdot h\cdot \dom U\alpha,\cod V\beta\cdot k\cdot\cod
  U\alpha\big)
  = (\dom V\beta)\cdot\phi_{a,b}(h,k)\cdot (\cod U\alpha)
\end{equation}
as depicted in the following diagram.
\begin{equation}
  \label{eq:49}
  \xymatrixcolsep{2cm}
  \diagram
  \cdot\ar[r]^{\dom U\alpha}\ar[d]_{Ua'}&
  \cdot\ar[r]^h\ar[d]_{Ua}&
  \cdot\ar[d]^{Vb}\ar[r]^{\dom V\beta}&
  \cdot\ar[d]^{Vb'}\\
  \cdot\ar[r]_{\cod U\alpha}\ar@{..>}[urrr]&
  \cdot\ar[r]_k\ar@{..>}[ur]
  &
  \cdot\ar[r]_{\cod V\beta}&
  \cdot
  \enddiagram
\end{equation}
So far, the definition of lifting operation is the one given
in~\cite{MR2506256}, but our categories are enriched in \Ord\ and the functors
$U$ and $V$ are locally monotone, so we require that the diagonal filler
satisfies: if $(h,k)$ and $(h',k')\colon Ua\to Vb$ are commutative squares in \C\
with $(h,k)\leq (h',k')$ (ie $h\leq h'$ and $k\leq k'$) then
\begin{equation}
  \phi_{a,b}(h,k)\leq\phi_{a,b}(h',k').\label{eq:51}
\end{equation}

\subsection{Lifting operations from \Ord-functorial factorisations}
\label{sec:lift-oper-from}
The idea of a functorial factorisation $\dom\Rightarrow E\Rightarrow \cod$, as
defined in Definition~\ref{df:4}, is that it induces a canonical lifting
operation between the forgetful \Ord-functors $U$ and $V$
\begin{equation}
  U\colon(L,\Phi)\text-\mathrm{Coalg}\longrightarrow\C^\two
  \longleftarrow(R,\Lambda)\text-\mathrm{Alg}\colon V.\label{eq:148}
\end{equation}
Here $\Phi\colon L\Rightarrow 1_{\C^\two}$ and
$\Lambda\colon 1_{\C^\two}\Rightarrow R$ are, respectively, the copointed
endo-\Ord-functor and the pointed endo-\Ord-functor on $\C^\two$ associated to
the given \Ord-functorial factorisation.

A coalgebra for $(L,\Phi)$ can be depicted as the commutative square on the left
below, while an algebra for $(R,\Lambda)$ is a commutative square on the right
\begin{equation}
  \label{eq:52}
  (f,s)\quad=\quad
  \diagram
  \cdot\ar[d]_f\ar@{=}[r]&\cdot\ar[d]^{Lf}\\
  \cdot\ar[r]^{s}&\cdot
  \enddiagram
  \qquad
  (g,p)\quad=\quad
  \diagram
  \cdot\ar[r]^{p}\ar[d]_{Rg}&\cdot\ar[d]^g\\
  \cdot\ar@{=}[r]&\cdot
  \enddiagram
\end{equation}
satisfying $Rf\cdot s=1$ and $p\cdot Lg=1$. Given a commutative square
$(h,k)\colon f\to g$, there is a canonical diagonal filler
\begin{equation}
  \label{eq:53}
  \phi_{(f,s),(g,p)}(h,k)=p\cdot E(h,k)\cdot s.
\end{equation}
It is immediate to see that these diagonal fillers form a lifting operation from
$U$ to $V$.

\begin{rmk}
  Even though an (\Ord-)functorial factorisation $f=Rf\cdot Lf$ as the one
  discussed in the previous paragraphs yields a lifting operation of
  $(L,\Phi)$-coalgebras against $(R,\Lambda)$-algebras, there is no guarantee
  of being able to find a diagonal filler for a commutative diagram of the
  form
  \begin{equation}
    \label{eq:54}
    \diagram
    \cdot\ar[d]_{Lf}\ar[r]&\cdot\ar[d]^{Rg}\\
    \cdot\ar[r]&\cdot
    \enddiagram
  \end{equation}
  since $Lf$ may not support an $(L,\Phi)$-coalgebra structure, and $Rg$ may not
  support an $(R,\Lambda)$-algebra structure. A natural way of endowing $Lf$ and $Rg$
  with the corresponding structures is to require that $(L,\Phi)$ extends to a
  comonad and $(R,\Lambda)$ extends to a monad; in this way, $Lf$ is a (cofree)
  coalgebra and $Rg$ is a (free) algebra. This one of the reasons for the form that the
  definition of \textsc{awfs} takes (see Definition~\ref{df:3}).\label{rmk:6}
\end{rmk}
There is an useful fact that is worth including at this point, and will be
useful in the proof of Theorem~\ref{thm:1}.
\begin{lemma}
  \label{l:2}
  For any {\normalfont\textsl{\textsc{awfs}}} $(\mathsf{L},\mathsf{R})$, the
  diagonals $\phi_{Lf,Rf}(Lf,Rf)$ are identity morphisms.
  \begin{equation}
    \label{eq:64}
    \diagram
    {\cdot}
    \ar[r]^-{Lf}\ar[d]_{Lf}
    &
    {\cdot}
    \ar[d]^{Rf}
    \\
    {\cdot}
    \ar[r]_-{Rf}\ar[ur]^1
    &
    {\cdot}
    \enddiagram
  \end{equation}
\end{lemma}
\begin{proof}
  If we write the commutative square of the statement as a pasting of two
  commutative squares $(1,Rf)$ and $(Lf,1)$, as displayed, we can easily compute
  the diagonal filler.
  \begin{equation}
    \label{eq:65}
    \diagram
    \cdot\ar[d]_{Lf}\ar@{=}[r]&
    \cdot\ar[d]_f\ar[r]^{Lf}&
    \cdot\ar[d]^{Rf}\\
    \cdot\ar[r]^{Rf}&
    \cdot\ar@{=}[r]&
    \cdot
    \enddiagram
  \end{equation}
  \begin{equation}
    \phi_{Lf,Rf}(Lf,Rf)=\pi_f\cdot K(Lf,Rf)\cdot\sigma_f=
    \pi_f\cdot K(Lf,1)\cdot K(1,Rf)\cdot\sigma_f=1\cdot 1=1.\qedhere
  \end{equation}
\end{proof}
\begin{rmk}
  \label{rmk:1}
  As pointed out in~\cite[\S 2.5]{MR3393453}, the commutativity of the two
  diagrams~\eqref{eq:30} that express the fact that $\Delta\colon LR\Rightarrow
  RL$ is a mixed distributive law is equivalent to the requirement that the
  diagonal filler of the displayed square be $\sigma_f\cdot\pi_f$.
  \begin{equation}
    \label{eq:108}
    \diagram
    Kf\ar[d]_{LRf}\ar[r]^-{\sigma_f}&
    KLf\ar[d]^{RLf}\\
    KRf\ar[r]_-{\pi_f}\ar@{..>}[ur]|{\sigma_f\cdot \pi_f}&
    Kf
    \enddiagram
  \end{equation}
\end{rmk}

\subsection{KZ lifting operations}
\label{sec:kz-lift-oper}

In the previous section we saw that each \textsc{awfs} canonically induced a
lifting operation. It is logical to expect that lifting operations that arise
from lax orthogonal \textsc{awfs}s carry extra structure. In this section we
identify this structure.
\begin{df}
  \label{df:7}
  Suppose given a lifting operation $\phi$ from $U\colon\mathcal{A}\to\C^\two$ to
  $V\colon\mathcal{B}\to\C^\two$ on an \Ord-category \C\ as defined in \S
  \ref{sec:lift-oper-categ}. We say that $\phi$ is a \slkz-\emph{lifting
  operation} if, for all $a\in\mathcal{A}$, $b\in\mathcal{B}$ and each
  commutative diagram as on the left, the inequality on the right holds.
  \begin{equation}
    \label{eq:55}
    \diagram
    \cdot\ar[d]_{Ua}\ar[r]^h&\cdot\ar[d]^{Vb}\\
    \cdot\ar[r]_k\ar[ur]^d&\cdot
    \enddiagram
    \quad\Longrightarrow{}\quad
    \phi_{a,b}(h,k)\leq d
  \end{equation}
  In other words, the diagonal filler given by the lifting operation $\phi$ is a
  lower bound of all possible diagonal fillers.
\end{df}
\begin{ex}
  \label{ex:8}
  Consider the \Ord-functor $0\colon \mathbf{1}\to\two$ that includes the
  terminal ordered set as the initial element of the ordered set $\two=(0\leq 1)$.
  There is a bijection between opfibration structures on a morphism $g\colon
  X\to Y$ in \Ord\ and \textsc{kz} lifting operations on $g$ against the
  morphism $0$. To see this, first notice that a commutative square
  \begin{equation}
    \label{eq:72}
    \diagram
    \mathbf{1}\ar[r]\ar[d]_0&X\ar[d]^g\\
    \two\ar[r]\ar@{..>}[ur]&Y
    \enddiagram
  \end{equation}
  is equally well given by an element $x\in X$ and an element $y\in Y$ such that
  $g(x)\leq y$. The existence of a diagonal filler is the existence of an element
  $x_y\in X$ with $x\leq x_y$ and  $g(x_y)=y$. This diagonal filler is a lower
  bound if for any other $x\leq \bar x$ with $g(\bar x)=y$ there is an
  inequality $x_y\leq \bar x$. The element $x_y$ is unique and the assignment
  $(x,y)\mapsto x_y$
  defines a split opfibration structure on $g$.
\end{ex}

\begin{thm}
  \label{thm:1}
  The following conditions are equivalent for an
  {\normalfont\textsl{\textsc{awfs}}} $(\mathsf{L},\mathsf{R})$ on an
  \Ord-category \C.
  \begin{enumerate}
  \item \label{item:1} The {\normalfont\textsl{\textsc{awfs}}} is a
    {\normalfont\textsl{\textsc{lofs}}}.
  \item \label{item:2} The lifting operation from the forgetful functor $U\colon
    \mathsf{L}\text-\mathrm{Coalg}\to\C^\two$ to the forgetful functor
    $V\colon\mathsf{R}\text-\mathrm{Alg}\to\C^\two$ is a
    \slkz-lifting operation.
  \end{enumerate}
\end{thm}
\begin{proof}
  Assume that $(\mathsf{L},\mathsf{R})$ is lax orthogonal, $(f,s)$ is an
  $\mathsf{L}$-coalgebra and $(g,p)$ is an $\mathsf{R}$-algebra. Given a
  diagonal filler $d$ as depicted, we must show $\phi_{(f,s),(g,p)}(h,k)\leq d$.
  \begin{equation}
    \label{eq:60}
    \diagram
    {\cdot}
    \ar[r]^-{h}\ar[d]_{f}
    &
    {\cdot}
    \ar[d]^{g}
    \\
    {\cdot}
    \ar[r]_-{k}\ar[ur]^d
    &
    {\cdot}
    \enddiagram
  \end{equation}
  Using the inequalities $1\leq s\cdot Rf$ and $1\leq Lg\cdot p$
  from~\eqref{eq:59}, we obtain
  \begin{equation}
    \big(
    \phi_{(f,s),(g,p)}(h,k)=p\cdot K(h,k)\cdot s\leq d
    \big)
    \Leftrightarrow
    \big(K(h,k)\leq Lg\cdot d\cdot Rf
    \big).
    \label{eq:62}
  \end{equation}
  There is a morphism $(Lg\cdot d\cdot Rf,k)\colon Rf\to Rg$ in $\C^\two$,
  as shown by the diagram below, which precomposed with the unit
  $\Lambda_f=(Lf,1)\colon f\to Rf$ of $\mathsf{R}$ equals $\Lambda_g\cdot
  (h,k)=(Lg\cdot h,k)\colon f\to Rg$.
  \begin{equation}
    \label{eq:61}
    \diagram
    \cdot\ar[r]^{Rf}\ar[d]_{Rf}&
    \cdot\ar[drr]_k\ar[r]^d&
    \cdot\ar[dr]^{g}\ar[r]^{Lg}&
    \cdot\ar[d]^{Rg}\\
    \cdot\ar[rrr]^k&&&
    \cdot
    \enddiagram
  \end{equation}
  On the other hand, by the lax idempotency of $\mathsf{R}$, we have that
  $K(h,k)$ is a left extension of $\Lambda_g\cdot (h,k)$ along $\Lambda_f$, so
  there exists $K(h,k)\leq Lg\cdot d\cdot Rf$, as desired.

  Conversely, assume that the lifting operation $\phi$ induced by the \textsc{awfs} is
  \textsc{kz}, and consider the commutative square
  \begin{equation}
    \label{eq:63}
    \xymatrixrowsep{.6cm}
    \xymatrixcolsep{1.7cm}
    \diagram
    \cdot\ar@{..>}[dr]^1\ar[rr]^{LRf}\ar[dd]_{LRf}&&\cdot\ar[dd]^{R^2f}\\
    &\cdot\ar[ur]|{LRf}\ar@{..}[dr]|{Rf}&\\
    \cdot\ar[rr]_{R^2f}\ar[ur]_{\pi_f}&&\cdot
    \enddiagram
  \end{equation}
  By Lemma~\ref{l:2}, $\phi$ provides the diagonal filler
  $\phi_{LRf,R^2f}(LRf,R^2f)=1$, so we have an inequality $1\leq LRf\cdot\pi_f$
  as required.
\end{proof}

\begin{thm}
  \label{thm:10}
  Let $(\mathsf{L},\mathsf{R})$ be a {\normalfont\textsl{\textsc{lofs}}} on an
  \Ord-category \C. Then, the following statements about a morphism $f$ of \C\
  are equivalent:
  \begin{enumerate}
  \item \label{item:23} $f$ has an (unique) $\mathsf{R}$-algebra structure
    (we simply say that $f$ is an $\mathsf{R}$-algebra).
  \item \label{item:24} $f$ is injective with respect to
    $\mathsf{L}$-coalgebras, in the sense that any commutative square
    \begin{equation}
      \xymatrixrowsep{.6cm}
      \diagram
      \cdot\ar[d]_\ell\ar[r]&\cdot\ar[d]^f\\
      \cdot\ar[r]&\cdot
      \enddiagram
    \end{equation}
    with $\ell\in\mathsf{L}\text-\mathrm{Coalg}$ has a diagonal filler.
  \item \label{item:25} $f$ admits a (non-necessarily unique)
    $(R,\Lambda)$-algebra structure.
  \item \label{item:26} $f$ is a retract in $\C^\two$ of an $\mathsf{R}$-algebra.
  \end{enumerate}
  The {\normalfont\textsl{\textsc{wfs}}} that underlies
  $(\mathsf{L},\mathsf{R})$ has as left part those morphisms in the image of the
  forgetful functor $\mathsf{L}\text-\mathrm{Coalg}\to\C^\two$ and as right part those
  morphisms in the image of the forgetful functor $\mathsf{R}\text-\mathrm{Alg}\to\C^\two$.
\end{thm}
\begin{proof}
  We have seen in \S \ref{sec:lift-oper-from} that
  \eqref{item:23}~implies~\eqref{item:24}. To prove that
  \eqref{item:24}~implies~\eqref{item:25},
  consider the diagonal filler below, which shows that $(p,1)\colon Rf\to f$ is
  is an $(\mathsf{R},\Lambda)$-algebra structure.
  \begin{equation}
    \diagram
    \cdot\ar[d]_{Lf}\ar@{=}[r]&\cdot\ar[d]^f\\
    \cdot\ar@{..>}[ur]^p\ar[r]_{Rf}&\cdot
    \enddiagram
  \end{equation}
  The implications
  \eqref{item:25}$\Rightarrow$\eqref{item:26}$\Rightarrow$\eqref{item:23} are
  particular instances of part of Lemma~\ref{l:14}, since $\mathsf{R}$ is lax
  idempotent.

  As mentioned in \S \ref{sec:underly-norm}, the underlying \textsc{wfs}
  $(\mathcal{L},\mathcal{R})$ of $(\mathsf{L},\mathsf{R})$ has as right class
  the algebras for the pointed endofunctor $(R,\Lambda)$. Then,
  $f\in\mathcal{R}$ (or, by duality, $f\in\mathcal{L}$) precisely when $f$ is an
  $\mathsf{R}$-algebra (an $\mathsf{L}$-coalgebra).
\end{proof}

\section{Horizontally ordered double categories and {\normalfont\textsc{lofs}s}}
\label{sec:double-categories}

\subsection{Horizontally ordered double categories}
\label{sec:horiz-order-double-1}

Double categories, introduced by
C.~Ehresmann~\cite{MR0152561}, can be succinctly described as internal
categories in the
cartesian category of categories. They consist of an internal graph of categories
and functors $\mathcal{G}_1\rightrightarrows\mathcal{G}_0$ (domain and codomain)
with an identity
functor $\id\colon\mathcal{G}_0\to\mathcal{G}_1$ and a composition functor
$\mathcal{G}_1\times_{\mathcal{G}_0}\mathcal{G}_1\to\mathcal{G}_1$ that satisfy
the usual associativity and identity axioms. The morphisms of $\mathcal{G}_0$
will be represented as horizontal arrows. The objects of $\mathcal{G}_1$ have a
domain and a codomain that are objects of $\mathcal{G}_0$, and will be
represented as vertical morphisms. Morphisms of $\mathcal{G}_1$ will be
represented as squares; for example a morphism $\alpha\colon x \to y$ in
$\mathcal{G}_1$ will be represented as
\begin{equation}
  \diagram
  \cdot\ar[d]_x\ar[r]\ar@{}[dr]|\alpha&\cdot\ar[d]^y\\
  \cdot\ar[r]&\cdot
  \enddiagram
\end{equation}
Objects of $\mathcal{G}_1$, ie vertical arrows, can be vertically composed, as
well as squares as the one above. 

\begin{df}
  \label{df:21}
  A \emph{horizontally ordered double category} is an internal category in the
  cartesian category
  $\Ord\text-\Cat$ of \Ord-categories and \Ord-functors. This means that in a horizontally ordered double category we can speak of
inequalities between horizontal morphisms and between squares.
A \emph{monotone double functor} between two horizontally
  ordered double categories is a double functor that preserves the inequalities
  between horizontal morphisms and between squares.
\end{df}

\begin{ex}
  \label{ex:9}
  Let \C\ be an \Ord-category. The horizontally ordered double category
  $\Sq(\C)$ has underlying graph $\dom$,
  $\cod\colon\C^\two\rightrightarrows\C$, so both horizontal and vertical
  morphisms are morphisms of $\C$, and squares are commutative squares in
  \C. The inequality between horizontal morphisms is the inequality between morphisms of
  \C. One square is less or equal than another, as depicted,
  \begin{equation}
    \label{eq:116}
    \diagram
    \cdot\ar[d]_x\ar[r]^h&\cdot\ar[d]^y\\
    \cdot\ar[r]^k&\cdot
    \enddiagram
    \qquad\leq\qquad
    \diagram
    \cdot\ar[d]_x\ar[r]^u&\cdot\ar[d]^y\\
    \cdot\ar[r]_v&\cdot
    \enddiagram
  \end{equation}
  if and only if $h\leq u$ and $k\leq v$.
\end{ex}
\begin{ex}
  \label{ex:10}
  \textsc{lari}s form a horizontally ordered double category.
  If $f\colon A\to B$ and $g\colon B\to C$ are \textsc{lari}s, with respective
  right adjoints $f^*$ and $g^*$, then their composition $g\cdot f\colon A\to B$
  is also a \textsc{lari} with right adjoint $f^*\cdot g^*$. This composition of
  \textsc{lari}s is clearly associative and has identities, namely the identity
  morphisms.
\end{ex}

\subsection{Lifting operations}
\label{sec:lifting-operations-1}
If $U\colon \mathcal{J}\to\C^\two$ is an \Ord-functor, there is an \Ord-category
$\mathcal{J}^{\pitchfork_{\mathkz}}$ over $\C^\two$ whose objects are morphisms $f$ of $\C$
with a \kz-lifting operation against $U$, ie with a \textsc{rali} structure on each
\begin{equation}
  \label{eq:133}
  \phi_{-,f}\colon
  \C(\cod Uj,\dom f)\longrightarrow\C^\two(Uj,f).
\end{equation}
A morphism is a morphism in $\C^\two$ that is compatible with these
\textsc{rali} structures in the obvious way. The ordering of morphisms is that
of $\C^\two$. The forgetful \Ord-functor
\begin{equation}
  \label{eq:132}
  U^{\pitchfork_{\mathkz}}\colon\mathcal{J}^{\pitchfork_{\mathkz}}\longrightarrow\C^\two
\end{equation}
is injective on objects, since \eqref{eq:133}~can be a \textsc{rali} in a unique way.

The construction
$(\mathcal{J},U)\mapsto(\mathcal{J}^{\pitchfork_{\kz}},U^{\pitchfork_{\mathkz}})$ is
part of a functor
\begin{equation}
  \label{eq:107}
  (-)^{\pitchfork_{\mathkz}}\colon
  (\Cat/\C^\two)^{\mathrm{op}}
  \longrightarrow
  \mathbf{CAT}/\C^\two.
\end{equation}
Explicitly, if $S\colon \mathcal{J}\to \mathcal{I}$ is an \Ord-functor over
$\C^\two$
\begin{equation}
  \xymatrixrowsep{.5cm}
  \diagram
  \mathcal{J}\ar[rr]^-S\ar[dr]_U&&\mathcal{I}\ar[dl]^{V}\\
  &\C^\two&
  \enddiagram
\end{equation}
then there is an \Ord-functor
\begin{equation}
  S^{\pitchfork_{\mathkz}}\colon
  \mathcal{I}^{\pitchfork_{\mathkz}}
  \longrightarrow
  \mathcal{J}^{\pitchfork_{\mathkz}}
\end{equation}
defined by the obvious observation that if the morphism on the left hand side of~\eqref{eq:136} is
a \textsc{rali}, then so is the one on the right hand side, since $Uj=VSj$.
\begin{equation}
  \label{eq:136}
  \C(\cod Vi,\dom f)\longrightarrow\C^\two(Vi,f)
  \qquad
  \C(\cod Uj,\dom f)\longrightarrow\C^\two(Uj,f).
\end{equation}

\begin{prop}
  \label{prop:7}
  Given an \Ord-functor $U\colon \mathcal{J}\to\C^\two$, there is a horizontally
  ordered double category with:
  \begin{itemize}
  \item objects, those of $\C$;
  \item vertical morphisms those morphisms of $\C$ that are objects of
    $\mathcal{J}^{\pitchfork_{\mathkz}}$;
  \item horizontal morphisms, the morphisms of $\C$;
  \item squares, commutative squares in $\C$.
  \end{itemize}
  We denote this horizontally ordered category by
  $\mathcal{J}^{\pitchfork_{\mathkz}}$.
  Moreover, $U$ defines an identity on objects double functor
  $\mathcal{J}^{\pitchfork_{\mathkz}}\to\Sq(\C)$.
\end{prop}
\begin{proof}
  We have to  prove the following: (a) if $f$ and $g$ are two composable
  morphisms and both are in $\mathcal{J}^{\pitchfork_{\mathkz}}$, then their
  composition $g\cdot f$ is also in
  $\mathcal{J}^{\pitchfork_{\mathkz}}$; (b) this composition is associative; (c)
  that any identity morphism is
  an object of $\mathcal{J}^{\pitchfork_{\mathkz}}$; (d) identity morphisms are
  identities for the composition of part (a).

  The first observation is that (b) and (d) are automatic because
  \eqref{eq:132}~is injective on objects, so we only need to prove (a) and (c).

  (a)~Suppose that $f$ and $g$ are composable objects of
  $\mathcal{J}^{\pitchfork_{\mathkz}}$, with lifting operations that we denote,
  respectively, $\phi_{-,f}$ and $\phi_{-,g}$. If $j\in\mathcal{J}$, then
  $\theta_j(h,k)\coloneq \phi_{j,f}(h,\phi_{j,g}(f\cdot h,k))$ provides a diagonal
  filler for the solid square $(h,k)\colon Uj\to g\cdot f$, as displayed.
  \begin{equation}
    \xymatrixcolsep{4cm}
    \diagram
    \cdot\ar[dd]_{Uj}\ar[r]^h&\cdot\ar[d]^f\\
    &\cdot\ar[d]^g\\
    \cdot\ar[r]_k\ar@{..>}[ur]|{\phi_{j,g}(f\cdot h,k)}
    \ar@{..>}[uur]^{\phi_{j,f}(h,\phi_{j,g}(f\cdot h,k))}&\cdot
    \enddiagram
  \end{equation}
  To prove that the lifting operation $\theta$ is a \textsc{kz}-lifting
  operation we have to prove that $\theta_j(h,k)$ is the least diagonal
  filler. Suppose that $d$ is another diagonal filler of the square. This
  implies that $f\cdot d$ is a diagonal filler of the square
  $(f\cdot h,k)\colon Uj\to g$, and therefore $\phi_{j,g}(f\cdot h,k)\leq f\cdot
  d$. We now have two morphisms in $\C^\two$, namely
  \begin{equation}
    (h,\phi_{j,g}(f\cdot d,k))\leq (h,f\cdot d)\colon Uj\longrightarrow f
  \end{equation}
  from where we obtain the required inequality
  \begin{equation}
    \label{eq:138}
    \theta_j(h,k)=\phi_{j,f}(h,\phi_{j,g}(f\cdot d,k))\leq \phi_{j,f}(h,f\cdot d) \leq d;
  \end{equation}
  the first inequality in~\eqref{eq:138} above arises from the fact that the lifting operation
  $\phi$ is \Ord-enriched (see \S \ref{sec:lift-oper-categ}), while the second
  inequality exists because $d$ is a diagonal filler of $(h,f\cdot d)\colon
  Uj\to f$.

  (c)~It remains to prove that identity morphisms are in
  $\mathcal{J}^{\pitchfork_{\mathkz}}$, for which we note that there is only one
  possible diagonal filler for a square of the form
  \begin{equation}
    \diagram
    \cdot\ar[d]_{Uj}\ar[r]&\cdot\ar[d]^1\\
    \cdot\ar[r]_k&\cdot
    \enddiagram
  \end{equation}
  namely, $k$ itself. This completes the proof.
\end{proof}

Given an \Ord-functor $U\colon\mathcal{J}\to\C^\two$, there is another
\begin{equation}
  \label{eq:75}
  \prescript{\pitchfork_{\mathkz}}{}{U}\colon
  \prescript{\pitchfork_{\mathkz}}{}{\mathcal{J}} \longrightarrow \C^\two
\end{equation}
that is constructed dually to $\mathcal{J}^{\pitchfork_{\mathkz}}$. More
explicitly, $\prescript{\pitchfork_{\mathkz}}{}{\mathcal{J}}$ has objects
$(f,\phi_{f,-})$ where $f\in\C^\two$ and $\phi$ is a \kz-lifting operation from $f$ to
$U$.
\begin{equation}
  \label{eq:73}
  \xymatrixcolsep{1.5cm}
  \diagram
  \cdot\ar[d]_f\ar[r]^h&\cdot\ar[d]^{Uj}\\
  \cdot\ar[r]_-k\ar@{..>}[ur]|{\phi_{f,j}(h,k)}&\cdot
  \enddiagram
\end{equation}
The \kz-lifting operation $\phi_{f,-}$ is a \textsc{rali} structure on the monotone
morphisms $\C(\cod(f),\dom Uj)\to\C^\two(f, Uj)$.
\begin{thm}
  \label{thm:5}
  Suppose given \Ord-functors
  \begin{equation}
    \label{eq:76}
    \mathcal{J}\xrightarrow{\phantom{m}U\phantom{m}}
    \C^\two \xleftarrow{\phantom{m}V\phantom{m}}\mathcal{I}.
  \end{equation}
  There is a bijection between:
  \begin{itemize}
  \item {\normalfont\slkz}-lifting operations from $U$ to $V$;
  \item \Ord-functors $\mathcal{I}\to\mathcal{J}^{\pitchfork_{\mathkz}}$;
  \item \Ord-functors
    $\mathcal{J}\to\prescript{\pitchfork_{\mathkz}}{}{\mathcal{I}}$.
  \end{itemize}
  These correspondences yield a contravariant adjunction in
  $\Ord\text-\Cat/\C^\two$ between $\prescript{\pitchfork_{\mathkz}}{}{(-)}$ and
  $(-)^{\pitchfork_{\mathkz}}$.
\end{thm}

\subsection{{\normalfont\textsc{lofs}s} and {\normalfont\textsc{kz}} lifting operations}
\label{sec:norm-lift-oper}
Suppose that $(\mathsf{L},\mathsf{R})$ is a \textsc{lofs} on the \Ord-category
\C. There is an \Ord-functor
\begin{equation}
  \label{eq:134}
  \mathsf{R}\text-\mathrm{Alg}\longrightarrow\mathsf{L}\text-\mathrm{Coalg}^{\pitchfork_{\mathkz}}
\end{equation}
introduced in~\cite{clementino15:_lax}, that equips each $\mathsf{R}$-algebra
with its canonical \kz-lifting operation against $\mathsf{L}$-coalgebras (see
Theorem~\ref{thm:1}). Using~\cite[\S 6.3]{MR3393453} one could deduce that
\eqref{eq:134}~is an isomorphism. We prefer, however, to give a self-contained
proof.
\begin{thm}
  \label{thm:9}
  The \Ord-functor~\eqref{eq:134} induced by a
  {\normalfont\textsl{\textsc{lofs}}} $(\mathsf{L},\mathsf{R})$ is an isomorphism.
\end{thm}
\begin{proof}
  Supposing that $(g,\phi_{-,g})$ is a \kz-lifting operation against the
  forgetful \Ord-functor
  $U\colon\mathsf{L}\text-\mathrm{Coalg}\to\C^\two$, we want to construct an
  $\mathsf{R}$-algebra structure on $g\colon A\to B$. There is a \kz-diagonal
  filler $p=\phi_{Lg,g}(1,Rg)$ as depicted below.
  \begin{equation}
    \diagram
    A\ar[d]_{Lg}\ar@{=}[r]&A\ar[d]^g\\
    Kg\ar[r]_-{Rg}\ar@{..>}[ur]^p&B
    \enddiagram
  \end{equation}
  Then $(p,1)\colon Rg\to g$ will be our candidate for an algebra structure. By
  the lax idempotency of $\mathsf{R}$, we only have to show $(p,1)\dashv
  \Lambda_g=(Lg,1)$ (see \S \ref{sec:lax-orthogonal-awfss}). We know that $p\cdot Lg=1$, and it remains to show $1\leq
  Lg\cdot p$. The commutativity the following diagram shows that $Lg\cdot p$ is
  a diagonal filler of the square $(Lg,Rg)\colon Lg\to Rg$.
  \begin{equation}
    \xymatrixrowsep{.5cm}
    \xymatrixcolsep{.5cm}
    \diagram
    A\ar[dd]_{Lg}\ar[dr]_1\ar[rr]^-{Lg}&&Kg\ar[dd]^{Rg}\\
    &A\ar[ur]^{Lg}\ar[dr]^g&\\
    Kg\ar[ur]^p\ar[rr]_-{Rg}&&B
    \enddiagram
  \end{equation}
  The canonical \kz-lifting operation, exhibited in Theorem~\ref{thm:1}, chooses
  the identity morphism as the diagonal filler of the outer square, by
  Lemma~\ref{l:2}, so we deduce $1\leq Lg\cdot p$. This completes the proof that
  $(p,1)\colon Rg\to g$ is an algebra structure.

  The next part of the proof is the verification that the assignment
  \begin{equation}
    \label{eq:135}
    \mathrm{ob}\bigl(
    \mathsf{L}\text-\mathrm{Alg}^{\pitchfork_{\mathkz}}
    \bigr)
    \longrightarrow
    \mathrm{ob}\bigl(\mathsf{R}\text-\mathrm{Alg}\bigr)
  \end{equation}
  constructed in the previous paragraph is an inverse of the effect
  of~\eqref{eq:134} on objects. Both \eqref{eq:134}~and~\eqref{eq:135} commute
  with the injective forgetful assignments from
  $ \mathrm{ob}\bigl( \mathsf{L}\text-\mathrm{Alg}^{\pitchfork_{\mathkz}}
  \bigr)$
  and $ \mathrm{ob}\bigl(\mathsf{R}\text-\mathrm{Alg}\bigr)$ to
  $\mathrm{ob}\bigl(\C^\two \bigr)$. This immediately implies that
  \eqref{eq:135}~is the inverse of~\eqref{eq:134} on objects.

  It remains to prove that \eqref{eq:134}~is fully faithful, in the
  \Ord-enriched sense. Suppose that $(h,k)\colon (f,\phi_{-,f})\to (g,\phi_{-,g})$ is a
  morphism in $\mathsf{L}\text-\mathrm{Coalg}^{\pitchfork_{\mathkz}}$, and let
  $p_f\colon Rf\to f$ and $p_g\colon Rg\to g$ be the associated algebra
  structures. We have the following string of equalities
  \begin{equation}
    h\cdot p_f=h\cdot\phi_{Lf,f}(1,Rf)=\phi_{Lf,g}(h,k\cdot Rf)= \phi_{Lg,g}(1,Rg)\cdot
    K(h,k)=p_g\cdot K(h,k),
  \end{equation}
  which are a result of the definition of lifting operations.
  \begin{equation}
    \diagram
    \cdot\ar[d]_{Lf}\ar@{=}[r]&
    \cdot\ar[r]^-h\ar[d]^f&
    \cdot\ar[d]^g\\
    \cdot\ar[r]_-{Rf}\ar@{..>}[ur]|{p_f}&
    \cdot\ar[r]_-k&
    \cdot
    \enddiagram
    =
    {
      \xymatrixcolsep{2.3cm}
      \diagram
      \cdot\ar[d]_{Lf}\ar[r]^h&
      \cdot\ar[d]^g\\
      \cdot\ar[r]_-{k\cdot Rf}\ar@{..>}[ur]|{\phi_{Lf,g}(Lf,k\cdot Rf)}&
      \cdot
      \enddiagram
    }
    =
    \diagram
    \cdot\ar[d]_{Lf}\ar[r]^h&
    \cdot\ar[d]_{Lg}\ar@{=}[r]&
    \cdot\ar[d]^g\\
    \cdot\ar[r]_{K(h,k)}&
    \cdot\ar@{..>}[ur]|{p_g}\ar[r]_-{Rg}&
    \cdot
    \enddiagram
  \end{equation}
  This shows that \eqref{eq:134}~is full on morphisms. That is faithful and full
  on 2-cells, or inequalities, follows from the fact \eqref{eq:134}~commutes
  with the forgetful \Ord-functors into $\C^\two$ and these forgetful
  \Ord-functors are faithful and full on inequalities.
\end{proof}
\begin{cor}
  \label{cor:4}
  For any {\normalfont\textsl{\textsc{lofs}}} $(\mathsf{L},\mathsf{R})$, the
  \Ord-categories $\mathsf{L}\text-\mathrm{Coalg}$ and
  $\mathsf{R}\text-\mathrm{Alg}$ are the object of the arrow part of horizontally
  ordered categories that we denote by $\mathsf{L}\text-\mathbb{C}\mathrm{oalg}$
  and $\mathsf{R}\text-\mathbb{A}\mathrm{lg}$. Furthermore, the respective
  \Ord-functors into $\C^\two$ are the arrow part of horizontally monotone
  double functors into $\Sq(\C)$.
\end{cor}
\begin{proof}
  We use the isomorphism of Theorem~\ref{thm:9} to transfer the structure of a
  horizontally ordered double category from
  $\mathsf{L}\text-\mathrm{Coalg}^{\pitchfork_{\mathkz}}$ to
  $\mathsf{R}\text-\mathrm{Alg}$; see Proposition~\ref{prop:7}. The statement
  about $\mathsf{L}$-coalgebras is dual.
\end{proof}

A straightforward modification of~\cite[Thm.~6]{MR3393453} yields the following
theorem.

\begin{thm}
  \label{thm:11}
  A horizontally monotone double functor
  $\mathbb{U}=(U,U_0)\colon\mathbb{D}\to\Sq(\C)$ is isomorphic over $\Sq(\C)$ to
  $\mathsf{R}\text-\mathbb{A}\mathrm{lg}\to\Sq(\C)$ for a
  {\normalfont\textsl{\textsc{lofs}}} $(\mathsf{L},\mathsf{R})$ if and only if
  \begin{itemize}
  \item $U$ is monadic and the induced \Ord-monad is lax idempotent.
  \item for each vertical arrow $f$ in $\mathbb{D}$ the following square is in
    the image of $U$.
    \begin{equation}
      \diagram
      \cdot\ar[d]_f\ar[r]^f&\cdot\ar[d]^1\\
      \cdot\ar@{=}[r]&\cdot
      \enddiagram
    \end{equation}
  \end{itemize}
\end{thm}

We conclude the section with a result on morphisms of \textsc{lofs}s.
\begin{prop}
  \label{prop:9}
  Suppose that $(\mathsf{L},\mathsf{R})$ and $(\mathsf{L}',\mathsf{R}')$ are
  {\normalfont\textsl{\textsc{lofs}}} on the \Ord-category \C, and
  $\varphi_f\colon Kf\to K'f$ a natural family of morphisms. Then, there is a
  bijection between the following sets, which, moreover, can have at most one element.
  \begin{enumerate}[label=(\alph*)]
  \item \label{item:27} Morphisms of {\normalfont\textsl{\textsc{lofs}}s}
    $(\mathsf{L},\mathsf{R}) \longrightarrow (\mathsf{L}',\mathsf{R}')$.
  \item \label{item:28} Comonad morphisms $\mathsf{L}\to\mathsf{L}'$.
  \item \label{item:29} Monad morphisms $\mathsf{R}\to\mathsf{R'}$.
  \end{enumerate}
\end{prop}
\begin{proof}
  First, there is at most one morphism of the kind in~\ref{item:27},
  \ref{item:28} and~\ref{item:29} by Lemma~\ref{l:16}, Lemma~\ref{l:15} and its
  dual form (ie, the version for comonads). Clearly, if there is a morphism as
  in~\ref{item:27}, then there are morphisms as in~\ref{item:28}
  and~\ref{item:29}, just by definition of morphism of \textsc{awfs}s (\S
  \ref{sec:categ-norm}).

  Suppose there is a morphism of comonads $Q$ from $\mathsf{L}$ to $\mathsf{L}'$,
  with components $Q_f\colon Lf\to L'f$. Due to the counit axiom, $(1,R'f)\cdot
  Q_f=(1,Rf)$, we have that $Q_f$ is of the form $(1,\varphi_f)$ for a morphism
  $\varphi_f\colon Kf\to K'f$. Let
  \begin{equation}
    Q_*\colon
    \mathsf{L}\text-\mathrm{Coalg}
    \longrightarrow
    \mathsf{L}'\text-\mathrm{Coalg}
  \end{equation}
  be the \Ord-functor induced by $Q$; it commutes with the forgetful
  \Ord-functors into $\C^\two$. Applying the
  functor $(-)^{\pitchfork_{\mathkz}}$ to $Q_*$ and employing the
  isomorphisms~\eqref{eq:134} (Theorem~\ref{thm:9}) we obtain an \Ord-functor,
  depicted by a dashed arrow.
  \begin{equation}
    \diagram
    \mathsf{R}'\text-\mathrm{Alg}\ar[d]_\cong \ar@{-->}[r]&
    \mathsf{R}\text-\mathrm{Alg}\ar[d]^\cong\\
    \mathsf{L}'\text-\mathrm{Coalg}^{\pitchfork_{\mathkz}}
    \ar[r]^-{Q_*^{\pitchfork_{\mathkz}}}&
    \mathsf{L}\text-\mathrm{Coalg}^{\pitchfork_{\mathkz}}
    \enddiagram
  \end{equation}
  The vertical isomorphisms were described in the proof of Theorem~\ref{thm:9},
  and this description can be used to describe the dashed arrow. If $(p,1)\colon
  R'f\to f$ is an $\mathsf{R}'$-algebra structure, the associated \kz-lifting
  operation $\phi_{-,f}$ defines a diagonal filler for each commutative square
  \begin{equation}
    \xymatrixcolsep{1.5cm}
    \diagram
    \cdot\ar[d]_\ell\ar[r]^h&\cdot\ar[d]^f\\
    \cdot\ar[r]_k\ar@{..>}[ur]|{\phi_{\ell,f}(h,k)}&\cdot
    \enddiagram
    \qquad
    \phi_{\ell,f}(h,k)=p\cdot K(h,k)\cdot s
  \end{equation}
  for any $\mathsf{L}'$-coalgebra $(1,s)\colon \ell\to L\ell$. Uppon applying
  $Q_*^{\pitchfork_{\mathkz}}$ we obtain a \kz-lifting operation $\psi_{-,f}$ of $f$
  against all $\mathsf{L}$-coalgebras. If $(1,t)\colon g\to Lg$ is an
  $\mathsf{L}$-coalgebra, its image under $Q_*$ is
  \begin{equation}
    g\xrightarrow{(1,t)} Lg\xrightarrow{(1,\varphi_g)} L'g
  \end{equation}
  and therefore $\psi_{g,f}(h,k)$ is the form
  \begin{equation}
    \psi_{g,f}(h,k)=
    \phi_{Q_*g,f}(h,k)=
    p\cdot K'(h,k)\cdot \varphi_g\cdot t =
    p\cdot \varphi_f\cdot K(h,k)\cdot t.
  \end{equation}
  We now obtain the $\mathsf{R}$-algebra structure on $f$ by $\psi_{Lf,f}(1,Rf)$,
  \begin{equation}
    \psi_{Lf,f}(1,Rf)=p\cdot\varphi_{f}\cdot K(1,Rf)\cdot \sigma_f=
    p\cdot\varphi_{f}.
  \end{equation}
  In conclusion, the dashed arrow in page~\pageref{prop:9} represents the
  \Ord-functor that sends an $\mathsf{R}'$-algebra $(p,1)\colon R'f\to f$ to the
  $\mathsf{R}$-algebra $(p\cdot \varphi_f,1)\colon Rf\to f$. This implies that
  $(\varphi_f,1)\colon Rf\to R'f$ is a monad morphism, and the set~\ref{item:29}
  is non-empty.

  We have seen that \ref{item:29}~has a member if \ref{item:28}~has a member. By
  a duality argument, ie by taking the opposite \Ord-category of $\C$, we
  deduce the converse: \ref{item:28}~has a member if
  \ref{item:29}~does. Furthermore, from the construction of the previous
  paragraph, we know that if $(1,\varphi_f)\colon Lf\to L'f$ is a comonad
  morphism, then the monad morphism must be of the form $(\varphi_f,1)\colon
  Rf\to R'f$, and vice versa. Therefore, the existence of a comonad morphism
  $\mathsf{L}\to\mathsf{L}'$, or the existence of a monad morphism
  $\mathsf{R}\to \mathsf{R}'$, are equivalent to the existence of a unique
  $\varphi_f\colon Kf\to K'f$ such that $(1,\varphi_f)\colon Lf\to L'f$ is a
  comonad morphism and $(\varphi_f,1)\colon Rf\to R'f$ is a monad morphism. In
  other words, equivalent to the existence of a unique morphism of
  \textsc{awfs}s $(\mathsf{L},\mathsf{R})\to(\mathsf{L}',\mathsf{R}')$.
\end{proof}

The above proposition is a reminder of the differences that exist between
general \textsc{awfs}s and those enriched over \Ord. In the general case, the
proposition does not hold; see~\cite[Lemma~6.9]{MR2781914}
or~\cite[Prop.~2]{MR3393453}.

\section{The definition of LOFS revisited}
\label{sec:defin-lofs-revis}

Lax orthogonal factorisation systems on \Ord-categories were defined in \S
\ref{sec:lax-orthogonal-awfss} as \Ord-enriched \textsc{awfs}s
$(\mathsf{L},\mathsf{R})$ whose comonad $\mathsf{L}$ is lax idempotent, or
equivalently, by Proposition~\ref{prop:1}, whose monad $\mathsf{R}$ is lax
idempotent. The definition of \textsc{awfs} includes a mixed distributive law
$\Delta\colon LR\Rightarrow RL$, with components
$(\sigma_f,\pi_f)\colon LRf\to RLf$. The axioms of a mixed distributive law in
this case amount to the commutativity of the diagrams in~\eqref{eq:30}, and they
are equivalent, as mentioned in Remark~\ref{rmk:1}, to the requirement that the
diagonal filler of the square below be $\sigma_f\cdot\pi_f$.
\begin{equation}
  \label{eq:109}
  \diagram
  Kf\ar[d]_{LRf}\ar[r]^-{\sigma_f}&
  KLf\ar[d]^{RLf}\\
  KRf\ar[r]_-{\pi_f}\ar@{..>}[ur]|{\sigma_f\cdot \pi_f}&
  Kf
  \enddiagram
\end{equation}
The main result of the section is the following.

\begin{thm}
  \label{thm:12}
  In the definition of {\normalfont\textsl{\textsc{lofs}}}, the distributive law
  axiom is redundant. More precisely, the following suffices to define a
  {\normalfont\textsl{\textsc{lofs}:}} a domain-preserving \Ord-comonad
  $\mathsf{L}$ and a codomain-preserving monad $\mathsf{R}$ on $\C^\two$ that
  define the same \Ord-functorial factorisation $f=Rf\cdot Lf$; both
  $\mathsf{L}$ and $\mathsf{R}$ should be lax idempotent.
\end{thm}
\begin{proof}
  All we need to show is that $\sigma_f\cdot\pi_f$ is the diagonal filler of the
  square~\eqref{eq:109}. The existence of a \kz-lifting operations for
  $\mathsf{R}$-algebras against $\mathsf{L}$-coalgebras does not depend on the
  distributivity axiom but it only suffices that both $\mathsf{L}$ and
  $\mathsf{R}$ be lax idempotent. Then, we only need to show that
  \begin{equation}
    \label{eq:110}
    \sigma_f\cdot\pi_f\leq d
  \end{equation}
  for the \kz-diagonal filler $d$ of the square~\eqref{eq:109}, for, in this
  case, the inequality is necessarily an equality. There are adjunctions
  $\sigma_f\dashv K(1,Rf)$ and $K(Lf,1)\dashv \pi_f$ since $\mathsf{L}$ and
  $\mathsf{R}$ are lax idempotent. Thus, the inequality~\eqref{eq:110} is equivalent to
  $1\leq K(1,Rf)\cdot d\cdot K(Lf,1)$, due to the inequalities~\eqref{eq:57} of
  \S \ref{sec:lax-orthogonal-awfss}. Consider the following diagram, where
  $(Lf,K(Lf,1))=L(Lf,1)$ is a morphism of $\mathsf{L}$-coalgebras and
  $(K(1,Rf),Rf)=R(1,Rf)$ is a morphism of $\mathsf{R}$-algebras.
  \begin{equation}
    \xymatrixcolsep{1.2cm}
    \diagram
    \cdot\ar[r]^-{Lf}\ar[d]_{Lf}&
    \cdot\ar[r]^-{\sigma_f}\ar[d]_{LRf}&
    \cdot\ar[d]^{RLf}\ar[r]^-{K(1,Rf)}&
    \cdot\ar[d]^{Rf}
    \\
    \cdot\ar[r]_-{K(Lf,1)}&
    \cdot\ar[r]_-{\pi_f}\ar[ur]^d&
    \cdot\ar[r]_-{Rf}&
    \cdot
    \enddiagram
    \qquad
    \diagram
    \cdot\ar[r]^{Lf}\ar[d]_{Lf}&\cdot\ar[d]^{Rf}\\
    \cdot\ar[r]_{Rf}\ar@{..>}[ur]^1&\cdot
    \enddiagram
  \end{equation}
  By the naturality of the diagonal fillers with respect to morphisms of
  $\mathsf{L}$-coalgebras and morphism of $\mathsf{R}$-algebras, we deduce that
  $K(1,Rf)\cdot d\cdot K(Lf,1)$ is the diagonal filler of the square on the
  right hand side, and hence equal to the identity morphism (see
  Lemma~\ref{l:2}). Therefore the inequality~\eqref{eq:110} holds, completing
  the proof.
\end{proof}

We can summarise the theorem above and Proposition~\ref{prop:1} in the following
way: given a domain-preserving \Ord-comonad $\mathsf{L}$ and a
codomain-preserving \Ord-monad $\mathsf{R}$ on $\C^\two$ that induce the same
\Ord-functorial factorisation $f=Rf\cdot Lf$, the following two statements are
equivalent, and when they hold we are in the presence of a \textsc{lofs}.
\begin{itemize}
\item One of $\mathsf{L},\mathsf{R}$ is lax idempotent and the distributive law
  axiom holds.
\item Both $\mathsf{L}$ and $\mathsf{R}$ are lax idempotent.
\end{itemize}

\section{Embeddings with respect to a monad}
\label{sec:embedd-with-resp}

Embeddings with respect to a lax idempotent monad were extensively exploited
in~\cite{MR1641443,MR1718976} and in~\cite{MR1718926}, where topological
embeddings were exhibited as an example (more on this in \S
\ref{sec:filter-monads}). In this section we begin our analysis of the interplay
between these embeddings and \textsc{lofs}s.
\begin{df}
  \label{df:13}
  If $S\colon\C\to \mathcal{B}$ is a locally monotone functor between
  \Ord-categories, an \emph{$S$-embedding} structure on a morphism $f$ in \C\
  is a \textsc{lari} structure in $Sf$ in $\mathcal{B}$. Recall that
  \textsc{lari} structures on a morphism in an \Ord-category are unique, which
  one usually rephrases by saying that being a \textsc{lari} is a
  \emph{property} of a morphism. Therefore, being an $S$-embedding in an \Ord-category
  is a property of morphisms.

  The \Ord-category of $S$-embeddings, denoted by $S\text-\mathrm{Emb}$, is the
  category whose objects are pairs $(f,r)$ where $f$ is a morphism in $\C$ and
  $Sf\dashv r$ is a \textsc{lari} in $\mathcal{B}$. A morphism
  $(f,r)\to(g,t)$ in this category is a morphism $(h,k)\colon f\to g$ in
  $\C^\two$ satisfying $Sh\cdot r=t\cdot Sk$. There is an obvious forgetful
  functor $S\text-\mathrm{Emb}\to\C^\two$ given on objects by $(f,r)\mapsto
  f$. We make $S\text-\mathrm{Emb}$ into an \Ord-category by declaring
  $(h,k)\leq(h',k')$ if this inequality holds in $\C^\two$; this makes the
  forgetful functor $U$ into a locally monotone functor that fits in a pullback
  square.
  \begin{equation}
    \label{eq:141}
    \diagram
    {S\text-\mathrm{Emb}}
    \ar[r]^-{}\ar[d]_{U}
    \ar@{}[dr]|{\mathrm{pb}}
    &
    {\lari(\mathcal{B})}
    \ar[d]^{}
    \\
    {\C^\two}
    \ar[r]_-{S^\two}
    &
    {\mathcal{B}^\two}
    \enddiagram
  \end{equation}
\end{df}
\begin{lemma}
  \label{l:17}
  $S$-embeddings in \C\ are the vertical morphisms of a horizontally ordered
  double category, with objects those of $\C$, horizontal morphisms the
  morphisms of \C\ and squares those commutative squares in \C\ that represent
  morphisms of $S$-embeddings.
  Furthermore, the pullback diagram displayed above is part of a pullback
  diagram of horizontally ordered double categories.
  \begin{equation}
    \label{eq:67}
    \diagram
    {S\text-\mathbb{E}\mathrm{mb}}
    \ar[r]^-{}\ar[d]_{U}
    \ar@{}[dr]|{\mathrm{pb}}
    &
    {\mathbb{L}\mathbf{ari}(\mathcal{B})}
    \ar[d]^{}
    \\
    {\Sq(\C)}
    \ar[r]_-{\Sq(S)}
    &
    {\Sq(\mathcal{B})}
    \enddiagram
  \end{equation}
\end{lemma}
\begin{proof}
  At the level of \Ord-categories of objects, the square of the statement has
  identity vertical arrows and
  $\operatorname{ob}S\colon\operatorname{ob}\mathcal{C}\to\operatorname{ob}\mathcal{D}$
  as horizontal arrows. Hence, it is a pullback at the level of \Ord-categories
  of objects. At the level of \Ord-categories of arrows, the square is precisely
  the pullback square~\eqref{eq:141}. Therefore,
  $S\text-\mathrm{Emb}\rightrightarrows\mathcal{C}$ has a unique internal
  category structure that makes~\eqref{eq:67} a pullback square of internal
  categories.
\end{proof}
\begin{lemma}
  \label{l:11}
  The forgetful \Ord-functor $S\text-\mathrm{Emb}\to\C^\two$ creates colimits,
  provided that $\C$ has and $S$ preserves colimits.
\end{lemma}
\begin{proof}
  In the pullback diagram~\eqref{eq:141}, the leftmost vertical \Ord-functor
  creates any colimit that is preserved by $S$ (and thus by $S^\two$), since the
  rightmost vertical \Ord-functor creates colimits.
\end{proof}

\begin{df}
  \label{df:19}
  If $\mathsf{T}$ is an \Ord-monad on \C, we shall call
  $F^{\mathsf{T}}$-embeddings \emph{$\mathsf{T}$-embeddings}, and denote the
  \Ord-category $F^{\mathsf{T}}\text-\mathrm{Emb}$ by
  $\mathsf{T}\text-\mathrm{Emb}$.
\end{df}
\begin{lemma}
  \label{l:8}
  Let $\mathsf{T}$ be an \Ord-monad on \C\ and
  $F^{\mathsf{T}}\dashv V^{\mathsf{T}}\colon\mathsf{T}\text-\mathrm{Alg}\to\C$
  the associated Eilenberg-Moore adjunction. If $V^\mathsf{T}$ is locally full, ie if $V^\mathsf{T}f\leq V^\mathsf{T}g$ implies
  $f\leq g$ for parallel morphism of algebras $f$ and $g$, then
  $\mathsf{T}$-embeddings coincide with $T$-embeddings.
\end{lemma}
For example, the above lemma applies when $\mathsf{T}$ is lax idempotent.

\begin{prop}
  \label{prop:3}
  Let $\mathsf{T}$ be a lax idempotent monad on an \Ord-category with a terminal
  object. The obvious \Ord-functor
  \begin{equation}
    \label{eq:97}
    \mathsf{T}\text-\mathrm{Emb}\longrightarrow
    \prescript{\pitchfork_{\mathkz}}{}{\bigl(\mathsf{T\text-\mathrm{Alg}}/1\bigr)}
  \end{equation}
  is an isomorphism.
\end{prop}
\begin{proof}
  We define the \Ord-functor~\eqref{eq:97} and show that it is bijective on
  objects at the same time by showing that a morphism $f$ of \C\ is a
  $\mathsf{T}$-embedding if and only if it has a right \kz-lifting operation
  against morphisms $A\to 1$ for all $\mathsf{T}$-algebras $A$.

  The forgetful
  \Ord-functor $V\colon \mathsf{T}\text-\mathrm{Alg}\to\C$ can be composed with the
  inclusion $\C\to\C^\two$ that sends $X$ to $(X\to 1)$, and then consider the
  $\prescript{\pitchfork_{\mathkz}}{}{(-)}$ of the resulting functor into $\C^\two$. An
  object of $
  \prescript{\pitchfork_{\mathkz}}{}{(\mathsf{T\text-\mathrm{Alg}}/1)}$ is a
  morphism $f\colon X\to Y$ of $\C$ with a \textsc{rali} structure on
  \begin{equation}
    \label{eq:115}
    \C(Y,V(A))=\C^\two(f,V1_A)\longrightarrow \C^\two(f,(VA\to 1))=\C(X,VA)
  \end{equation}
  In other words, each morphism $X\to A$ can be extended along $f$ and this
  extension is a left Kan extension.
  \begin{equation}
    \diagram
    X\ar[d]_f\ar[r]&A\\
    Y\ar@{..>}[ur]&
    \enddiagram
  \end{equation}
  The morphism \eqref{eq:115}~can be written as
  \begin{equation}
    \label{eq:139}
    \C(Y,V(A))\cong
    \mathsf{T}\text-\mathrm{Alg}(F^{\mathsf{T}}Y,A)
    \xrightarrow{\mathsf{T}\text-\mathrm{Alg}(F^{\mathsf{T}}f,1)}
    \mathsf{T}\text-\mathrm{Alg}(F^{\mathsf{T}}X,A)
    \cong
    \C(X,V(A))
  \end{equation}
  which has a \textsc{rali} structure, for all $\mathsf{T}$-algebras $A$, if and
  only if $F^{\mathsf{T}}f$ has a \textsc{lari} structure. This defines a
  bijection between the objects of the domain and codomain of~\eqref{eq:97}.

  It remains to define~\eqref{eq:97} on morphisms and to verify that it is
  bijective on these morphisms, and locally full on inequalities. Suppose that
  $f$ and $g$ are $\mathsf{T}$-embeddings. A morphism
  $(h,k)\colon f\to g$ is a morphism in the codomain of~\eqref{eq:97} if it is
  compatible with the \textsc{rali} structures on the morphisms~\eqref{eq:115}
  corresponding to $f$ and $g$; in other words, if $(h,k)$ induces a morphism of
  \textsc{rali}s. This is equivalent to requiring that $(h,k)$ should induce a
  morphism of \textsc{rali}s between the \textsc{rali}s~\eqref{eq:139} that
  correspond to $f$ and $g$. By Yoneda lemma, this means that $(h,k)$ is a
  morphism of $\mathsf{T}$-embeddings. This defines a functor~\eqref{eq:97} that
  is bijective on morphisms.

  It remains to show that \eqref{eq:97}~is locally full on morphisms, but this
  is easy and left to the reader.
\end{proof}

\begin{prop}
  \label{prop:6}
  Let $\mathsf{T}$ be a lax idempotent monad on an \Ord-category with a terminal
  object. The obvious \Ord-functor
  \begin{equation}
    \label{eq:128}
    \mathsf{T}\text-\mathrm{Alg}/1\longrightarrow
    (\mathsf{T}\text-\mathrm{Emb})^{\pitchfork_{\mathkz}}_1
    \subset
    (\mathsf{T}\text-\mathrm{Emb})^{\pitchfork_{\mathkz}}
  \end{equation}
  is an isomorphism between $\mathsf{T}\text-\mathrm{Alg}$ and the fiber of
  $\cod\colon
  (\mathsf{T}\text-\mathrm{Emb})^{\pitchfork_{\mathkz}}\longrightarrow\C$ over
  $1$.
\end{prop}
\begin{proof}
  We will show that a morphism $A\to 1$ is in
  $(\mathsf{T}\text-\mathrm{Emb})^{\pitchfork_{\mathkz}}$ if and only if $A$ is
  a $\mathsf{T}$-algebra.

  The components $\eta_X\colon X\to TX$ of the unit of the monad $\mathsf{T}$
  are $\mathsf{T}$-embeddings due to the adjunction
  $T\eta_X\dashv\mu_X$. Furthermore, for any morphism $u\colon X\to Y$, there is
  a morphism $(u,Tu)\colon \eta_X\to\eta_Y$ in $\mathsf{T}\text-\mathrm{Emb}$
  because $Tu\cdot \mu_X=\mu_Y\cdot T^2u$.

  Suppose that $A\to 1$ has a \kz-lifting operation against
  $\mathsf{T}$-embeddings, which provides a diagonal filler to the square
  displayed below.
  \begin{equation}
    \diagram
    A\ar[d]_{\eta_A}\ar@{=}[r]&A\ar[d]\\
    TA\ar[r]\ar@{..>}[ur]^a&1
    \enddiagram
  \end{equation}
  We will show that $a$ is a $\mathsf{T}$-algebra structure.

  It is not hard to verify that the diagonal filler of the square
  \begin{equation}
    \diagram
    {A}
    \ar[r]^-{\eta_A}\ar[d]_{\eta_A}
    &
    {TA}
    \ar[d]^{}
    \\
    {TA}
    \ar[r]_-{}
    &
    {1}
    \enddiagram
  \end{equation}
  is the identity morphism, where $TA$ is equipped with the \kz-lifting
  operation induced by its free $\mathsf{T}$-algebra structure. On the other
  hand, $\eta_A\cdot a$ is another diagonal filler, so there is a inequality
  $1_{TA}\leq \eta_A\cdot a$. Thus, $a\dashv \eta_A$ which is equivalent to
  saying that $a$ is a $\mathsf{T}$-algebra structure on $A$.

  We leave to the reader the verification that the \Ord-functor of the statement
  if full and faithful.
\end{proof}

\begin{cor}
  \label{cor:6}
  In the conditions of Proposition~\ref{prop:6},
  the unit of the component at $\mathsf{T}\text-\mathrm{Emb}$ of the adjunction
  of Theorem~\ref{thm:5}
  \begin{equation}
    \label{eq:140}
    \mathsf{T}\text-\mathrm{Emb}\longrightarrow
    \prescript{\pitchfork_{\mathkz}}{}{\bigl(\mathsf{T}\text-\mathrm{Emb}^{\pitchfork_{\mathkz}}\bigr)}
  \end{equation}
  is an isomorphism.
\end{cor}
\begin{proof}
  Continuing with the notation used in Proposition~\ref{prop:6},
  the inclusion of $\mathsf{T}\text-\mathrm{Emb}^{\pitchfork_{\mathkz}}_1$ into
  $\mathsf{T}\text-\mathrm{Emb}^{\pitchfork_{\mathkz}}$ induces an \Ord- functor in
  the opposite direction
  \begin{equation}
    \prescript{\pitchfork_{\mathkz}}{}{\bigl(\mathsf{T}\text-\mathrm{Emb}^{\pitchfork_{\mathkz}}\bigr)}
    \longrightarrow
    \prescript{\pitchfork_{\mathkz}}{}{\bigl(\mathsf{T}\text-\mathrm{Emb}^{\pitchfork_{\mathkz}}_1\bigr)}.
  \end{equation}
  We can form a morphism from right to left, displayed below, where the two
  isomorphisms are those given by the Propositions~\ref{prop:3} and~\ref{prop:6}.
  \begin{equation}
    \mathsf{T}\text-\mathrm{Emb}\xrightarrow{\cong}
    \prescript{\pitchfork_{\mathkz}}{}{\bigl(\mathsf{T}\text-\mathrm{Alg}/1\bigr)}
    \xrightarrow{\cong}
    \prescript{\pitchfork_{\mathkz}}{}{\bigl(\mathsf{T}\text-\mathrm{Emb}^{\pitchfork_{\mathkz}}_1\bigr)}
    \longleftarrow
    \prescript{\pitchfork_{\mathkz}}{}{\bigl(\mathsf{T}\text-\mathrm{Emb}^{\pitchfork_{\mathkz}}\bigr)}
  \end{equation}
  The resulting \Ord-functor
  \begin{equation}
    \label{eq:142}
    \prescript{\pitchfork_{\mathkz}}{}{\bigl(\mathsf{T}\text-\mathrm{Emb}^{\pitchfork_{\mathkz}}\bigr)}\longrightarrow\mathsf{T}\text-\mathrm{Emb}
  \end{equation}
commutes with the forgetful \Ord-functors into $\C^\two$. Since these forgetful
functors are injective on objects and on morphisms, and full on inequalities
between morphisms, we deduce that~\eqref{eq:142} is necessarily an inverse for
the component of the unit of the statement.
\end{proof}

\begin{cor}
  \label{cor:3}
  If $(\mathsf{L},\mathsf{R})$ is a {\normalfont\textsl{\textsc{lofs}}} on an
  \Ord-category with a terminal object, then there is a canonical \Ord-functor
  \begin{equation}
    \label{eq:130}
    \mathsf{L}\text-\mathrm{Coalg}\longrightarrow \mathsf{R}_1\text-\mathrm{Emb}
  \end{equation}
  where $\mathsf{R}_1$ is the \Ord-monad on $\C\cong \C/1$ that is the
  restriction of $\mathsf{R}$. 
\end{cor}
\begin{proof}
  The inclusion of
  $\mathsf{R}_1\text-\mathrm{Alg}\hookrightarrow\mathsf{R}\text-\mathrm{Alg}$,
  given by $A\mapsto(A\to 1)$, induces the unlabelled arrow in the following
  string of \Ord-functors over $\C^\two$,
  \begin{equation}
    \mathsf{L}\text-\mathrm{Coalg}\cong
    \prescript{\pitchfork_{\mathkz}}{}{\bigl(\mathsf{R}\text-\mathrm{Alg}\bigr)}
    \longrightarrow
    \prescript{\pitchfork_{\mathkz}}{}{\bigl(\mathsf{R}_1\text-\mathrm{Alg}/1\bigr)}
    \cong
    \mathsf{R}_1\text-\mathrm{Emb}
  \end{equation}
  where the last isomorphism is provided by Proposition~\ref{prop:3}.
\end{proof}
The \Ord-functor of Corollary~\ref{cor:3} may be described more explicitly. If
$f\colon X\to Y$ is an $\mathsf{L}$-coalgebra, then the corresponding
$\mathsf{R}_1$-embedding structure is given by the adjunction $R_1f\dashv
r\colon R_1Y\to R_1X$ where $r$ is the unique morphism of
$\mathsf{R}_1$-algebras that composed with the unit $\eta_Y\colon Y\to R_1Y$
equals the \kz-lifting corresponding to the square displayed below.
\begin{equation}
  \diagram
  {X}
  \ar[r]^-{\eta_X}\ar[d]_{f}
  &
  {R_1A}
  \ar[d]^{!=R(!)}
  \\
  {Y}\ar@{..>}[ur]|{r\cdot\eta_Y}
  \ar[r]_-{!}
  &
  {1}
  \enddiagram
\end{equation}

\section{KZ-reflective LOFSs}
\label{sec:reflective-lofss}
We begin by summarising the most basic definitions of~\cite{MR779198} around
reflective factorisation systems.

An \textsc{ofs} $(\mathscr{E},\mathscr{M})$ (or even a pre-factorisation system,
which is similar to a \textsc{ofs} but without the requirement that each
morphism should be a composition of one in $\mathscr{E}$ followed by one in
$\mathscr{M}$) on a category with a terminal object $\mathcal{C}$, induces a
reflective subcategory of $\mathcal{C}$ formed by those objects $X$ for which $X\to
1$ belongs to $\mathscr{M}$. In the other direction, each reflective subcategory
$\mathcal{B}\subseteq\mathcal{C}$ induces a pre-factorisation system
$(\mathscr{E},\mathscr{M})$ whose $\mathscr{E}$ is formed by all the morphisms
that are orthogonal to each object of $\mathcal{B}$. With an obvious ordering on
reflective subcategories and pre-factorisation systems, these two constructions
form an adjunction (a Galois correspondence). Those pre-factorisation systems
obtained from reflective subcategories are called \emph{reflective}, and are
characterised as those for which $g\cdot f\in \mathscr{E}$ and $g\in\mathscr{E}$
implies $f\in\mathscr{E}$.

In this section we consider the analogous notion of \kz-reflective \textsc{lofs}
and find a characterisation that mirrors the case of \textsc{ofs}s.

\begin{df}
  \label{df:20}
  We say that the $\Ord$-monad $\mathsf{T}$ on $\C$ is \emph{fibrantly
    \slkz-generating} if the forgetful \Ord-functor
  $\mathsf{T}\text-\mathrm{Emb}\to\C^\two$ has a right adjoint (in the
  \Ord-enriched sense).
\end{df}
\begin{prop}
  \label{prop:4}
  Assume that \C\ is a cocomplete and finitely complete \Ord-catego\-ry. Then
  $\mathsf{T}$ is fibrantly \slkz-generating if and only if there exists an
  \Ord-enriched {\normalfont\textsl{\textsc{awfs}}} $(\mathsf{L},\mathsf{R})$
  for which $\mathsf{L}\text-\mathrm{Coalg}\cong\mathsf{T}\text-\mathrm{Emb}$
  over $\C^\two$. Furthermore, this {\normalfont\textsl{\textsc{awfs}}} is lax orthogonal.
\end{prop}
\begin{proof}
  The implication in one direction is clear; indeed, if
  $\mathsf{T}\text-\mathrm{Emb}$ is isomorphic over $\C^\two$ to
  $\mathsf{L}\text-\mathrm{Coalg}$ then the condition of Definition~\ref{df:20}
  holds.

  Assume that $\mathsf{T}$ is fibrantly \kz-generating.
  The forgetful \Ord-functor
  $\lari(\C)\to \mathsf{T}\text-\mathrm{Alg}^\two$ is comonadic by
  Lemma~\ref{l:5}. The \Ord-functor $\mathsf{T}\text-\mathrm{Emb}\to\C^\two$ is
  a pullback of the comonadic \Ord-functor mentioned, therefore, it satisfies
  all the hypotheses of (the \Ord-enriched version) of Beck's comonadicity
  theorem, except perhaps for the hypothesis of being a left adjoint. Together with
  Definition~\ref{df:20}, we deduce that $\mathsf{T}\text-\mathrm{Emb}$ is
  comonadic over $\C^\two$.

  The \Ord-category of $\mathsf{T}$-embeddings forms part of a horizontally
  ordered double category $\mathsf{T}\text-\mathbb{E}\mathrm{mb}$, as in
  Lemma~\ref{l:17}. We will be able to apply the dual of Theorem~\ref{thm:11} if
  we show the following: if $f$ is a $\mathsf{T}$-embedding, then the square on
  the left is a morphism of $\mathsf{T}$-embeddings $1\to f$. This is equivalent
  to saying that the square on the right is a morphism of \textsc{lari}s
  $1\to F^Tf$, which is easily seen to hold.
  \begin{equation}
    \label{eq:69}
    \diagram
    \cdot\ar[d]_1\ar@{=}[r]&\cdot\ar[d]^f\\
    \cdot\ar[r]^{f}&\cdot
    \enddiagram
    \qquad
    \diagram
    \cdot\ar[d]_1\ar@{=}[r]&\cdot\ar[d]^{F^Tf}\\
    \cdot\ar[r]^{F^Tf}&\cdot
    \enddiagram
  \end{equation}
  We deduce, by a dual form of Theorem~\ref{thm:11}, that
  $\mathsf{T}\text-\mathrm{Emb}$ is $\mathsf{L}\text-\mathrm{Coalg}$ for an
  \textsc{awfs} $(\mathsf{L},\mathsf{R})$.

  It remains to show that this
  \textsc{awfs} is a \textsc{lofs}, for which we appeal to the dual version of
  \cite[Cor.~6.9]{LF:kz-generated-lofs}, which we explain here without
  proof. By definition of $\mathsf{T}\text-\mathrm{Emb}$, there is a pullback
  diagram
  \begin{equation}
    \label{eq:153}
    \diagram
    \mathsf{T}\text-\mathrm{Emb}\ar[r]\ar[d]_U&
    \mathsf{E}\text-\mathrm{Coalg}\ar[d]\\
    \C^\two\ar[r]^-{(F^T)^\two}&
    \mathsf{T}\text-\mathrm{Alg}^\two
    \enddiagram
  \end{equation}
  where \C\ is cocomplete and the free algebra \Ord-functor $F^T$ is a left
  adjoint. The comonad $\mathsf{E}$ on $\mathsf{T}\text-\mathrm{Alg}^\two$ is
  the one of \S \ref{sec:laris-awfss} and exists since \C, and thus
  $\mathsf{T}\text-\mathrm{Alg}$, has finite limits.
  We are in the dual conditions of Corollaries~6.9 and 6.10
  of~\cite{LF:kz-generated-lofs}, which guarantees that the comonad
  corresponding to the comonadic $U$ is lax idempotent.
\end{proof}

\begin{df}
  \label{df:22}
  The \Ord-category of lax idempotent monads on the \Ord-catego\-ry \C, denoted by
  $\mathbf{LIMnd}(\C)$, has morphisms $\mathsf{T}\to\mathsf{S}$ natural
  transformations that are compatible with the multiplication and unit of the
  monads, in the usual manner.

  We will denote by $\mathbf{LIMnd}_{\mathrm{fib}}(\C)$ the full sub-\Ord-category of
  $\mathbf{LIMnd}(\C)$ consisting of those monads that are fibrantly \kz-generating,
  in the sense of Definition~\ref{df:20}.
\end{df}

When \C\ is cocomplete and finitely complete, we have a situation that can be summarised by the following diagram of
\Ord-functors.
\begin{equation}
  \label{eq:117}
  \diagram
  \mathbf{LOFS}(\C)\ar@{^(->}[d]_{(-)\text-\mathrm{Coalg}}\ar[dr]^{\tilde\Phi}&
  \mathbf{LIMnd}_{\mathrm{fib}}(\C)\ar@{^(->}[d]^{I}\ar@{..>}[l]_{\Psi}\\
  \Ord\text-\Cat/\C^\two &
  \mathbf{LIMnd}(\C)\ar[l]_-{\tilde\Psi}
  \enddiagram
\end{equation}
The vertical \Ord-functors are full and faithful, the one on the right being
just an inclusion. The one on the left sends each lax orthogonal \textsc{awfs}
on $\C$ to the \Ord-category $\mathsf{L}\text-\mathrm{Coalg}$ over $\C^\two$.
The \Ord-functor $\tilde\Psi$ sends a lax idempotent monad $\mathsf{T}$ on \C\
to the category $\prescript{\pitchfork_{\kz}}{}{(\mathsf{T}\text-\mathrm{Alg}/1)}$
over $\C^\two$, and has a lifting to an \Ord-functor $\Psi$ that sends a
fibrantly \kz-generating $\mathsf{T}$ to the \textsc{lofs} $(\mathsf{L},\mathsf{R})$
on \C\ that satisfies
$\mathsf{L}\text-\mathrm{Coalg}\cong\mathsf{T}\text-\mathrm{Emb}$ --~see
Proposition~\ref{prop:4}. Finally, $\tilde\Phi$ sends $(\mathsf{L},\mathsf{R})$
to $\mathsf{R}_1$, the restriction of $\mathsf{R}$ to the slice $\C/1\cong\C$.

It will be convenient to use the following relaxed notion of adjunction.
Suppose given a diagram of functors and a
natural transformation, that may be enriched as needed, as displayed.
\begin{equation}
  \label{eq:118}
  \diagram
  \mathcal{A}\ar@/_/[dr]_F\dtwocell<\omit>{^<-5>\theta}&
  \mathcal{B}\ar[d]^I\ar[l]_G\\
  &\mathcal{D}
  \enddiagram
\end{equation}
\begin{df}
\label{df:14}
Following~\cite[\S 2]{MR0222138}, we say that \emph{$\theta$ exhibits $G$ as a
  $I$-right adjoint of $F$}, and \emph{$F$ as a $I$-left adjoint of $G$}
denoted by $F\dashv_IG$, if
\begin{equation}
  \label{eq:119}
  \mathcal{A}(A,G(B))\xrightarrow{F}\mathcal{D}(F(A),FG(B))
  \xrightarrow{\mathcal{D}(1,\theta_B)} \mathcal{D}(F(A),I(B))
\end{equation}
is invertible. 
\end{df}

It is easy to prove that if $I\colon\mathcal{B}\to\mathcal{D}$ is fully
faithful and $\theta$ is an isomorphism, then $G$ is fully faithful.

\begin{thm}
  \label{thm:4}
  In the situation of the diagram~\eqref{eq:117}, the \Ord-functor $\tilde\Phi$
  is a $I$-left adjoint of $\Psi$. Moreover, $\Psi$ is fully faithful.
\end{thm}
\begin{proof}
  We have to exhibit a natural bijection
  \begin{equation}
    \label{eq:120}
    \mathbf{LIMnd}(\C)(\mathsf{R}_1,\mathsf{T})
    \cong
    \mathbf{LOFS}(\C)((\mathsf{L},\mathsf{R}),\Psi(\mathsf{T}))
  \end{equation}
  using our knowledge of the existence of natural isomorphisms
  \begin{equation}
    \label{eq:131}
    \mathbf{LIMnd}(\C)(\mathsf{R}_1,\mathsf{T})
    =
    \mathbf{Mnd}(\C)(\mathsf{R}_1,\mathsf{T})
    \cong
    \Ord\text-\Cat/\C
    \bigl(
    \mathsf{T}\text-\mathrm{Alg}, \mathsf{R}_1\text-\mathrm{Alg}
    \bigr)
  \end{equation}
  \begin{equation}
    \label{eq:129}
    \mathbf{LOFS}((\mathsf{L},\mathsf{R}),\Psi(\mathsf{T}))
    =
    \mathbf{AWFS}((\mathsf{L},\mathsf{R}),\Psi(\mathsf{T}))
    \cong
    \bigl(
    \Ord\text-\Cat/\C^\two
    \bigr)
    \bigl(
    \mathsf{L}\text-\mathrm{Coalg},
    \mathsf{T}\text-\mathrm{Emb}
    \bigr).
  \end{equation}

  Suppose that
  $H\colon\mathsf{L}\text-\mathrm{Coalg}\to\mathsf{T}\text-\mathrm{Emb}$ is an
  \Ord-functor over $\C^\two$. From this data we have to produce a monad
  morphism $\mathsf{R}_1\to\mathsf{T}$, or what is equivalent, an \Ord-functor
  \begin{equation}
    \label{eq:121}
    \mathsf{T}\text-\mathrm{Alg}\longrightarrow\mathsf{R}_1\text-\mathrm{Alg}
  \end{equation}
  where the notation on the right means the \Ord-category of
  $\mathsf{R}$-algebras with co\-domain~$1$. We can use $H$, the adjunction between
  $\prescript{\pitchfork_{\mathkz}}{}{(-)}$ and $(-)^{\pitchfork_{\mathkz}}$,
  and Theorem~\ref{thm:9} to
  define an \Ord-functor over $\C^\two$
  \begin{equation}
    \label{eq:122}
    \mathsf{T}\text-\mathrm{Alg}/1\longrightarrow
    \bigl(
    \prescript{\pitchfork_{\mathkz}}{}{\bigl(\mathsf{T}\text-\mathrm{Alg}/1\bigr)}
    \bigr)^{\pitchfork_{\mathkz}}
    \cong
    \bigl(\mathsf{T}\text-\mathrm{Emb}\bigr)^{\pitchfork_{\mathkz}}
    \xrightarrow{H^{\pitchfork_{\mathkz}}}
    \bigl(\mathsf{L}\text-\mathrm{Coalg}\bigr)^{\pitchfork_{\mathkz}}
    \cong
    \mathsf{R}\text-\mathrm{Alg}
  \end{equation}
  that assigns to each $\mathsf{T}$-algebra $A$ an $\mathsf{R}$-algebra of the
  form $A\to 1$. This is the $\Ord$-functor~\eqref{eq:121} we seek.

  In addition, the adjunction between
  $\prescript{\pitchfork_{\mathkz}}{}{(-)}$ and $(-)^{\pitchfork_{\mathkz}}$
  implies that for any $N\colon \mathsf{T}\text-\mathrm{Alg}/1\to
  \mathsf{R}\text-\mathsf{Alg}$ over $\C^\two$ there exists a unique $H\colon
  \mathsf{L}\text-\mathrm{Coalg}\to\mathsf{T}\text-\mathrm{Emb}$ over $\C^\two$
  such that \eqref{eq:122}~equals $N$. This means that we have established the
  necessary bijection.

  If $(\mathsf{L},\mathsf{R})=\Psi(\mathsf{T})$, the counit $\theta$ of the
  $I$-adjunction,
  \begin{equation}
    \label{eq:126}
    \diagram
    \mathbf{LOFS}(\C)\ar@/_/[dr]_{\tilde\Phi}&
    \mathbf{LIMnd}_{\mathrm{fib}}(\C)\ar@{^(->}[d]^I
    \dtwocell<\omit>{^<4>\theta}\ar[l]_-\Psi\\
    &\mathbf{LIMnd}(\C)
    \enddiagram
  \end{equation}
  has component at $\mathsf{T}$ the morphism of monads
  \begin{equation}
    \theta_{\mathsf{T}}\colon\tilde\Phi\Psi(\mathsf{T}) \longrightarrow
    \mathsf{T}
  \end{equation}
  corresponding in the construction of the previous paragraphs to the
  \Ord-functor $H$ that is the isomorphism
  $\mathsf{L}\text-\mathrm{Coalg}\cong\mathsf{T}\text-\mathrm{Emb}$. It follows
  from~\eqref{eq:122} that $\theta_{\mathsf{T}}$ is an isomorphism provided that
  \begin{equation}
    \label{eq:127}
    \mathsf{T}\text-\mathrm{Alg}/1\longrightarrow
    \bigl(\mathsf{T}\text-\mathrm{Emb}\bigr)^{\pitchfork_{\mathkz}}
  \end{equation}
  is an isomorphism, which was proved in Proposition~\ref{prop:6}. As mentioned
  above the present theorem, the invertibility of $\theta$ implies that $\Psi$
  is fully faithful.
\end{proof}


\begin{df}
  \label{df:23}
  We call a \textsc{lofs} \emph{\slkz-reflective} if it is isomorphic to one of the
  form $\Psi(\mathsf{T})$, for a fibrantly \kz-generating lax idempotent monad
  $\mathsf{T}$.
\end{df}
\begin{prop}
  \label{prop:10}
  For a reflective {\normalfont\textsl{\textsc{lofs}}} $(\mathsf{L},\mathsf{R})$ on an \Ord-category
  with terminal object, there is an isomorphism
  $\mathsf{L}\text-\mathrm{Coalg}\cong \mathsf{R}_1\text-\mathrm{Emb}$ over
  $\C^\two$ and $(\mathsf{L},\mathsf{R})\cong\Psi(\mathsf{R}_1)$.
\end{prop}
\begin{proof}
  Suppose that $(\mathsf{L},\mathsf{R})\cong\Psi(\mathsf{T})$ for a lax idempotent
  monad $\mathsf{T}$.
  By hypothesis,
  $\mathsf{L}\text-\mathrm{Coalg}\cong\mathsf{T}\text-\mathrm{Emb}$ for an
  \Ord-monad $\mathsf{T}$ on $\C^\two$. On the other hand,
  $\mathsf{R}\text-\mathrm{Alg}\cong\mathsf{L}\text-\mathrm{Coalg}^{\pitchfork_{\mathkz}}$
  for any \textsc{lofs}, as we saw in Theorem~\ref{thm:9}. Therefore,
  \begin{equation}
    \label{eq:143}
    \mathsf{R}_1\text-\mathrm{Alg}=
    \mathsf{R}\text-\mathrm{Alg}_1\cong
    \mathsf{T}\text-\mathrm{Emb}^{\pitchfork_{\mathkz}}_1\cong
    \mathsf{T}\text-\mathrm{Alg}
  \end{equation}
  where the subscript $1$ denotes the fiber of the various categories fibered
  over $\C$ via the codomain functor. The last isomorphism of the sequence is
  the one provided by Proposition~\ref{prop:6}. Since the isomorphism
  $\mathsf{R}_1\text-\mathrm{Alg}\cong \mathsf{R}\text-\mathrm{Alg}$ constructed
  is over $\C$, we obtain an isomorphism between $\mathsf{R}_1$ and $\mathsf{T}$.
\end{proof}

\begin{notation}
\label{not:E,M}
In this section we will denote by $(\mathsf{E},\mathsf{M})$ the \textsc{lofs} on
\C\ whose $\mathsf{E}$-coalgebras are \textsc{lari}s in \C\ and whose
$\mathsf{M}$-algebras are split opfibrations in \C.
\end{notation}
\begin{df}
  \label{df:24}
We will refer to those \textsc{lofs}s $(\mathsf{L},\mathsf{R})$ that admit a morphism
$(\mathsf{E},\mathsf{M})\to(\mathsf{L},\mathsf{R})$ as
\emph{sub-{\normalfont\textsl{\textsc{lari lofs}}s}}. If such morphism exists,
it is unique.
\end{df}
Not all \textsc{lofs}s are sub-\textsc{lari}. For example, the initial
\textsc{awfs} (the one that factors a morphism $f$ as $f=Rf\cdot Lf$ with
$Lf=1_{\dom(f)}$ and $Rf=f$) is orthogonal and, thus, lax orthogonal. Coalgebras
for the associated comonad are the invertible morphisms in \C. It is clear that
not every \textsc{lari} is an isomorphism, so this \textsc{lofs} is not
sub-\textsc{lari}.
\begin{prop}
  \label{prop:5}
  {\normalfont\slkz}-reflective {\normalfont\textsl{\textsc{lofs}s}} are
  sub-{\normalfont\textsl{\textsc{lari}}}.
\end{prop}
\begin{proof}
  By definition, $\mathsf{L}\text-\mathrm{Coalg}$ is isomorphic over $\C^\two$ to
  $\mathsf{T}\text-\mathrm{Emb}$, for a certain $\mathsf{T}$. We have to show
  that there exists a (unique) \Ord-functor
  \begin{equation}
    \label{eq:124}
    \lari(\C)\longrightarrow \mathsf{T}\text-\mathrm{Emb}
  \end{equation}
  over $\C^\two$. By definition of $\mathsf{T}\text-\mathrm{Emb}$ as a pullback
  (see Definition~\ref{df:13}) it suffices to exhibit a commutative square
  \begin{equation}
    \diagram
    {\lari(\C)}
    \ar[r]^-{}\ar[d]_{}
    &
    {\lari(\mathsf{T}\text-\mathrm{Alg})}
    \ar[d]^{}
    \\
    {\C^\two}
    \ar[r]_-{F^{\mathsf{T}}}
    &
    {\mathsf{T}\text-\mathrm{Alg}^\two}
    \enddiagram
  \end{equation}
  where the vertical arrows are the obvious forgetful \Ord-functors.
  The \Ord-functor $F^{\mathsf{T}}$ obviously induces another
  $\lari(\C)\to\lari(\mathsf{T}\text-\mathrm{Alg})$ that makes the diagram
  commutative, since any \Ord-functor preserves \textsc{lari}s.
\end{proof}

\begin{df}
\label{df:25}
We shall be interested in \textsc{lofs} $(\mathsf{L},\mathsf{R})$ that satisfy
the following cancellation properties:
\begin{itemize}
\item {If $g$ and $g\cdot f$ are $\mathsf{L}$-coalgebras, then $f$ is an
    $\mathsf{L}$-coalgebra.}
\item If, in the following diagram, $g$, $g'$, $g\cdot f$ and $g'\cdot f'$ are
  $\mathsf{L}$-coalgebras and $(v,w)$ and $(u,w)$ are morphisms of
  $\mathsf{L}$-coalgebras, then $(u,v)$ is a morphism of
  $\mathsf{L}$-coalgebras.
  \begin{equation}
    \xymatrixrowsep{.5cm}
    \diagram
    \cdot\ar[d]_f\ar[r]^u&\cdot\ar[d]^{f'}\\
    \cdot\ar[r]^v\ar[d]_g&\cdot\ar[d]^{g'}\\
    \cdot\ar[r]^-w&\cdot
    \enddiagram
  \end{equation}
\end{itemize}
We call these \textsc{lofs}s \emph{cancellative}.
\end{df}
The definition of cancellative \textsc{lofs} regards being a \textsc{lari} as a
property. As a result, it does not extend from \Ord-categories to 2-categories
without modification.
\begin{ex}
  \label{ex:7}
  For \textsc{lofs}s that are \textsc{ofs}s on a category, or in other words,
  when both the comonad and the monad of the \textsc{lofs}s are idempotent, the
  second condition of the definition above is superfluous. Therefore,
  cancellative \textsc{ofs}s are precisely the reflective \textsc{ofs}, as shown
  in~\cite[Thm.~2.3]{MR779198}. This is the result that we will generalise in
  Theorem~\ref{thm:7}.
\end{ex}
\begin{lemma}
  \label{l:13}
  The {\normalfont\textsl{\textsc{lofs}}} $(\mathsf{E},\mathsf{\mathsf{M}})$
  is cancellative.
\end{lemma}
\begin{proof}
  Recall that $\mathsf{E}$-coalgebras are the same as \textsc{lari}s. Suppose
  that $f$ and $g$ are composable morphisms and that $g\dashv r$ and
  $(g\cdot f)\dashv t$ are \textsc{lari} structures. Defining $s=t\cdot g$, we
  have that $s\cdot f=t\cdot g\cdot f=1$. It remains to prove that
  $f\cdot s=f\cdot t\cdot g\leq 1$, which is equivalent to
  $g\cdot f\cdot t\cdot g\leq g$, and this inequality holds since
  $g\cdot f\cdot t\leq 1$.
  \begin{equation}
    \label{eq:112}
    \xymatrixcolsep{2cm}
    \diagram
    \bullet
    \ar@<-13pt>@{<-}[dd]_{t}
    \ar@<-7.5pt>@{}[dd]|{\vdash}
    \ar[d]|f
    \ar[r]^u&
    \bullet
    \ar@<13pt>@{<-}[dd]^{t'}
    \ar@<7.5pt>@{}[dd]|{\dashv}
    \ar[d]|{f'}\\
    \bullet
    \ar[d]|g
    \ar@{<-}@<13pt>[d]^r
    \ar@{}@<7.5pt>[d]|{\dashv}
    \ar[r]^v&
    \bullet
    \ar[d]|{g'}
    \ar@{<-}@<-13pt>[d]_{r'}
    \ar@{}@<-7.5pt>[d]|{\vdash}
    \\
    \bullet
    \ar[r]^w&
    \bullet
    \enddiagram
  \end{equation}
  Now suppose given morphisms of \textsc{lari}s
  $(u,w)\colon g\cdot f\to g'\cdot f'$ and $(v,w)\colon g\to g'$, as
  depicted. We have to show that $(u,v)\colon f\to f'$ is a morphism of
  \textsc{lari}s, ie that $u\cdot t\cdot g=t'\cdot g'\cdot v$,
  which holds by the following string of equalities
  \begin{equation}
    u\cdot t\cdot g = t'\cdot w \cdot g = t'\cdot g'\cdot v
  \end{equation}
  completing the proof.
\end{proof}

\begin{thm}
  \label{thm:7}
  For a sub-{\normalfont\textsl{\textsc{lari lofs}}} $(\mathsf{L},\mathsf{R})$
  on a finitely complete \Ord-category, the following
  statements are equivalent:
  \begin{enumerate}
  \item\label{item:32} It is cancellative.
  \item\label{item:33} It is reflective.
  \end{enumerate}
\end{thm}
\begin{proof}
  When $\mathsf{L}\text-\mathrm{Coalg}$ is isomorphic to
  $\mathsf{T}\text-\mathrm{Emb}$ for some lax idempotent $\mathsf{T}$, it always
  satisfies the cancellation properties of Definition~\ref{df:25} since \textsc{lari}s do: if
  $g$ and $g\cdot f$ are $\mathsf{T}$-embeddings, ie if $Tg$ and
  $T(g\cdot f)=Tg\cdot Tf$ are \textsc{lari}s, then $Tf$ is a \textsc{lari},
  which is to say that $f$ is a $\mathsf{T}$-embedding; and similarly for
  morphisms. See Lemma~\ref{l:13}.

  Conversely, suppose that $(\mathsf{L},\mathsf{R})$ is cancellative
  (Definition~\ref{df:25}) and
  there is a morphism of \textsc{awfs}s
  $(\mathsf{E},\mathsf{M})\to(\mathsf{L},\mathsf{R})$, or equivalently, there is
  an \Ord-functor $\lari(\C)\to\mathsf{L}\text-\mathrm{Coalg}$ over
  $\C^\two$. We shall show that the \Ord-functor
  $\mathsf{L}\text-\mathrm{Coalg}\to\mathsf{R}_1\text-\mathrm{Emb}$ of
  Corollary~\ref{cor:3} is an isomorphism, so
  $(\mathsf{L},\mathsf{R})\cong\Psi(\mathsf{R}_1)$ is reflective.

  If $f\colon X\to Y$ is an $\mathsf{R}_1$-embedding, then consider
  the following commutative diagram.
  \begin{equation}
    \diagram
    X\ar[d]_{L!}\ar[r]^f&Y\ar[d]^{L!}\\
    R_1X\ar[d]\ar[r]^-{R_1f}&R_1Y\ar[d]\\
    1\ar@{=}[r]&1
    \enddiagram
  \end{equation}
  The morphisms $L!$ are cofree $\mathsf{L}$-coalgebras while $R_1f$ is a
  \textsc{lari} and therefore an $\mathsf{L}$-coalgebra. So, $L!\cdot f$ is an
  $\mathsf{L}$-coalgebra and $f$ is an $\mathsf{L}$-coalgebra by the
  cancellation hypothesis. This means that each $\mathsf{R}_1$-embedding is an
  $\mathsf{L}$-coalgebra, and all that remains to prove is that morphisms of
  $\mathsf{R}_1$-embeddings are morphisms of $\mathsf{L}$-coalgebras.

  Let $(u,v)\colon f\to f'$ be a morphism of $\mathsf{R}_1$-embeddings, so
  $(R_1u,R_1v)\colon R_1f\to R_1f'$ is a morphism of \textsc{lari}s, and,
  therefore, a morphisms of $\mathsf{L}$-coalgebras. It follows that $(u,R_1v)$,
  depicted on the left below, is a morphism of $\mathsf{L}$-coalgebras.
  \begin{equation}
    \label{eq:111}
    \xymatrixcolsep{1.5cm}
    \diagram
    X\ar[d]_{L!}\ar[r]^u&
    X'\ar[d]^{L!}
    \\
    R_1X
    \ar[d]_{R_1f}
    \ar[r]^-{R_1u}
    &
    R_1X'
    \ar[d]^{R_1f'}
    \\
    R_1Y\ar[r]^-{R_1v}&
    R_1Y'
    \enddiagram
    \qquad
    =
    \qquad
    \diagram
    X\ar[r]^u\ar[d]_f&
    X'\ar[d]^{f'}\\
    Y\ar[r]^v\ar[d]_{L!}&
    Y'\ar[d]^{L!}\\
    R_1Y\ar[r]^-{R_1v}&
    R_1Y'
    \enddiagram
  \end{equation}
  On the other hand, $(v,R_1v)$ is a morphism of $\mathsf{L}$-coalgebras, being
  the image under $L$ of the morphism $(v,1)\colon (Y\to 1)\to(Y'\to 1)$. By the
  second part of Definition~\ref{df:25}, we deduce that $(u,v)$ is a morphism of
  $\mathsf{L}$-coalgebras, as required. This shows that
  $\mathsf{L}\text-\mathrm{Coalg}\to\mathsf{R}_1\text-\mathrm{Emb}$ is an
  isomorphism, completing the proof.
\end{proof}

\section{Simple adjunctions}
\label{sec:simple-monads}
In \S \ref{sec:simple-reflections} we saw that a reflection $\mathsf{T}$ on \C\ is simple
if and only if $\mathsf{T}\text-\mathrm{Iso}\to\C^\two$ is comonadic. In this section we
generalise that result in three directions. First, we work with \Ord-enriched
categories, \Ord-enriched functors and so on. Secondly, the 2-dimensional aspect
introduced by the enrichment over \Ord\ allows us to substitute isomorphisms by
\textsc{lari}s and $T$-isomorphisms by $\mathsf{T}$-embeddings. Thirdly, even
though \S \ref{sec:simple-reflections} speaks of reflections, the
constructions therein only need an adjunction (not necessarily a reflection) and this is
the framework we choose.


\begin{df}
  \label{df:8}
  Let $S\dashv G\colon\mathcal{B}\to\C$ be an adjunction between locally monotone
  functors on \Ord-categories, of which we require $\C$ to have pullbacks and
  $\mathcal{B}$ to have comma-objects. We can always construct a monad
  $\mathsf{R}$ on $\C^\two$ by considering the comma-object $Kf=GSf\downarrow
  \eta_Y$ and defining $Rf\colon Kf\to Y$ as the second projection.
  \begin{equation}
    \label{eq:68}
    \diagram
    X\ar[dr]|{Lf}\ar@/^12pt/[drr]^{\eta_X}\ar@/_12pt/[ddr]_f&&\\
    &Kf\ar[r]^{q_f}\ar[d]_{Rf}&
    GSX\ar[d]^{GSf}\\
    &Y\ar[r]_{\eta_Y}\ar@{}[ur]|\geq&GSY
    \enddiagram
  \end{equation}
  The \Ord-functorial factorisation $f=Rf\cdot Lf$ has an associated locally
  monotone copointed endofunctor $\Phi\colon L\Rightarrow 1$, where the component
  $\Phi_f$ is provided by the commutative square displayed.
  \begin{equation}
    \label{eq:82}
    \diagram
    {\cdot}
    \ar@{=}[r]^-{}\ar[d]_{Lf}
    &
    {\cdot}
    \ar[d]^{f}
    \\
    {\cdot}
    \ar[r]^-{Rf}
    &
    {\cdot}
    \enddiagram
  \end{equation}
\end{df}
We continue with the notation of previous sections, where
$(\mathsf{E},\mathsf{M})$ denotes the \textsc{lofs} whose
$\mathsf{E}$-coalgebras are the \textsc{lari}s.
\begin{rmk}
  \label{rmk:7}
  The comma-square of Definition~\ref{df:8} can be obtained by pulling back
  along $\eta_Y$ the image under $G$ of the projection $M(Sf)\colon Sf\downarrow
  SY\to SY$.
  \begin{equation}
    \diagram
    Kf\ar[d]_{Rf}\ar[r]\ar@{}[dr]|{\mathrm{pb}}&
    G(Sf\downarrow SY)\ar[d]|{G(MSf)}\ar[r]\ar@{}[dr]|\geq&
    GSX\ar[d]^{GSf}\\
    Y\ar[r]_-{\eta_Y}&GSY\ar@{=}[r]&GSY
    \enddiagram
  \end{equation}
\end{rmk}
\begin{lemma}
  \label{l:6}
  There is a pullback square of locally monotone endofunctors of $\C^\two$, as
  depicted on the left. There is a pullback of \Ord-categories, as depicted on
  the right.
  \begin{equation}
    \label{eq:83}
    \diagram
    {L}
    \ar@{}[dr]|{\mathrm{pb}}
    \ar[r]^-{}\ar[d]_{\Phi}
    &
    {G^\two {E}S^\two}
    \ar[d]^{G^\two\Phi^{{E}} S^\two}
    \\
    {1_{\C^\two}}
    \ar[r]^-{\eta^\two}
    &
    {G^\two S^\two}
    \enddiagram
    \qquad
    \diagram
    (L,\Phi)\text-\mathrm{Coalg}
    \ar[r]^-{}\ar[d]_{U}\ar@{}[dr]|{\mathrm{pb}}
    &
    {({E},\Phi^{{E}})\text-\mathrm{Coalg}}
    \ar[d]^{}
    \\
    {\C^\two}
    \ar[r]_-{S^\two}
    &
    {\mathcal{B}^\two}
    \enddiagram
  \end{equation}
\end{lemma}
\begin{proof}
  In order to obtain a pullback square as on the left hand side of the
  statement, we need to give two pullback squares: one corresponding to the
  domain component and another corresponding to the codomain component. We define
  the domain component of $L\to G^\two ES^\two$ to be the unit
  $\eta\colon 1\to GS$; this is possible since $\dom E=1$. The resulting has
  horizontal morphisms both equal to $\eta$ and vertical morphisms equal to the
  identity, since $\dom\Phi^E=1$. This square is manifestly a pullback.  The
  codomain component we choose is the pullback square of Remark~\ref{rmk:7}.

  The fact that there is a pullback of $\Ord$-functors as on the right hand side
  of the statement follows easily, and it is a well-known fact (see, eg,
  \cite[Prop.~9.2]{MR589937}).
\end{proof}
As a consequence of the previous lemma, the pullback square in~\eqref{eq:141}
that defines ${S}\text-\mathrm{Emb}$
factors as two pullback squares, as depicted.
\begin{equation}
  \label{eq:66}
  \diagram
  S\text-\mathrm{Emb}\ar[d]\ar[r]\ar@{}[dr]|{\mathrm{pb}}&
  \mathsf{E}\text-\mathrm{Coalg}\ar[d]^{\cong}\\
  (L,\Phi)\text-\mathrm{Coalg}\ar[r]\ar@{}[dr]|{\mathrm{pb}}\ar[d]&
  ({E},\Phi^{{E}})\text-\mathrm{Coalg}\ar[d]\\
  \C^\two\ar[r]^{S^\two}&
  \mathcal{B}^\two
  \enddiagram
\end{equation}
The isomorphism
$\mathsf{E}\text-\mathrm{Coalg}\cong(E,\Phi^E)\text-\mathrm{Coalg}$, which is
just the inclusion, was exhibited in Lemma~\ref{l:5}.
The \Ord-functor $S\text-\mathrm{Emb}\to(L,\Phi)\text-\mathrm{Coalg}$ is an
isomorphism, being the pullback of an isomorphism. The remark that follows describes this
functor and its inverse in more explicit terms.
\begin{rmk}
  \label{rmk:5}
  Suppose that $f\colon X\to Y$ has a structure of $(L,\Phi)$-coalgebra, given by
  $(1,s)\colon f\to Lf$, where $s\colon Y \to Kf$. This structure corresponds
  bijectively to an $r_f\colon SY\to SX$ in $\mathcal{B}$ with $r_f\cdot Sf=1$
  and $Sf\cdot r_f\leq 1$, in a way that can be explicitly described:
  $r_f\colon SY\to SX$ is the morphism whose transpose under the adjunction
  $S\dashv G$ is $q_f\cdot s\colon Y\to Kf\to GSX$, ie
  \begin{equation}
    \label{eq:85}
    r_f=\big(SY\xrightarrow{Ss}SKf\xrightarrow{Sq_f}SGSX\xrightarrow{\varepsilon_{SX}}SX\big).
  \end{equation}
  and
  \begin{equation}
    \label{eq:86}
    Rf\cdot s=1\qquad q_f\cdot s=\big(Y\xrightarrow{\eta_Y}GSY\xrightarrow{Gr_f}GSX\big).
  \end{equation}
\end{rmk}
\begin{df}
  \label{df:10}
  We say that the adjunction $S\dashv G$ is
  \emph{simple} (or \emph{simple with respect to $(\mathsf{E},\mathsf{M})$}) if,
  for each $f\colon X\to Y$ in \C,
  the morphism $Lf$ has an $S$-embedding structure given by
  \begin{equation}
    \label{eq:77}
    \big(SX\xrightarrow{SLf}SKf\big)\dashv
    \big(SKf\xrightarrow{Sq_f}SGSX\xrightarrow{\varepsilon_{SX}}SX\big).
  \end{equation}
  where $\varepsilon$ is the counit of $S\dashv G$.
  This amounts to the existence of the inequality $SLf\cdot\varepsilon_{SX}\cdot Sq_f\leq 1$.
\end{df}
The following theorem is an analogue to the characterisation of simple reflections
of \S \ref{sec:simple-reflections}.
\begin{thm}
  \label{thm:2}
  The following statements are equivalent.
  \begin{enumerate}
  \item \label{item:11} The adjunction $S\dashv G$ is simple.
  \item \label{item:10} The locally monotone forgetful functor $U\colon
    S\text-\mathrm{Emb}\to\C^\two$ has a right adjoint and the induced comonad
    has underlying functor $L$ and counit $\Phi\colon L\Rightarrow 1_{\C^\two}$.
  \item \label{item:9} The locally monotone copointed endofunctor $\Phi\colon
    L\Rightarrow 1_{\C^\two}$ admits a comultiplication $\Sigma\colon L\Rightarrow L^2$
    making $\mathsf{L}=(\mathsf{L},\Phi,\Sigma)$ into a comonad whose category of coalgebras is
    isomorphic to $S\text-\mathrm{Emb}$ over $\C^\two$.
\end{enumerate}
\end{thm}
\begin{proof}
  Clearly (\ref{item:9}) implies (\ref{item:10}). The opposite implication holds
  if $U$ is comonadic, which is if it has a right adjoint, by Beck's
  Theorem~\ref{thm:3} and Lemma~\ref{l:4}, showing that (\ref{item:10}) implies
  (\ref{item:9}).

  Let us now prove that (\ref{item:9}) implies~(\ref{item:11}). Let $f\colon
  X\to Y$ be a morphism of \C. The comultiplication $\Sigma_f\colon Lf\to L^2f$
  is of the form $\Sigma_f=(1,\sigma_f)$ for $\sigma_f\colon Kf\to KLf$. One of
  the counit axioms of the comonad says
  \begin{equation}
    \label{eq:81}
    1=\big(Kf\xrightarrow{\sigma_f}KLf\xrightarrow{K(1,Rf)}Kf\big)
  \end{equation}
  and upon composing with the projection $q_f\colon Kf\to GSX$ we have
  \begin{equation}
    \label{eq:84}
    q_f=q_f\cdot K(1,Rf)\cdot \sigma_f=q_{Lf}\cdot \sigma_f=Gr_{Lf}\cdot \eta_{X}
  \end{equation}
  where we have used, first the definition of $K$ as a comma-object
  (Definition~\ref{df:8}), and then the fact that $\sigma_f$ is an
  $(L,\Phi)$-coalgebra structure on $Lf$ together with the explicit description
  of the isomorphism $S\text-\mathrm{Emb}\cong(L,\Phi)\text-\mathrm{Coalg}$
  (Remark~\ref{rmk:5}); as before, $r_{Lf}\colon SKf\to SX$ denotes the right
  adjoint retract that endows $Lf$ with an $S$-embedding structure. By
  adjointness, the equality~\eqref{eq:84} is equivalent to $r_{Lf}=
  \varepsilon_{SX}\cdot Sq_f$, which is precisely saying that $S\dashv G$ is
  simple.

  Finally, we prove that (\ref{item:11}) implies~(\ref{item:10}). For each $g\colon X\to Y$, the
  morphism $Lg\colon X\to Kg$ has an $S$-embedding structure, given by
  \begin{equation}
    \label{eq:87}
    r_{Lg}=\varepsilon_{SX}\cdot Sq_g\colon SKg\longrightarrow SX.
  \end{equation}
  This defines a functor $J\colon\C^\two\to S\text-\mathrm{Emb}$, since the
  image of any morphism $(h,k)\colon f\to g$ is compatible with the right
  adjoints $r_{Lf}$ and $r_{Lg}$. To wit,
  \begin{equation}
    \label{eq:92}
    r_{Lg}\cdot SK(h,k)=
    \varepsilon_{SZ}\cdot Sq_g\cdot SK(h,k)=\varepsilon_{SZ}\cdot SGSh\cdot Sq_f
    =Sh\cdot\varepsilon_{SX}\cdot Sq_f=
    Sh\cdot r_{Lf}.
  \end{equation}
  It is clear that $J$ is a locally monotone functor. We shall show that it
  is a right adjoint to the forgetful functor $U\colon S\text-\mathrm{Emb}\to\C^\two$.

  Given an $S$-embedding $(f,r_f)$ in \C, consider its associated
  $(L,\Phi)$-coalgebra structure, as described in Remark~\ref{rmk:5}:
  \begin{equation}
    \label{eq:88}
    (1,s_f)\colon (f,r_f)\longrightarrow (Lf,r_{Lf})\qquad
    \diagram
    X\ar@{=}[r]\ar[d]_f&X\ar[d]^{Lf}\\
    Y\ar[r]^-{s_f}&Kf
    \enddiagram
  \end{equation}
  where $s_f$ is defined by the equalities
  \begin{equation}
    \label{eq:89}
    Rf\cdot s_f=1_X\qquad
    q_f\cdot s_f=Gr_{f}\cdot \eta_{Y}\colon Y\to GSY\to GSX.
  \end{equation}
  If we equip $Lf$ with the $S$-embedding structure $r_{Lf}$ of~\eqref{eq:87},
  then $(1,s_f)$ becomes a morphism in $S\text-\mathrm{Emb}$, since
  \begin{equation}
    \label{eq:90}
    r_{Lf}\cdot Ss_f= \varepsilon_{SX}\cdot Sq_f\cdot Ss_f= \varepsilon_{SX}\cdot
    SGr_f\cdot S\eta_Y=r_f\cdot \varepsilon_{SY}\cdot S\eta_{Y}=r_f.
  \end{equation}
  Furthermore, \eqref{eq:88}~are the components of a natural transformation
  $\Psi\colon 1_{S\text-\mathrm{Emb}}\Rightarrow JU$.
  To see this, if $(h,k)\colon f\to g$ is a morphism in $S\text-\mathrm{Emb}$, where
  $g\colon Z\to W$, we have to show the equality $K(h,k)\cdot s_f=s_g\cdot
  k$. This holds since we have
  \begin{multline}
    \label{eq:91}
    q_g\cdot K(h,k)\cdot s_f= GSh\cdot q_f\cdot s= GSh\cdot Gr_f\cdot\eta_Y=\\=
    Gr_g\cdot GSk\cdot\eta_Y= Gr_g\cdot \eta_W\cdot k=q_g\cdot s_g\cdot k
  \end{multline}
  \begin{equation}
    Rg\cdot K(h,k)\cdot s_f=k\cdot Rf\cdot s=k=
    Rg\cdot s_g\cdot k.\label{eq:93}
  \end{equation}
  To complete the proof, we show that the transformation $\Psi$ with
  components~\eqref{eq:88} is the unit of an adjunction $U\dashv J$ with counit
  $\Phi\colon JU=L\Rightarrow 1_{\C^\two}$. The triangle identity
  $\Phi_{U(f,r_f)}\cdot U\Psi_f=1$ holds, since it amounts to $Rf\cdot
  s_f=1$. The other triangle identity, $J\Phi_f\cdot \Psi_{Jf}=1$, requires a
  bit more of work. The morphism of $S$-embeddings $\Psi_{Jf}$ has the form
  $(1,\sigma_f)\colon Lf\to L^2f$, and is defined by $RLf\cdot \sigma_f=1$ and
  \begin{equation}
    q_{Lf}\cdot \sigma_f=\big(Kf\xrightarrow{q_f}GSX\big).
    \label{eq:94}
  \end{equation}
  The morphism $J\Psi_f$ equals $(1,K(1,Rf))$, so
  the triangular equality translates into
  $ K(1,Rf)\cdot\sigma_f=1$. Both sides are equal to $Rf$ upon composing with
  $Rf$, so it remains to show that $q_{f}\cdot K(1,Rf)\cdot\sigma_f=q_{f}$. This
  equality follows easily from what we already know about $\sigma_f$.
  \begin{equation}
    \label{eq:96}
    q_{f}\cdot K(1,Rf)\cdot\sigma_f= q_{Lf}\cdot\sigma_f=q_f.
  \end{equation}
  This completes the proof of the statement~(\ref{item:10}), and so, the
  proof of the theorem.
\end{proof}
\begin{thm}[Beck]
  \label{thm:3}
  A functor $U\colon \mathcal{T}\to\mathcal{A}$ is comonadic if and only if
  \begin{enumerate}
  \item \label{item:12} It has a right adjoint.
  \item \label{item:14} $U$ creates equalisers of parallel pairs of morphisms in
    $\mathcal{T}$ whose image under $U$ has an absolute equaliser in $\mathcal{A}$.
  \end{enumerate}
\end{thm}
\begin{lemma}
  \label{l:4}
  In a pullback diagram of functors, as displayed, $U$ satisfies
  condition~(\ref{item:14}) of Beck's
  Theorem~\ref{thm:3} if $V$ does so.
  \begin{equation}
    \label{eq:78}
    \diagram
    \mathcal{T}\ar[r]^Q\ar[d]_U&\mathcal{S}\ar[d]^V\\
    \mathcal{A}\ar[r]^-S&\mathcal{B}
    \enddiagram
  \end{equation}
\end{lemma}

\begin{rmk}
\label{rmk:9}
Even if $U\colon S\text-\mathrm{Emb}\to\C^\two$ is comonadic, the requirement
that the associated comonad has underlying copointed endofunctor $(L,\Phi)$
is necessary for Theorem~\ref{thm:2} to hold. This can be seen at the same time
as exploring what the theorem means in the case that the \Ord-categories \C\ and
$\mathcal{B}$ are ordinary categories. In this case, a \textsc{lari} in
$\mathcal{B}$ is an isomorphism, so $S\text-\mathrm{Emb}$ is the full
subcategory $S\text-\mathrm{Emb}\subset\C^\two$ of morphisms inverted by $S$. It
may very well be the case that $S\text-\mathrm{Emb}\subset\C^\two$ is a coreflective
subcategory while the adjunction $S\dashv G$ is not simple. For example, if \C\
has finite limits and intersection of all strong
monomorphisms~\cite[Thm.~3.3]{MR779198}.
\end{rmk}

\section{Simple monads}
\label{sec:simple-monads-1}

\begin{df}
  \label{df:12}
  Let $\C$ be an \Ord-category that admits comma-objects and pullbacks. A monad
  $\mathsf{T}=(T,\eta,\mu)$ on \C\ whose functor part $T$ is locally
  monotone (ie, \Ord-enriched) is \emph{simple} if the free $\mathsf{T}$-algebra adjunction is
  simple in the sense of Definition~\ref{df:10}.
  \begin{equation}
    \label{eq:71}
    \diagram
    \C\ar@<6pt>[r]^{F^\mathsf{T}}\ar@<-6pt>@{<-}[r]_{U^\mathsf{T}}
    \ar@{}[r]|-\bot&
    \mathsf{T}\text-\mathrm{Alg}
    \enddiagram
  \end{equation}
\end{df}
Explicitly, $\mathsf{T}$ is simple when, for each $f\colon X\to Y$ in \C, the
morphism $F^\mathsf{T}(L^\mathsf{T}f)$ is a right adjoint of $\varepsilon^\mathsf{T}_{F^\mathsf{T}X}\cdot F^\mathsf{T}q^\mathsf{T}_f$,
with these morphisms defined by the following diagram, where the square is
a comma-object.
\begin{equation}
  \label{eq:74}
  \diagram
  X\ar[dr]|{Lf}\ar@/^12pt/[drr]^{\eta_X}\ar@/_12pt/[ddr]_{f}&
  &\\
  &Kf\ar[d]|{Rf}\ar@{}[dr]|{\geq}\ar[r]^-{q_f}&
  TX\ar[d]^{Tf}\\
  &Y\ar[r]_-{\eta_Y}&
  TY
  \enddiagram
\end{equation}

We will be specially interested in simple monads that are lax idempotent.
\begin{lemma}
  \label{l:7}
  A lax idempotent \Ord-monad $\mathsf{T}$ on \C\ is simple if
  and only if there is an adjunction $T(Lf)\dashv \mu_X\cdot Tq_f$, where
  $\mu_X$ is the multiplication of $\mathsf{T}$.
\end{lemma}
\begin{proof}
  The simplicity of $\mathsf{T}$ is the existence of an inequality
  \begin{equation}
    \label{eq:104}
    F^\mathsf{T}Lf\cdot \varepsilon_{F^\mathsf{T}X}\cdot F^\mathsf{T}q_f \leq 1.
  \end{equation}
  Applying the forgetful \Ord-functor $U^\mathsf{T}$ one obtains
  \begin{equation}
    \label{eq:105}
    TLf\cdot \mu_X\cdot Tq_f\leq 1
  \end{equation}
  and thus the adjunction of the statement. All this holds for a general
  \Ord-monad $\mathsf{T}$. If $\mathsf{T}$ is lax idempotent, the forgetful
  \Ord-functor $U^\mathsf{T}\colon\mathsf{T}\text-\mathrm{Alg}\to\C$ is locally full
  and in particular it reflects inequalities between morphisms. It follows that
  \eqref{eq:105}~implies~\eqref{eq:104}.
\end{proof}
\begin{cor}\label{cor:new}
A lax idempotent \Ord-monad $\mathsf{T}$ on \C\ is simple if
  and only if $TLf\cdot q_f\leq\eta_{Kf}$.
\end{cor}
\begin{proof}
By lax idempotency of $\mathsf{T}$, the left extension of $TLf\cdot q_f:Kf\to
TKf$ along $\eta_{Kf}$ is $\mu_{Kf}\cdot T(TLf\cdot q_f)=TLf\cdot \mu_X\cdot
Tq_f$; see Definition~\ref{df:2}~(\ref{item:8}). Therefore, \eqref{eq:105} holds
if, and only if, $TLf\cdot q_f\leq \eta_{Kf}$.
\end{proof}
Putting together Theorem~\ref{thm:2} and Definition~\ref{df:20}, we have:
\begin{cor}
  \label{cor:1}
  Simple lax idempotent monads on \Ord-categories with comma-objects are fibrantly generating.
\end{cor}
This means that, if $\C$ has comma-objects, each simple lax idempotent monad
$\mathsf{T}$ induces a \textsc{lofs} $(\mathsf{L},\mathsf{R})$ with
$\mathsf{L}\text-\mathrm{Coalg}$ isomorphic to $\mathsf{T}\text-\mathrm{Emb}$
over $\C^\two$.

\begin{prop}
  \label{prop:2}
  The monad $\mathsf{P}$ on \Ord\ described in Example~\ref{ex:3} is simple.
\end{prop}
\begin{proof}
  The proof uses Corollary~\ref{cor:new}, for which we shall need the description of the
  comma-object $Kf$ of~\eqref{eq:74} as
  \begin{equation}
    Kf=
    \bigl\{(W,y)\in P(X)\times Y: f_*(W)\subseteq\mathnormal\downarrow y\bigr\}
    =
    \bigl\{(W,y)\in P(X)\times Y: W\subseteq f^{*}(\mathnormal\downarrow y)\bigr\}
  \end{equation}
  and of the morphism $Lf\colon X\to Kf$ as $Lf(x)=(\mathnormal\downarrow
  x,f(x))$.

  We must show that
  \begin{equation}
    (Lf)_*\cdot q_f\leq \eta_{Kf}.
  \end{equation}
  Evaluating on $(W,y)\in Kf$, we have
  \begin{equation}
    (Lf)_*\cdot q_f(W,y)=(Lf)_*(W) \subseteq \eta_{Kf}(W,y)=\mathnormal\downarrow(W,y)
  \end{equation}
  if and only if
  \begin{equation}
    W\subseteq (Lf)^*\bigl(\mathnormal\downarrow(W,y)\bigr)
    =\bigl\{x\in X:(\mathnormal\downarrow x,f(x))\leq (W,y)\bigr\}.
  \end{equation}
  This last inequality always holds, since, for $w\in W$, the inclusion
  $\mathnormal\downarrow w\subseteq W$ always holds, and $f(w)\leq y$, because
  $f_*(W)\subseteq\mathnormal\downarrow y$.
\end{proof}
For each morphism $f\colon X\to Y$ there is a ``comparison'' morphism
\begin{equation}
  \kappa\colon T(Tf\downarrow\eta_Y)\longrightarrow T^2f\downarrow T\eta_Y
\end{equation}
induced by the
universal property of comma-objects. More explicitly, $\kappa$ is a morphism, as
displayed in the diagram below, unique with the property of making the triangles
\fbox{1} and \fbox{2} commutative.
\begin{equation}
\xymatrix@C=22pt@R=8pt{TX\ar[rrrrrr]^{T\eta_X}\ar[dddd]_{Tf}\ar[rd]|-{TLf}&&&&&&T^2X\ar[dddd]^{T^2f}\\
&TKf\ar[rrrrru]|(0.6){Tq_f}_(0.4){\fbox{1}}\ar[dddl]|(0.6){TRf}\ar@{}@<7pt>[ddl]^(0.3){\fbox{2}}\ar[dr]^\kappa\\
&&T^2f\downarrow T\eta_Y\ar[rrrruu]_{p_f}\ar[ddll]^{p_Y}\ar@{}[rrrrdd]|{\geq}\\
\\
TY\ar[rrrrrr]_{T\eta_Y}&&&&&&TTY}
\end{equation}
\begin{prop}
  \label{prop:12}
  A lax idempotent \Ord-monad $\mathsf{T}$ is simple provided that, for every
  $f$ and $u:Kf\to TKf$, $u\leq\eta_{Kf}$ whenever
  $\kappa\cdot u\leq \kappa\cdot\eta_{Kf}$, where $\kappa$ is the comparison
  morphism $TKf\to T^2f\downarrow T\eta_Y$.
\end{prop}
\begin{proof}
From
\begin{gather}
  p_f\cdot\kappa\cdot TLf\cdot q_f=T\eta_X\cdot q_f\leq\eta_{TX}\cdot
  q_f=p_f\cdot \kappa\cdot\eta_{Kf}
  \\
  p_Y\cdot \kappa\cdot TLf\cdot q_f=Tf\cdot
  q_f\leq \eta_Y\cdot Rf=p_Y\cdot\kappa\cdot\eta_{Kf}
\end{gather}
and the definition of comma-object one has $\kappa\cdot TLf\cdot
q_f\leq\kappa\cdot\eta_{Kf}$, and the conclusion follows from the hypothesis and
Corollary~\ref{cor:new}.
\end{proof}
For example, the above proposition applies in the cases when $\kappa$ is a full morphism.

\section{Submonads of simple monads}
\label{sec:subm-simple-monads}
The aim of the present section is to provide easy criteria that will allow us to
recognise simple submonads of simple lax idempotent monads. These results will
be later used in Corollary~\ref{cor:new3} of \S \ref{sec:filter-monads}.

\begin{lemma}
  \label{l:10}
  Let $\mathsf{T}$ be an \Ord-monad.
  If $\mathsf{T}$ is lax idempotent, then $\mathsf{T}$-embeddings
  are full if and only if the components of the unit $X\to TX$ are full.
\end{lemma}
\begin{proof}
  By definition of lax idempotent monad, the unit components
  $\eta_X\colon X\to TX$ are $\mathsf{T}$-embeddings, and, hence, they are full
  provided that $\mathsf{T}$-embeddings are full.

  Conversely, suppose that $f\colon X\to Y$ is a $\mathsf{T}$-embedding. Then,
  $\eta_Y\cdot f=Tf\cdot \eta_X$ is full, being a composition of the
  \textsc{lari} $Tf$ and the full morphism $\eta_X$. Therefore, $f$ is full.
\end{proof}

\begin{prop}
  \label{prop:8}
  Suppose that $\varphi\colon \mathsf{S}\to \mathsf{T}$ is a monad morphism between
  \Ord-monads and that its components $\varphi_X$ are
  $\mathsf{T}$-embeddings. If $\mathsf{T}$ is lax idempotent and the components of the unit $\eta_X:X\to TX$ are full, then $\mathsf{S}$ is lax idempotent, with full unit components $e_X:X\to SX$.
\end{prop}
\begin{proof}
That $\mathsf{S}$ is lax idempotent follows from the following calculations and fullness of $T\varphi_X\cdot\varphi_{SX}=\varphi_{TX}\cdot S\varphi_X$:
\[\varphi_{TX}\cdot S\varphi_X\cdot Se_X=T\eta_X\cdot \varphi_X\leq \eta_{TX}\cdot\varphi_X=\varphi_{TX}\cdot e_{TX}\cdot \varphi_X=\varphi_{TX}\cdot S\varphi_X\cdot e_{SX}.\]
Moreover, with $\eta_X=\varphi_X\cdot e_X$ full, also $e_X$ is full.
\end{proof}

We say that a morphism $f\colon X\to Y$ is a \emph{pullback-stable $\mathsf{T}$-embedding}
if the pullback of $f$ along any morphism into $Y$ is a
$\mathsf{T}$-embedding.
\begin{thm}
  \label{thm:8}
  Suppose that $\varphi\colon\mathsf{S}\to\mathsf{T}$ is a monad morphism
  between \Ord-monads whose components are pullback-stable
  $\mathsf{T}$-embeddings, and that $\mathsf{T}$-embeddings are full. If $\mathsf{T}$ is lax idempotent, then
  $\mathsf{S}$ is simple whenever $\mathsf{T}$ is so.
\end{thm}
\begin{proof}
Let us denote the unit of $\mathsf{S}$ by $e\colon 1\Rightarrow S$, and  the
\Ord-functorial factorisations obtained from $\mathsf{S}$ and
$\mathsf{T}$ following the construction of the comma-object~\eqref{eq:74}, respectively, by
\begin{equation}
  \label{eq:144}
  \bigl(X\xrightarrow{L_Sf}K_Sf\xrightarrow{R_Sf}Y\bigr)
  =
  \bigl(X\xrightarrow{f}Y\bigr)
  =
  \bigl(X\xrightarrow{L_Tf}K_Tf\xrightarrow{R_Tf}Y\bigr)
\end{equation}
Consider the following diagram
where $K_Tf=Tf\downarrow \eta_Y$, $K_Sf=Sf\downarrow e_Y$, and $L_Tf=\overline{\varphi}_f\cdot L_Sf$, and note that $\fbox{1}$ is a pullback.
\[\xymatrix@C=15pt@R=10pt{X\ar[rrr]^{e_X}\ar[dddd]_f\ar[rd]^{L_Sf}&&&SX\ar[rrr]^{\varphi_X}&&&TX\ar[dddd]^{Tf}\\
&K_Sf\ar[rru]_{t_f}\ar[rd]^{\overline{\varphi}_f}\ar[dddl]^{R_Sf}\ar@{}[rrrrru]|{\fbox{1}}\\
&&K_Tf\ar[rrrruu]_{q_f}\ar[ddll]^{R_Tf}\ar@{}[rrrrdd]|{\geq}\\
\\
Y\ar[rrr]_{e_Y}&&&SY\ar[rrr]_{\varphi_Y}&&&TY}\]
By Corollary \ref{cor:new} to conclude that $S$ is simple it is enough to show that $SL_Sf\cdot t_f\leq e_{K_Sf}$. And this inequality follows from the following calculations, using the fullness of $T\overline{\varphi_f}\cdot \varphi_{K_Sf}$.
\begin{multline}
  T\overline{\varphi}_f\cdot \varphi_{K_Sf}\cdot SL_Sf\cdot t_f=
  T\overline{\varphi}_f\cdot TL_Sf\cdot \varphi_X\cdot t_f=
  T\overline{\varphi}_f\cdot TL_Sf\cdot
  q_f\cdot\overline{\varphi}_f
  \leq
  \\
  \leq
  \eta_{Kf}\cdot\overline{\varphi}_f=
  T\overline{\varphi}_f\cdot\varphi_{K_Sf}\cdot e_{K_Sf}\qedhere
\end{multline}
\end{proof}

\begin{cor}
  \label{cor:2}
  Suppose that $\varphi\colon\mathsf{S}\to\mathsf{T}$ is a monad morphism between
  \Ord-monads whose components are $\mathsf{T}$-embeddings, and where
  $\mathsf{T}$ is lax idempotent and simple, with full unit components $X\to TX$. Then:
  \begin{enumerate}
  \item $\mathsf{S}$ is lax idempotent and simple, with full unit components $X\to SX$;
  \item every $\mathsf{S}$-embedding is a $\mathsf{T}$-embedding;
  \item $\mathsf{S}$-embeddings are full.
  \end{enumerate}
\end{cor}

\begin{proof}
(1) follows from Proposition \ref{prop:8}, while (3) follows directly from (2) and our assumptions. To show (2), first note that the unit components $e_X:X\to SX$ are $\mathsf{T}$-embeddings since both $\eta_X=\varphi_X\cdot e_X$ and $\varphi_X$ are. Now let $f:X\to Y$ be an $\mathsf{S}$-embedding. As a \textsc{lari}, $Sf$ is a $\mathsf{T}$-embedding, and so is $f$ because both $e_Y$ and $e_Y\cdot f=Sf\cdot e_X$ are $\mathsf{T}$-embeddings.
\end{proof}

\section{Filter monads}
\label{sec:filter-monads}

In this section we exhibit \textsc{awfs}s on the \Ord-category of \Tzero\
topological spaces arising from simple lax idempotent \Ord-monads.
These factorisations were constructed in \cite{MR2927175}.

As mentioned in Example \ref{ex:2} each \Tzero\ topological space $X$ carries an order given by
\begin{equation}
  \label{eq:158}
  x\leq y \text{ if and only if } y\in\overline{\{x\}}
\end{equation}
--~this is the opposite of what is usually called the \emph{specialisation order}. This
induces an order structure on each hom-set $\mathbf{Top}_0(X,Y)$ by defining $f\leq
g$ if $f(x)\leq g(x)$, for all $x\in X$, making $\mathbf{Top}_0$ into an
\Ord-enriched category.

A comma-object $f\downarrow g$ in $\mathbf{Top}_0$ can be described as
the subspace of $X\times Y$ defined by the subset $\{(x,y)\in X\times
Y\colon{}f(x)\leq g(y)\}$.
\begin{equation}
  \label{eq:149}
  \diagram
  f\downarrow g \ar[r]^-{d_1}\ar[d]_{d_0}
  \drtwocell<\omit>{'\leq}
  &
  Y\ar[d]^g
  \\
  X\ar[r]_-f
  &
  Z
  \enddiagram
\end{equation}

Denote by $\mathsf{F}\colon{}\mathbf{Top}_0\to\mathbf{Top}_0$ the filter
monad. If $X$ is a \Tzero\ space, $FX$ is the set of filters of open sets of $X$,
with topology generated by the subsets $U^\varhash=\{\varphi\in
FX:U\in\varphi\}$, where $U\in\mathcal{O}(X)$. The (opposite of the)
specialisation order on $FX$ results in the opposite of the inclusion of
filters. In particular, $FX$ is a poset. If $f\colon X\to Y$ is continuous, then $Ff$ is defined by
$Ff(\varphi)=\{V\in\mathcal{O}(Y):f^{-1}(V)\in\varphi\}$. The unit of the
monad has components $\eta_X\colon X\to FX$, where $\eta_X(x)$ is the
\emph{principal filter} generated by $x$, that is
$\eta_X(x)=\{U\in\mathcal{O}(X):x\in U\}$.
The multiplication of the monad has
components $\mu_X\colon F^2X\to FX$, given by
$\mu_X(\Theta)=\{U\in\mathcal{O}(X):U^\varhash\in\Theta\}$.

Observe that $\eta_X$ is a full morphism. It is in fact an \emph{embedding} meaning a topological embedding, in the usual sense:
a continuous function that is an homeomorphism onto its image, where the latter
is equipped with the subspace topology.

It was shown in \cite{MR0367013}
that the category of algebras for this monad is
isomorphic to the category whose objects are continuous lattices
\cite{MR0404073} and morphisms poset maps that preserve directed sups and
arbitrary infs. Our choice of the (opposite of the) specialisation order on
spaces, which is the opposite of the order used in \cite{MR0367013}, grants a
few comments as a way of avoiding confusion. A space $X\in\mathbf{Top}_0$ has an
$\mathsf{F}$-algebra structure precisely when the opposite of the poset
$(X,\leq)$ is a continuous lattice, where $\leq$ is the order
\eqref{eq:158}. The topology of the space $X$ can be recovered as the Scott
topology of the continuous lattice $(X,\leq)^{\mathrm{op}}$. A morphism of
$\mathsf{F}$-algebras $f\colon X\to Y$ is a continuous function that preserves
arbitrary suprema, as a poset map $(X,\leq)\to(Y,\leq)$
\cite[Thm. 4.4]{MR0367013}.

The filter monad $\mathsf{F}$ was shown to be lax idempotent in
\cite{MR1718976}, where it is also proved that a continuous function $f$ between
\Tzero\ spaces is an embedding if and only if $Ff$ is a \textsc{lari}. In other
words, $\mathsf{F}$-embeddings are precisely the topological embeddings.

%

%
\begin{thm}
  \label{thm:filtermonad}
  The \Ord-monad $\mathsf{F}$ is simple.
\end{thm}
\begin{proof}
We verify the hypothesis of Proposition~\ref{prop:12}.
For any pair of continuous maps $f\colon X\to Z$ and $g\colon Y\to Z$, the
comparison morphism
\begin{equation}
  \kappa\colon{}F(f\downarrow g)\longrightarrow
  F f\downarrow Fg\subset FX\times FY
\end{equation}
sends a filter $\varphi$ on $f\downarrow g$ to the pair of filters
$(\psi_0,\psi_1)$
\begin{equation}
  \label{eq:150}
  \psi_0=\{U\in \mathcal O(X): d_0^{-1}(U)\in\varphi\}
  \qquad
  \psi_1=\{V\in \mathcal O(Y): d_1^{-1}(V)\in\varphi\}
\end{equation}
where $d_0$ and $d_1$ are the projections from $f\downarrow g$ to $X$ and $Y$, respectively.
Given $(x,y)\in f\downarrow g$, recall that its image under the unit is
\begin{equation}
  \label{eq:151}
  \eta_{f\downarrow g}(x,y)=\{W\in \mathcal O(f\downarrow g):(x,y)\in W\}
\end{equation}
We have $(F d_0)\eta_{f\downarrow g}(x,y)=\eta_Xd_0(x,y)=\eta_X(x)$, and similarly,
$(Fd_1)\eta_{f\downarrow g}(x,y)=\eta_Y(y)$.

The hypothesis of Proposition~\ref{prop:12} will be satisfied if we show that
$\kappa\cdot u\leq \kappa\cdot\eta_{f\downarrow g}$ implies $u\leq
\eta_{f\downarrow g}$; or, in terms of filters, if we show that, given $\varphi\in
F(f\downarrow g)$, $(x,y)\in f\downarrow g$ as above, the inequalities $\psi_0\leq \eta_X(x)$ and
$\psi_1\leq \eta_Y(y)$ imply $\varphi\leq \eta_{f\downarrow g}(x,y)$. By
definition of the (opposite) specialisation order, we
need to show the two inclusions
\begin{gather}
  \{U\in\mathcal O(X): d_0^{-1}(U)\in\varphi\}\supseteq \{U\in\mathcal O(X):x\in
  U\}
  \\
  \{V\in\mathcal O(Y):d_1^{-1}(V)\in\varphi\}\supseteq \{V\in\mathcal O(Y):y\in
  V\}
\end{gather}
imply $\varphi\supseteq \{W\in\mathcal O(f\downarrow g):(x,y)\in W\}$.  Given
$x\in U\in\mathcal O(X)$, $y\in V\in\mathcal O(Y)$, then
\begin{equation}
  \label{eq:154}
  (U\times V) \cap( f\downarrow g)=
  d_0^{-1}(U)\cap d_1^{-1}(V)\in\varphi.
\end{equation}
But any neighbourhood $W$ of $(x,y)$ contains another of the form $(U\times
V)\cap(f\downarrow g)$, so $W\in\varphi$, completing the proof.
\end{proof}

Since every principal filter is completely prime, and so in particular prime and
proper, and $\mu_X(\Theta)$ is completely prime (resp. prime, proper) whenever
$\Theta$ is so, the functors $F_1$, $F_\omega$ and $F_\Omega$ that assign to
each space $X$ the space of proper (resp. prime, completely prime) filters are
the functor part of submonads $\mathsf{F}_1$, $\mathsf{F}_\omega$ and
$\mathsf{F}_\Omega$ of the filter monad, with the monad morphisms defined
pointwise by the corresponding embeddings.  Hence, using Corollary~\ref{cor:2},
we can immediately conclude:

\begin{cor}
\label{cor:new3}
The \Ord-monads of proper filters, of prime filters and of completely prime filters are lax idempotent and simple.
\end{cor}

Therefore these monads induce \textsc{lofs}s $(\mathsf{L}_\alpha,\mathsf{R}_\alpha)$, with $\alpha=0,1,\omega,\Omega$ (denoting $\mathsf{F}$ by $\mathsf{F}_0$), with associated weak factorisation systems $(\mathcal{L}_\alpha,\mathcal{R}_\alpha)$, where $\mathcal{L}_0$ is the class of embeddings, $\mathcal{L}_1$ is the class of dense embeddings, $\mathcal{L}_\omega$ is the class of flat embeddings, and $\mathcal{L}_\Omega$ is the class of completely flat embeddings \cite{MR1641443, MR1718976, MR2927175}. Moreover, $\mathcal{R}_\alpha$ is the class of morphisms which are injective with respect to $\mathcal{L}_\alpha$ (see \cite{MR2927175} for details).

\section{Metric spaces}
\label{sec:metric-spaces}

It is an insight of
Bill~Lawvere~\cite{Lawvere:Metricspaces,Lawvere:Metricspaces-reprint} that
metric spaces can be regarded as enriched categories and that, from this point of view,
completeness can be interpreted in terms of ``modules.'' The necessary base of enrichment
is the category of \emph{extended real numbers} \R.

The category \R\ has objects the real non-negative numbers plus an extra object $\infty$,
and has one morphism $\alpha\to\beta$ if and only if $\alpha\geq\beta$; $\infty$ is
an initial object and $0$ a terminal object. One can use the addition of real
numbers to define a symmetric monoidal structure on \R, with the convention that
adding $\infty$ always produces $\infty$. The unit object of this
tensor product is $0$. Furthermore, \R\ is closed, with internal hom
$[\alpha,\beta]$ equal to $\beta-\alpha$ if this difference is non-negative, and
equal to zero otherwise, with the convention that $[\alpha,\infty]=\infty$,
$[\infty,\infty]=0$ and $[\infty,\alpha]=0$.

A small \R-category can be described as a set $A$ with a distance function
$A(-,-)\colon A\times A\to\R$ that satisfies $A(a,a)=0$ for all $a\in A$ and the
triangular inequality. In general, it may very well happen that $A(a,b)=0$
even if $a\neq b$; the distance may not be symmetric, ie $A(a,b)\neq A(b,a)$,
and, the distance between two points may be $\infty$. We regard \R-categories as
generalised metric spaces and think of $A(a,b)\in\R$ as the ``distance'' from
$a$ to $b$.

For example, $\R$ itself is a generalised metric space with distance from
$\alpha$ to $\beta$ given by $[\alpha,\beta]$.

Each generalised metric space $A$ has an \emph{opposite} $A^{\mathrm{op}}$ with
the same points and distance $A^{\mathrm{op}}(a,b)=A(b,a)$.
We will concentrate on \emph{skeletal} generalised metric spaces, ie those
spaces
$A$ for which $A(a,b)=0=A(b,a)$ implies $a=b$. For example, \R\ is skeletal.

\R-enriched functors $f\colon A\to B$  are identified with functions $A\to B$
that are non-expansive: $A(a,b)\geq B(f(a),f(b))$. It is easy to verify that there exists a unique
\R-natural transformation $f\Rightarrow g\colon A\to B$ if and only if
$0=B(f(a),g(a))$ for all $a\in A$.
In this way we obtain an \Ord-category $\mathbf{Met}_{\mathrm{sk}}$ of skeletal generalised
metric spaces, with objects the skeletal  \R-categories, morphisms the \R-functors
and inequality $f\leq g$ between two of them given by the existence of a \R-natural
transformation $f\Rightarrow g$. Observe that $\mathbf{Met}_{\mathrm{sk}}(A,B)$
is not only a preorder but a poset, because $B$ is skeletal.

There is a notion of colimit suited to enriched categories, known as
\emph{weighted colimit} (or \emph{indexed colimit} in older texts);
see~\cite{Kelly:BCECT,Kelly:BCECTrep} for a standard reference. Each family of
weights induces a lax idempotent \Ord-monad on $\mathbf{Met}_{\mathrm{sk}}$
whose algebras are the skeletal generalised metric spaces that admit colimits
with weights in the family (see \cite[Theorems~6.1 and
6.3]{Kelly:MonadicityChosenColim}).
This monad is in fact simple (\S \ref{sec:simple-monads-1}),
as shown in the more general context in~\cite[\S 12]{clementino15:_lax}. It
follows from the theory developed herein that there is a \textsc{lofs} on
$\mathbf{Met}_{\mathrm{sk}}$ whose left morphisms are the embeddings with
respect to that monad and whose fibrant objects are the skeletal generalised
metric spaces that admit all $\Phi$-colimits (see Proposition~\ref{prop:4} and
Corollary~\ref{cor:1}). The rest of the section is occupied by the example of a
particular class of colimits that admit an explicit description.

The class of absolute colimits, ie the weights whose associated colimits are
preserved by any \R-functor whatsoever, generates a simple lax idempotent monad
$\mathsf{Q}$ on $\mathbf{Met}_{\mathrm{sk}}$.
Putting together \cite{Lawvere:Metricspaces}~and~\cite{MR749468} one can give a
description of $\mathsf{Q}$ in terms of Cauchy sequences. 

Cauchy sequences in a skeletal generalised metric space $A$ are defined in the
same way as for classical metric spaces. Two
Cauchy sequences $(a_n)$ and $(b_n)$ are \emph{equivalent} if both $A(a_n,b_n)$
and $A(b_n,a_n)$ have limit $0$.  Denote by $QA$ the set of equivalence classes of
Cauchy sequences in $A$ with distance $QA([a_n],[b_n])=\lim_nA(a_n,b_n)$. It is
not hard to see that $QA$ is a skeletal generalised metric space.

The assignment $A\mapsto QA$ is part of an \Ord-monad $\mathsf{Q}$ on
$\mathbf{Met}_{\mathrm{sk}}$, with unit $A\to QA$ the map that sends $a\in A$ to
the constant sequence on $a$, that we denote by $c_a$.

Convergence of a sequence $(x_n)$ to a point $a$ in generalised metric space $A$
differs from ordinary convergence in metric spaces only in that we have to
require that both $A(a,x_n)$ and $A(x_n,a)$ converge to $0$ in \R. The following
assertions are equivalent for a skeletal generalised metric space $A$: it is an
algebra for $\mathsf{Q}$; the canonical isometry $A\to QA$ has a left adjoint;
$A$ is a retract of a space of the form $QB$; every Cauchy sequence in $A$
converges. Spaces that satisfy these equivalent properties are known as
\emph{Cauchy-complete}.

If $(\mathsf{L}_{Q},\mathsf{R}_{Q})$ is the \kz-reflective~\textsc{lofs} on
$\mathbf{Met}_{\mathrm{sk}}$ generated by $\mathsf{Q}$, the
$\mathsf{L}_Q$-coalgebras, or left maps of the factorisation, are the
$\mathsf{Q}$-embeddings and can be characterised as follows.
\begin{prop}
  \label{prop:11}
  A non-expansive map $f\colon A\to B$ between skeletal geneneralised
  spaces is a $\mathsf{Q}$-embedding if and
  only if it is an isometry and for each $b\in B$ the non-expansive
  function $B(f-,b)\colon A^{\mathrm{op}}\to B$ can be written as
  $B(f-,b)=\lim_n A(-,x_n)$ for a Cauchy sequence $(x_n)$ in $A$.
\end{prop}
\begin{proof}
  First, if $Qf$ has a retract $r$, then $Qf$ is an isometry and thus $f$ is an
  isometry; for,
  $B(f(a),f(a'))=QB(c_{f(a)},c_{f(a')})=QB(Qf(c_a),Qf(c_a'))=QA(c_a,c_{a'})=A(a,a')$.

  If $r$ is moreover a right adjoint of $Qf$, and, for a given $b\in B$, $r(c_b)$ has an associated
  Cauchy sequence $(x_n)$ in $A$, we must have
  \begin{equation}
    \label{eq:137}
    B(f(a),b)=QB\bigl(c_{f(a)},c_b\bigr)=QB\bigl(Qf(c_a),c_b\bigr)=
    QA\bigl(c_a,r(c_b)\bigr)=\lim_nA(a,x_n)
  \end{equation}
  for all $a\in A$.

  Conversely, suppose that $f$ is an isometry and $B(f-,b)=\lim_nA(-,x_n)$. We
  must define an equivalence class of Cauchy sequences $r[b_n]\in QA$ for each
  $[b_n]\in QB$ in a way such that $QB([f(a_n)],[b_n])=QA([a_n],r[b_n])$. Since
  any Cauchy sequence is a limit of constant sequences (eg, $b_n=\lim_n c_{b_n}$),
  it suffices to define $r$ and to verify this equality for constant sequences; ie we
  have to give $r[c_b]\in QA$ such that $B(f(a),b)=QA(c_a,r[c_b])$. Since we know
  that $B(f-,b)=\lim_nA(-,x_n)$, we may set $r[c_b]=[x_n]$ and the equality
  holds. In this way we prove that there is an adjunction $Qf\dashv r\colon QB\to
  QA$. It remains to prove that $r\cdot Qf=1$, but $f$ is an isometry, which
  implies that $Qf$ is an isometry and
  therefore one-to-one, so the equality follows from the adjunction triangle
  equation $Qf\cdot r\cdot Qf=Qf$.
\end{proof}
It follows from the general theory that, given a $\mathsf{Q}$-embedding
$f\colon A\to B$ and a non-expansive function $h\colon A\to C$ into
Cauchy-complete skeletal generalised metric space $C$, there is an extension $d$.
\begin{equation}
  \label{eq:114}
  \xymatrixrowsep{.6cm}
  \diagram
  A\ar[r]^h\ar[d]_f&C\\
  B\ar@{..>}[ur]_d&
  \enddiagram
\end{equation}
Furthermore, Cauchy-complete skeletal generalised metric spaces are precisely
those injective with respect to the $\mathsf{Q}$-embeddings. In terms of
sequences, the extension $d$ is
given by $d(b)=\lim_nh(x_n)$, where $(x_n)$ is a Cauchy sequence in $A$ such
that $B(f-,b)=\lim_nA(-,x_n)$.

\begin{cor}
  \label{cor:5}
  Let $f\colon A\to B$ be a non-expansive function between skeletal
  generalised metric spaces, and assume that $B$ is a metric space. Then,
  $f$ is a $\mathsf{Q}$-embedding if and only if it is a dense isometry.
\end{cor}
\begin{proof}
  If $f$ is a $\mathsf{Q}$-embedding and $b\in B$, there is a Cauchy sequence
  $(x_n)$ in $A$ such that $\lim_nA(-,x_n)=\lim_nB(f-,b)$. Given
  $\varepsilon>0$, there is a $n_0$ such that $A(x_n,x_m)<\varepsilon/2$ if
  $n,m\geq n_0$. Thus,
  for $m\geq n_0$ we have
  \begin{equation}
    \label{eq:113}
    B(f(x_m),b)=\lim_nB(f(x_m),f(x_n))=\lim_nA(x_m,x_n)\leq\varepsilon
    /2<\varepsilon.
  \end{equation}
  It follows that $(f(x_m))$ converges to $b$, and $f$ is dense. Observe that we
  have used that the distance of $B$ is symmetric.

  Conversely, if $f$ is a dense isometry, any $b\in B$ is $\lim_nf(x_n)$ for
  some sequence $(x_n)$ in $A$, which is Cauchy since $f$ preserves distances
  and $(f(x_n))$ converges. Then $B(f(a),b)=\lim_n A(a,x_n)$ for all $a\in A$,
  and Proposition~\ref{prop:11} applies.
\end{proof}

The definition of $QA$ given in terms of Cauchy sequences immediately tells us
that if $A$ is a metric space then $QA$ is a metric space too; ie, its distance
function is symmetric.
We deduce:
\begin{cor}
  The {\normalfont\textsl{\textsc{lofs}}} $(\mathsf{L}_Q,\mathsf{R}_Q)$
  restricts to an {\normalfont\textsl{\textsc{ofs}}} on the category of metric
  spaces. Its left maps are the dense isometries.
\end{cor}

\appendix
\section{Accessible AWFSs}
\label{sec:accessible-awfss}

In \S \ref{sec:reflective-lofss} we characterised those \textsc{lofs}s
``fibrantly generated'' by a lax idempotent monad. In this section we explore
what more can be said in the case when the base \Ord-category is locally
presentable and all the monads and comonads involved are accessible. We confine
our discussion to this appendix, as we will assume familiarity with the basic
theory of accessible and locally presentable categories,
for which the standard references are~\cite{MR1031717} and~\cite{MR1294136}. 

We start with a result about ordinary (instead of enriched) accessible
\textsc{awfs}s. These are \textsc{awfs} whose comonad and monad are accessible
functors; in fact, it suffices that only one of them should be
accessible. See~\cite{MR3393453} for details.
\begin{prop}
  \label{prop:15}
  Let $F$ be a left adjoint functor between a locally presentable category $\C$
  and an accessible category $\mathcal{A}$, and $(\mathsf{G},\mathsf{S})$ be an accessible
  {\normalfont\textsl{\textsc{awfs}}} on $\mathcal{A}$. Given the following pullback
  of double categories
  \begin{equation}
    \label{eq:152}
    \diagram
    {\mathbb{L}}\ar[r]\ar[d]
    &
    {\mathsf{G}}\text-\mathbb{C}\mathrm{oalg}
    \ar[d]^{}
    \\
    {\Sq(\C)}
    \ar[r]^-{\Sq(F)}
    &
    {\Sq(\mathcal{A})}
    \enddiagram
  \end{equation}
  there exists an accessible {\normalfont\textsl{\textsc{}}}
  $(\mathsf{L},\mathsf{R})$ on \C\ such that
  $\mathsf{L}\text-\mathbb{C}\mathrm{oalg}\cong\mathbb{L}$ over $\Sq(\C)$ and
  the vertical category of $\mathbb{L}$ is locally presentable.
\end{prop}
\begin{proof}
  If suffices to prove that the functor $U\colon \mathcal L\to\C^\two$ is
  comonadic (see~\cite[Prop.~4]{MR3393453}). By the dual version of Lemma~\ref{l:4}, it
  suffices to show that it has a left adjoint. Being the pullback of a functor that
  creates colimits (indeed, comonadic) along a cocontinuous functor, $U$ creates
  colimits too, so $\mathcal L$ is cocomplete and $U$ cocontinuous. On the other
  hand, $\mathcal{L}$ is accessible, being the limit of a diagram of accessible
  categories and accessible functors (see~\cite[Thm.~5.1.6]{MR1031717}). It follows that $\mathcal{L}$
  is locally presentable, and therefore the cocontinuous functor $U$ is a left
  adjoint.
\end{proof}

\begin{df}
\label{df:9}
\Ord-enriched categories or functors will be called accessible or
locally presentable if their underlying (ordinary) categories or functors are
so. An \textsc{awfs} $(\mathsf{L},\mathsf{R})$ on an accessible \Ord-category is
accessible if one of the following equivalent conditions holds: the
endofunctor $L$ is accessible; the endofunctor $R$ is accessible; the category
of $\mathsf{L}$-coalgebras is accessible; the category of $\mathsf{R}$-algebras
is accessible.
\end{df}

In what follows we maintain the terminology and notations of \S
\ref{sec:reflective-lofss}.
Split opfibrations in an \Ord-category with comma-objects \C\ are the algebras
for the monad $\mathsf{M}$ on $\C^\two$ given by $M(f)=(f\downarrow 1)$ (see
Notation~\ref{not:E,M}).
\begin{lemma}
  \label{l:19}
  Split opfibrations in \Ord-categories are full morphisms.
\end{lemma}
\begin{proof}
  Recall from \S \ref{sec:ord-enrich-categ} that a morphism $p\colon X\to Y$ in
  an \Ord-category $\mathcal{A}$ is full if the monotone morphism
  $\mathcal{A}(Z,p)\colon\mathcal{A}(Z,X)\to\mathcal{A}(Z,Y)$ between posets is
  full in the usual sense. If $p$ is a split opfibration, then
  $\mathcal{A}(Z,p)$ is a split opfibration of posets. Then, it suffices to
  prove that split opfibrations of posets are full. This is an easy
  verification: if $p:X\to Y$ is a split opfibration and $p(x)\leq p(y)$,
  then there is an opcartesian lifting $x\leq \tilde y$ with $p(\tilde y)=p(y)$,
  and $\tilde y\leq y$. Thus $x\leq y$.
\end{proof}

In this section we will make explicit the distinction between \Ord-enriched
categories, functors and monads and their ordinary counterparts by adding to the
latter the subscript $(-)_\circ$; this is the same notation employed
in~\cite{Kelly:BCECT,Kelly:BCECTrep} and elsewhere.

There is a theory of locally finitely presentable enriched categories, developed
in detail in~\cite{Kelly:StructuresFiniteLimits}. Furthermore, much of this
theory carries over to locally presentable categories enriched in a locally
finitely presentable symmetric monoidal closed category (in our case, \Ord).
There will be very few facts about locally presentable \Ord-categories that we
shall need, so we point the reader to~\cite[7.4]{Kelly:StructuresFiniteLimits}
for some guidance about the overall theory.

\begin{df}
\label{df:17}
Let $\kappa$ be a regular cardinal. An object $X$ of a cocomplete \Ord-category
is $\kappa$-presentable if $\C(X,-)\colon\C_\circ\to\Ord$ preserves
$\kappa$-filtered colimits. We say that \C\ is a \emph{locally
  $\kappa$-presentable} \Ord-category if it is cocomplete (in the \Ord-enriched
sense) and has a small full sub-\Ord-category $\mathcal{G}\subseteq\C$
consisting of $\kappa$-presentable objects and such that the associated
``nerve'' functor $\C\to[\mathcal{G}^\mathrm{op},\Ord]$ reflects isomorphisms.
A locally presentable \Ord-category is one that is $\kappa$-presentable for some
$\kappa$.
\end{df}
The first thing we need to mention is that if $\C$ is a locally presentable
\Ord-category, then  it is automatically complete and its underlying category $\C_\circ$ is locally presentable in
the usual sense (with the same accessibility exponent). An \Ord-functor between locally
presentable \Ord-categories is said to be \emph{accessible} when its underlying
functor is accessible in the usual sense; this is because preservation of
conical colimits is just preservation of those colimits by the underlying
functor. An \Ord-monad is accessible if its underlying functor is so. If
$\mathsf{T}$ is an accessible \Ord-monad on the locally presentable
\Ord-category \C, then $\mathsf{T}\text-\mathrm{Alg}$ is locally
presentable.

\begin{rmk}
  \label{rmk:10}
  In locally $\kappa$-presentable category \C, finite limits commute with
  $\kappa$-filtered colimits. In fact all that is necessary is the existence of
  a family of $\kappa$-presentable objects $\{G_i\}$ such that the functors
  $\C(G_i,-)\colon \C_0\to\Ord$ are jointly conservative (ie, a morphism $f$ is
  an isomorphism if each $\C(G_i,f)$ is an isomorphism).
\end{rmk}

\begin{df}
  \label{df:18}
  An \Ord-enriched \textsc{awfs} $(\mathsf{L},\mathsf{R})$ on a locally
  presentable \Ord-category $\C$ is \emph{accessible} if its underlying ordinary
  \textsc{awfs} on the accessible ordinary category $\C_\circ$ is accessible.
\end{df}

\begin{thm}
  \label{thm:14}
  Let \C\ be a locally presentable \Ord-category.
  Then, accessible lax idempotent monads on \C\ are
  fibrantly \slkz-generating. The {\normalfont\textsl{\textsc{lofs}}}
  $\Psi(\mathsf{T})$ generated by an accessible lax idempotent monad $\mathsf{T}$
  is accessible.
\end{thm}
\begin{proof}
  We have to show that there is an \Ord-enriched \textsc{awfs}
  $(\mathsf{L},\mathsf{R})$ for which
  $\mathsf{L}\text-\mathrm{Coalg}\cong\mathsf{T}\text-\mathrm{Emb}$. We first
  show $\lari(\mathsf{T}\text-\mathrm{Alg})_\circ$ is an accessible
  category. Even though we know that the category
  $\mathsf{T}\text-\mathrm{Alg}_\circ$ is accessible by
  \cite[Thm.~5.1.6]{MR1031717}, it is not enough for our purposes, as our proof
  involves $\Ord$-enriched (co)limits, and we have to argue as follows.

  The existence of limits in the \Ord-category \C\ ensures the same for
  $\mathsf{T}\text-\mathrm{Alg}$. By hypothesis, \C\ is locally
  $\kappa$-presentable and $\mathsf{T}$ preserves $\kappa'$-filtered colimits,
  but we may assume $\kappa=\kappa'$ by raising the accessibility exponent
  (see~\cite{MR1031717}). Then
  $\mathsf{T}\text-\mathrm{Alg}$ has $\kappa$-filtered colimits and the family
  $\{T(G):G\in\mathcal{G}\}$ satisfies the conditions of Remark~\ref{rmk:10}, so
  finite limits commute with $\kappa$-filtered colimits in
  $\mathsf{T}\text-\mathrm{Alg}$ (the latter can be shown to be cocomplete but
  we do not need it here). The comonad $\mathsf{E}$ on
  $\mathsf{T}\text-\mathrm{Alg}^\two$ whose coalgebras are \textsc{lari}s
  (Lemma~\ref{l:5}) was described in \S \ref{sec:laris-awfss} by means of finite
  limits (specifically, comma-objects) and therefore preserves $\kappa$-filtered
  colimits. In particular,
  $\lari(\mathsf{T}\text-\mathrm{Alg})_\circ$ is accessible and comonadic over
  $\mathsf{T}\text-\mathrm{Alg}^\two_\circ$.

  We next show that that there is an accessible ordinary \textsc{awfs}
  $(\mathsf{L},\mathsf{R})$ with an isomorphism of categories
  $\mathsf{L}\text-\mathrm{Coalg}\cong\mathsf{T}\text-\mathrm{Emb}_\circ$ over
  $\C^\two_\circ$ by applying Proposition~\ref{prop:15}, whose hypotheses we now
  verify. We have an accessible \textsc{awfs} $(\mathsf{E},\mathsf{M})$ on
  $\mathsf{T}\text-\mathrm{Alg}_\circ$, by the previous paragraph.
  By definition, $\mathsf{T}\text-\mathrm{Emb}$ is the pullback of
  $\lari(\mathsf{T}\text-\mathrm{Alg})_\circ\to\mathsf{T}\text-\mathrm{Alg}^\two_\circ$
  along
  $(F^{\mathsf{T}})^{\two}_\circ\colon \C^\two_\circ\to
  \mathsf{T}\text-\mathrm{Alg}^\two_\circ$. An application of
  Proposition~\ref{prop:15} produces the required accessible \textsc{awfs} on
  $\C_\circ$.


  All that remains is to show that it is an \Ord-enriched
  \textsc{awfs}, or equivalently, that the comonad $\mathsf{L}$ (whose category of
  coalgebras is $\mathsf{T}\text-\mathrm{Emb}_\circ$) is
  \Ord-enriched. Or, equivalently still, that
  $U\colon \mathsf{T}\text-\mathrm{Emb}\to\C^\two$ has an \Ord-enriched right
  adjoint. We have shown above that the ordinary functor $U_\circ$ has a right
  adjoint, say $W$. All we have to show is that
  the monotone map
  \begin{equation}
    \label{eq:125}
    \mathsf{T}\text-\mathrm{Emb}(f,Wg)\xrightarrow{U}
    \C^\two(Uf,UWg)\xrightarrow{\C^\two(1,(1,Rg))}\C^\two(Uf,g)
  \end{equation}
  is not only an isomorphism of sets but also an isomorphism of posets. This
  amounts to showing that it is a full morphism of posets.
  Before doing so, we need the following observation.

  The functor
  $\mathsf{E}_\circ\text-\mathrm{Coalg}\to \mathsf{L}\text-\mathrm{Coalg}$ that
  expresses the fact that each \textsc{lari} is canonically a
  $\mathsf{T}$-embedding, induces a morphism of \textsc{awfs}
  $(\mathsf{E}_\circ,\mathsf{M}_\circ) \to(\mathsf{L},\mathsf{R})$, and thus a
  morphism of monads $\mathsf{M}_\circ\to\mathsf{R}$; in this argument we have
  used \cite[Prop.~2]{MR3393453} twice. It follows that each
  $\mathsf{R}$-algebra is an $\mathsf{M}$-algebra, ie a split opfibration.

  Returning to~\eqref{eq:125}, the first arrow is full because an inequality
  between morphims of $\mathsf{T}$-embeddings is, by Definition~\ref{df:13}, an
  inequality between them as morphisms in $\C^\two$. The second morphism
  in~\eqref{eq:125} is also full, because $Rg$ is a split opfibration (see the
  previous paragraph) and Lemma~\ref{l:19}. Therefore, $W$ extends to an
  \Ord-enriched adjoint to $U$, completing the proof.
\end{proof}

\begin{thm}
  \label{thm:6}
  If \C\ is a locally presentable \Ord-category, the fully faithful
  \Ord-functor
  \begin{equation}
    \label{eq:123}
    \Psi\colon \mathbf{LIMnd}_{\mathrm{acc}}(\C)
    \longrightarrow
    \mathbf{LOFS}_{\mathrm{acc}}(\C)
  \end{equation}
  exhibits the \Ord-category of accessible lax idempotent monads as a reflective
  full sub-\Ord-category of the category of accessible
  {\normalfont\textsl{\textsc{lofs}}s}.
  Its replete image consists of all cancellative
  sub-{\normalfont\textsl{\textsc{lari lofs}s}} that are accessible.
\end{thm}
\begin{proof}
  The \Ord-functor $\Psi$ from $\mathbf{LIMnd}_{\mathrm{fib}}(\C)$ to
  $\mathbf{LOFS}$ restricts to the subcategories of accessible lax idempotent
  monads and accessible \textsc{lofs}s, by Theorem~\ref{thm:14} yielding an \Ord-functor as in the
  statement. We know from Proposition~\ref{prop:5} that $\Psi(\mathsf{T})$ is
  always sub-\textsc{lari}.

  Clearly, the monad $\Phi(\mathsf{L},\mathsf{R})=\mathsf{R}_1$ is accessible if
  $(\mathsf{L},\mathsf{R})$ is an accessible \textsc{awfs}, so we obtain a left
  adjoint $\Phi$ to the fully faithful \Ord-functor $\Psi$ of the statement.
  Its unit
  \begin{equation}
    \varpi\colon
    (\mathsf{L},\mathsf{R})\longrightarrow\Psi\Phi(\mathsf{L},\mathsf{R})=\Psi(\mathsf{R}_1)
  \end{equation}
  is the morphism of \textsc{awfs}s that corresponds to the \Ord-functor that is
  the inclusion of $\mathsf{L}\text-\mathrm{Coalg}$ into
  $\mathsf{R}_1\text-\mathrm{Emb}$, and the former is invertible if and only if
  the latter is so. We may now apply Theorem~\ref{thm:7} to deduce that
  $(\mathsf{L},\mathsf{R})$ is cancellative precisely when the unit $\varpi$ is
  invertible, which is another way of saying that $(\mathsf{L},\mathsf{R})$ is
  in the replete image of $\Phi$.
\end{proof}


\begin{ex}
  \label{ex:15}
  There are accessible monads that are not simple, as exhibited below. This
  means that, even though the monad induces an \textsc{lofs}, it cannot be
  obtained through the methods of \S \ref{sec:simple-monads} and \S
  \ref{sec:simple-monads-1}. One example that involves only ordinary categories,
  which we may regard as locally discrete \Ord-categories, is
  \cite[Example~4.2]{MR779198}, where the monad $D$ on the category of abelian
  groups $\mathbf{Ab}$ is given by $A\mapsto A/2A$ (quotient by $2A=\{2a:a\in
  A\}$). If $f\colon 0\to D(\mathbb{Z})=\mathbb{Z}/2\mathbb{Z}$ is the unique
  possible morphism, then the comma-object $Kf$ is the pullback of $f$ along the
  quotient map $\mathbb{Z}\to\mathbb{Z}/2\mathbb{Z}$. In other words, this
  pullback is the inclusion $2\mathbb{Z}\hookrightarrow{}\mathbb{Z}$. The
  morphism $Lf\colon 0\to 2\mathbb{Z}$ is the unique possible, and $D(Lf)$ is
  not an isomorphism (equivalently, a \textsc{lari}) since $D(2\mathbb{Z})\ncong 0$.

  This example can be modified to show that, for example, the monads on the
  \Ord-categories of (commutative) monoids in \Ord\ that sends a monoid
  $(V,e,\otimes)$ to the coequalizer of the pair of morphisms $V\to V$ that are
  $x\mapsto (x\otimes x)$ and $x\mapsto e$, is not simple. Nonetheless, this
  monad gives rise to a \textsc{lofs}, by Theorem~\ref{thm:6}.
\end{ex}
\bibliographystyle{abbrv}
\bibliography{loawfs.bib}
\end{document}